\documentclass[a4paper,11pt,reqno]{amsart}
\usepackage{vmargin}
\usepackage[pdftex]{graphicx}
\usepackage{amsfonts}
\usepackage{amssymb}
\usepackage[sorted]{amsrefs} 
\usepackage{mathrsfs}
\usepackage{stmaryrd}
\usepackage{amsmath}
\usepackage{amsthm}
\usepackage[shortlabels]{enumitem}
\usepackage{nomencl}
\usepackage[bookmarks=true]{hyperref}   % Hyperref in DVI and PDF
\usepackage{esint}
\usepackage{dsfont}
\usepackage{upgreek}
\usepackage{xcolor}
\usepackage{slashed}
\usepackage{scalerel}
\usepackage{comment}
\usepackage{todonotes}
\usepackage[utf8]{inputenc}
\usepackage{faktor}
\usepackage{mathtools}
\usepackage{sansmath}
\usepackage{nicefrac}
\usepackage{upgreek}
\usepackage{bbm}

\usepackage{color} 

\definecolor{pink}{rgb}{1,0,1}

\definecolor{purple}{rgb}{0.4,0.2,1}

\hypersetup{
	colorlinks,%
}

\author{Lashi Bandara}
\author{Anisa Hassan}
\title{The space of rough Riemannian metrics}
\date{\today}

\address{Lashi Bandara,
	Deakin Mathematics Group, 
	Deakin University Melbourne Burwood Campus, 
	221 Burwood Highway, Burwood, Victoria, Australia, 3125
}
\urladdr{\href{https://lashi.org}{https://lashi.org}}
\email{\href{mailto:lashi.bandara@deakin.edu.au}{lashi.bandara@deakin.edu.au}}

\address{Anisa Hassan, 
	Mathematics Department, 
	Brunel University of London, 
	Kingston Lane, Uxbridge, Middlesex,  UB8 3PH
}
\email{\href{mailto:anisa.hassan@brunel.ac.uk}{anisa.hassan@brunel.ac.uk}}


\def\colour{\colour}

\newcommand{\cbrac}[1]{\left(#1\right)}

\newcommand{\dbrac}[1]{\left\{#1\right\}}

\newcommand{\Av}{\mathrm{Av}} 		% Average

\newcommand{\modulus}[1]{|#1|}

\newcommand{\norm}[1]{\| #1 \|}			% Norm
\newcommand{\set}[1]{\dbrac{#1}}

\newcommand{\close}[1]{\overline{#1}} 
\newcommand{\union}{\cup} 
\newcommand{\interior}[1]{\mathring{#1}}	% Interior

\newcommand{\tensor}{\otimes}

\newcommand{\Ell}{\mathrm{Ell}_{\mathrm{loc}}}			% 

\newcommand{\cE}{\mathcal{E}}

\newcommand{\cM}{\mathcal{M}}
\newcommand{\cN}{\mathcal{N}}
\newcommand{\cU}{\mathcal{U}}
\DeclareMathOperator{\graph}{graph}

\DeclareMathOperator{\ident}{id}
\DeclareMathOperator{\Id}{I}

\DeclareMathOperator{\spt}{spt}
\newcommand{\rest}[1]{{{\lvert_{}}_{}}_{#1}}

\newcommand{\Met}{\mathrm{Met}}

\newcommand{\Lp}[2][{}]{{\rm L}^{#2}_{\rm #1}}		% L_{#1}^{#2}
\newcommand{\Ck}[2][{}]{{\rm C}^{#2}_{\rm #1}}		% C^k_c
\newcommand{\Hard}[2][{}]{{\rm H}^{#2}_{\rm #1}}		% H^k_{c}  Sobolev space
\newcommand{\Sob}[2][{}]{{\rm W}^{#2}_{\rm #1}}		% W^{k,p}_{c} Sobolev space 
\newcommand{\SobH}[2][{}]{\Hard[#1]{\rm #2}}
\newcommand{\Tensors}[1][{}]{{\mathcal{T}}^{(#1)}}	% Tensors
\DeclareSymbolFont{script}{U}{eus}{m}{n}	% Font to produce wedge product for forms. 
\DeclareMathSymbol{\bwedge}{0}{script}{"5E}	% Actual wedge product symbol which scales within the text
\newcommand{\Forms}[1][{}]{\bwedge^{#1} {\kern 0.05em}}		% P-Forms\\
\newcommand{\tanb}{{\rm T}}		% Tangent Bundle
\newcommand{\cotanb}{{\rm T}^\ast}	% Cotangent bundle
\DeclareMathOperator{\End}{End} 
\DeclareMathOperator{\Sym}{Sym} 
\DeclareMathOperator{\e}{e}
\newcommand{\Hom}[1][*]{\mathcal{H}^{#1}}
\newcommand{\Dir}{{\rm D} }

\newcommand{\spec}{\mathrm{spec}}		% Spectrum
\DeclareMathOperator{\diag}{diag} 

\newcommand{\mg}{\mathrm{g}}
\newcommand{\mgt}{\widetilde{\mg}}
\newcommand{\mh}{\mathrm{h}}
\newcommand{\mht}{\widetilde{\mh}}
\newcommand{\Leb}{\mathcal{L}}			%Lebesgue measure
\newcommand{\R}{\mathbb{R}}			% Real numbers
\newcommand{\C}{\mathbb{C}}			% Real numbers
\newcommand{\Q}{\mathbb{Q}}			% Real numbers
\newcommand{\Na}{\mathbb{N}}			% Real numbers

\newcommand{\conj}[1]{\overline{#1}}		% Conjugation
\newcommand{\dom}{\mathrm{dom}}

\newcommand{\ran}{\mathrm{ran}}

\newcommand{\inprod}[1]{\left\langle #1 \right\rangle}	% inner product braces
\DeclareMathOperator{\divv}{div}		% Divergence
 
\DeclareMathOperator{\Lap}{\Delta}

\newcommand{\Ball}{\mathrm{B}}			% Ball

\newcommand{\Ric}{{\rm Ric}}			% Ricci Curvature
\newcommand{\dR}{\mathrm{dR}}

\newcommand{\Sing}{\mathrm{Sing}}
\newcommand{\Reg}{\mathrm{Reg}}
\newcommand{\extd}{{\rm d}}			% Exterior Derivative

\newcommand{\Comp}{\mathrm{Comp}}
\newcommand{\met}{\mathrm{d}}
\newcommand{\rmet}{{\mathbbm{d}\kern -0.16em \mathrm{l}}}
\newcommand{\action}{\eta}
\newcommand{\Riem}{\mathrm{Riem}} 
\newcommand{\ic}{\mathrm{in}} 
\DeclareMathOperator{\len}{len}
\newcommand{\Spa}{\mathrm{Spa}}
\newcommand{\Hil}{\mathcal{H}}

\newcommand{\hker}{\uprho}

\DeclareMathOperator{\rank}{rank}
\newcommand{\op}{\mathrm{op}}

\newcommand{\Hinfty}{\ensuremath{\mathrm{H}^\infty}}

\DeclareMathOperator*{\esssup}{ess\, sup}
\AtBeginDocument{%
	
}
\renewcommand{\Re}{\mathrm{Re}\ }

\renewcommand{\emptyset}{\varnothing}
\renewcommand{\epsilon}{\varepsilon}

\newtheorem{thm}{Theorem}[section]
\newtheorem{lem}[thm]{Lemma}
\newtheorem{cor}[thm]{Corollary}
\newtheorem{prop}[thm]{Proposition}

\theoremstyle{definition}
\newtheorem{defn}[thm]{Definition}
\newtheorem{rem}[thm]{Remark}

\begin{document}

\vspace*{-2em}
\begin{abstract}
On a smooth connected manifold, we consider all possible locally elliptic and locally bounded measurable coefficient Riemannian metrics called rough Riemannian metrics.
We equip this set with an extended metric which is connected if and only if the manifold is compact.
We prove that this is a complete length extended metric space and the components on which the distance is finite are path-connected.
Moreover, we identify  the closure of smooth metrics in this space to be continuous metrics.
\end{abstract} 
\maketitle

\tableofcontents\vspace*{-2em}

\parindent0cm
\setlength{\parskip}{0.8\baselineskip}

% CONTENT STARTS HERE: 

\section{Introduction} 

Non-smooth geometry is a vast, important and contemporary field of mathematics boasting multiple subfields which has enjoyed considerable attention and development of late.
Our focus in this paper is restricted to a special class of non-smooth geometries - those that arise from non-smooth metrics on smooth manifolds.
More precisely, on a smooth, connected manifold $\cM$ without boundary, we consider \emph{rough Riemannian metrics} which are supported on $\cM$.
These are locally bounded and locally elliptic measurable coefficient versions of smooth Riemannian metrics.
However, rather than focusing on results pertaining to particular rough metrics, we consider the set of all possible such metrics $\Met(\cM)$ and endow it with a natural extended metric space structure.

The induced topology can be thought of as capturing the notion of ``$\Lp{\infty}$'' without having to fix a background metric.
We identify important features of this extended metric space, including connected components $\Comp(\mg)$ associated to $\mg \in \Met(\cM)$, on which many interesting geometric and analytic properties, including $\Lp{p}$-spaces, are preserved.
In addition to approximation techniques afforded by the metric structure, preservation of vital quantities on components allow for the possibility of exploiting perturbation and operator-theoretic methods in the study of geometric problems.

Before we continue to describe our framework and results, we must emphasise that rough Riemannian metrics, either implicitly or not explicitly being named as such, have been studied for some time as they arise naturally in geometry, physics and engineering. 
While it is prohibitive to provide an exhaustive list of contributions, we provide some references to studies that are relevant to the results of this paper.

The earliest examples known to us arise in the works of De Giorgi in \cite{DeGiorgi}, Moser in \cite{Mo1,Mo2}, and Nash in \cite{Nash}.
There, they considered second-order operators of the form $-\divv A \nabla$ where $A$ are uniformly elliptic, bounded and measurable coefficients.
They obtained a priori regularity estimates for solutions of elliptic and parabolic equations featuring $-\divv A \nabla$. 
Though standard, these results are hailed today as remarkable and are cornerstone in the modern treatment of PDE.
In Subsection~\ref{S:DivForm}, we show that operators $-\divv A \nabla$  are Laplacians related to rough Riemannian metrics modulo multiplication by  measurable functions uniformly bounded above and below.
Another related example includes Lipschitz boundary value problems as they involve pullbacks to the upper-half space via lipeomorphisms, which induce bounded, measurable coefficients operators. 
These can again be captured in the rough Riemannian metric language as shown in Subsection~\ref{S:Pullback}.

Rough Riemannian metrics feature in a deliberate form in \cite{Teleman} due to Teleman.
There, the author studied signature operators for rough Riemannian metrics,  even allowing for the the underlying manifold to be only Lipschitz, though restricting to the compact setting.
Results obtained by De Cecco-Palmieri \cite{DP90, DP91, DP95} also live in the Lipschitz manifold context, this time allowing for non-compactness. 
These authors prove important properties of an induced distance associated to a rough Riemannian metric.
In \cite{Norris}, Norris proves Varadhan asymptotics, relating the induced distance of a rough Riemannian metric on $\cM$ to small-time asymptotics of the heat kernel related to its Laplacian, again on a Lipschitz manifold.
Sturm in \cite{Sturm} also discusses such metrics and demonstrates that the distance structure does not determine the rough Riemannian metric, referred to there as the ``diffusion coefficients''.

Our choice to consider smooth differentiable manifolds is intentional. 
If we allow the manifold to be merely Lipschitz, then the transition maps on the tangent bundle are, in general, locally bounded and measurable.
Therefore, the tangent bundle cannot support smooth or even continuous sections.
In our study, we are particularly interested in studying the role of higher regularity structures, such as smooth or continuous metrics, within  $\Met(\cM)$.
Since every $\Ck{1}$-manifold is smoothable, our choice of assuming that $\cM$ has a smooth differentiable structure is warranted. 

Since every rough metric $\mg \in \Met(\cM)$ induces a measure-metric structure $(\cM,\met_{\mg},\mu_{\mg})$, it would seem sensible to study the set of such triples to better interface 
with methods and techniques, such as pointed Gromov-Hausdorff convergence, in the measure-metric space setting. 
However, we show in Subsection~\ref{S:InducedMMs} that $(\cM,\met_{\mg},\mu_{\mg})$ does not uniquely determine $\mg$. 
Therefore, $\Met(\cM)$ is strictly richer in structure which justifies our choice to study $\Met(\cM)$.

Rough Riemannian metrics on smooth manifolds have been studied now for some time.
In their current form, they were identified by Bandara in \cite{BRough} as a class of geometric invariants to the famous \emph{Kato square root problem}, as seen from a first-order perspective via the \Hinfty-functional calculus due to  Axelsson(Rosén)-Keith-McIntosh in \cite{AKMC}.
Salient points of this problem and its relationship to our work here is outlined in Subsection~\ref{S:FCKato}.
Rough Riemannian metrics were also featured in the study of the so-called Gigli-Mantegazza flow, initially formulated by Gigli-Mantegazza in \cite{GM}. 
Regularity properties of this flow for rough Riemannian initial data was studied in \cite{BLM} by Bandara-Lakzian-Munn and by Bandara in \cite{BCont}.
Weyl asymptotics associated to the Laplacians induced by rough Riemannian metrics on smooth compact manifolds were obtained by Bandara-Nursultanov-Rowlett in \cite{BNR}.
More recently, the term \emph{geometric singularities} with respect to rough Riemannian metrics was made precise in \cite{BH} by Bandara-Habib, where versions of the Hodge theorem are obtained. 
This justification rests ultimately on their results which assert the kernel of the Hodge-Dirac operator induced by a rough Riemannian metric is isomorphic to the smooth singular cohomology when the underlying manifold is compact.
Although similar results for the de Rham cohomology are obtained in \cite{Teleman} for Lipschitz manifolds, the study in \cite{BH} exploits the smooth differentiable structure to access the smooth singular cohomology.

The geometric perspective developed in this paper helps provide an alternative vantage point to pre-existing results and simultaneously provide new tools in the analysis of problems in geometry and PDE.
As aforementioned, we are mainly concerned with the entire collection of rough Riemannian metrics.
On connected components $\Comp(\mg)$, many interesting properties are preserved. 
For instance, the $\Lp{p}$-spaces are fixed (as sets) on $\Comp(\mg)$ with equivalent norms for $p \in [1,\infty]$ as given in Subsection~\ref{S:Preserve}.
Poincaré inequalities are also preserved across $\Comp(\mg)$ in a quantifiable manner, given in Subsection~\ref{S:Poincare}.
The presence of a single metric $\mh \in \Comp(\mg)$ which is smooth, complete and with Ricci curvature bounded below ensures that the heat kernel of any Laplacian induced by a metric in $\Comp(\mg)$ enjoys Hölder regularity.
This result, due to Saloff-Coste in \cite{SC}, rephrased in the language of this paper, is showcased in Subsection~\ref{S:Ricci}.

Our study in this paper was borne out of considerations in \cite{BRough} to develop a ``meta'' understanding of the Kato square root problem. 
There, it is shown through \emph{large} perturbations permitted by the stability of the functional calculus that solutions to the Kato square root problem can be obtained on geometries that violate the geometric hypotheses initially used to prove a smooth Kato square root problem. 
While \cite{McIntosh72} provides a counterexample to the Kato square root problem as an abstract construction, in general, there are no known geometric counterexamples. 
Perhaps more seriously, there are no known geometric counterexamples to the \Hinfty-functional calculus. 
One important feature of the first-order characterisation of the Kato square root problem in \cite{AKMC} is ``functional calculus implies Kato''.
Therefore, finding a geometric counterexample to the Kato square root problem yields a geometric counterexample to functional calculus for free.
The study of $\Met(\cM)$ initiated in this paper is expected to provide scaffolding to better understand the structure of these problems and ultimately, produce geometric counterexamples. 

This paper is organised as follows.
In the following section, Section~\ref{S:Results}, we provide a brief of the notation, the main constructions, definitions and theorems.
This is designed so that the reader can easily access and utilise the significant results of this paper without having to labour through proofs and technicalities.
Section~\ref{S:Examples} is dedicated to showcasing examples, applications and consequences of our results.
In Section~\ref{S:RRMs}, we present some important results regarding the induced measure and metric structures of rough Riemannian metrics, before discussing the entire space of rough Riemannian metrics. 
Importantly, we identify a group within $\Lp[loc]{\infty}$, playing a role similar to the general-linear group, which acts on $\Met(\cM)$.
Section~\ref{S: SmoothRRM} is dedicated to the study of smooth Riemannian metrics and we identify their closure in the endowed topology.
Lastly, completeness is studied in Section~\ref{S:Completeness} and connectedness, in particular path-connectedness, in Section~\ref{S:Connectedness}.

\section*{Acknowledgements} 
Some part of this research was conducted during a research visit by L.B. to Brunel University of London.
L.B. acknowledges the gracious hospitality of Anne-Sophie Kaloghiros and Brunel University of London. 
A.H. was supported by UKRI EPSRC Studentship EP/W524542/1 2790227.
Both authors wish to express their gratitude to Sergey Mikhailov for his interest and encouragement of this paper. 

\section{Results} 
\label{S:Results} 

\subsection{Notation}
\label{S:Notation}

Throughout, we fix $\cM$ to be a smooth connected manifold by which we mean it is equipped with a smooth differentiable structure and it has no boundary.
A vector bundle $\cE \to \cM$ is also assumed to be smooth and its dual bundle is denoted by $\cE^\ast$.
The endomorphism bundle is denoted by $\End \cE$ and the symmetric endomorphism bundle by $\Sym\End \cE$.
Tensors of rank $(r,s)$ are denoted by $\Tensors[r,s] \cE := \otimes_{j=1}^r  \cE^\ast \otimes_{k=1}^s \cE$.
By definition, $\Tensors[0,0] \cE$ are functions.
The set of $k$-differentiable sections of $\cE$ are denoted by $\Ck{k}(\cM;\cE)$. 
Those that are compactly supported are denoted by $\Ck[cc]{k}(\cM;\cE)$.

When there is a metric $\mh$ on $\cE \to \cM$, we assume the induced metric on $\End(\cE)$ is the operator norm, $\norm{\cdot}_{\op,\mh(x)}$ on fibres of $\cE$.
Although $\End(\cE)$ can be identified with $\Tensors[1,1]\cE$, the induced metric on $\cE$ is given by $\mh^\ast \tensor \mh$.
Pointwise, the former norm is the operator norm and the latter the Frobenius norm. 
Note they are equivalent with the constant depending on $\rank \cE$.  

The differentiable structure of a manifold $\cM$ affords us with a (metric independent) measure structure (see Section~\ref{S:RRMs}) as well as a notion of zero measure.  
This allows us to talk about the measurable sections of a vector bundle $\cE \to \cM$, denoted by $\Lp{0}(\cM;\cE)$.
Furthermore, this leads to metric independent local Lebesgue spaces $\Lp[loc]{p}(\cM;\cE)$, the set of $u \in \Lp{0}(\cM;\cE)$ such that $\displaystyle{\int_{K} \modulus{u}_{\mh^K}^p\ d\mu < \infty}$ for $p \in [1, \infty)$ where $\mh^K$ is any smooth metric for $\cE$ on every compact $K \subset \cM$ and $\mu$ is a measure induced by a smooth Riemannian metric.
Similarly, the space $\Lp[loc]{\infty}(\cM;\cE)$ is the set of $u\in \Lp{0}(\cM;\cE)$ such that $|u|_{\mh^K} < \infty$ almost-everywhere for $\mh^K$ a smooth metric for $\cE$ on compact $K \subset \cM$.

\subsection{Rough Riemannian metrics}
We begin with recalling the notion of a rough Riemannian metric from \cite{BRough}.

\begin{defn}[Rough Riemannian metric (RRM)]
\label{Def:RRM}
Suppose that $\mg \in \Lp{0}(\Sym\End \Tensors[2,0] \cM)$  satisfies the \emph{local comparability condition}: for every $x \in M$, there exists a chart $(U, \psi)$ near $x$ and constant $C_U \geq 1$ (dependent on $(U, \psi)$), such that
\begin{equation}\label{cc1}
C_U^{-1}|u|_{\psi^{*}\delta(y)} \leq |u|_{\mg(y)} \leq C_U|u|_{\psi^{*}\delta(y)},
\end{equation}
for all $u\in \tanb_{y}U$ and for almost-every $y\in U$, where $|u|_{\mg(y)}= \sqrt{\mg(y)[u,u]}$. 
Then we say that $\mg$ is a \emph{Rough Riemannian metric} or \emph{RRM} for short. 
\end{defn}

\begin{rem}
Equivalently, an RRM is a $\mg \in \Lp{0}(\Sym\End \Tensors[2,0]\cM)$ such that $\mg \in \Lp[loc]{\infty}(\Sym\End \Tensors[2,0]\cM)$ and $\mg^{-1} \in \Lp[loc]{\infty}(\Sym\End\Tensors[2,0]\cM)$.
Here, $\mg^{-1}$ is identified with the co-metric, locally given by $\mg^{-1} = (\mg_{ij})^{-1}$. 
See Lemma~\ref{Lem:RoughMetEquiv}.
\end{rem}

It is easy to see that every smooth, or even continuous coefficient Riemannian metric is a RRM.
As in the case of a smooth metric, a RRM induces a Radon measure which takes the local expression  
\[ 
d\mu_{\mg}(x)= \sqrt{\det \mg(x)}\ d\psi^{*}\mathcal{L}(x)
\] inside a chart $(U,\psi)$ with $\Leb$ the Lebesgue measure (see Proposition~\ref{Prop:Radon}). 
Although it is not possible to naïvely minimise along curves as in the smooth case to obtain a distance, the expression
\begin{equation}
\label{eq:d_g}
\met_{\mg}(x,y)=\sup \{w(y)-w(x): w \in \Ck{0,1}(\cM),\quad |\nabla w|_{\mg^{-1}}\leq 1\ \text{a.e.}\},   
\end{equation}
where $\modulus{\nabla w}_{\mg^{-1}} = \mg^{ij}\ \partial_i w \partial_j w$ inside a chart with $(\mg^{ij}) = (\mg_{ij})^{-1}$,
yields a well-defined distance on $\cM$. 
This metric is, in fact, a length metric (see Proposition~\ref{Prop:Intrinsic}).
It is is equivalent to the minimisation along curves if $\mg$ is smooth (see Proposition~\ref{Prop:Distance}). 
Therefore, every rough Riemannian metric $\mg$ induces a Radon measure-metric length space $(\cM,\met_{\mg},\mu_{\mg})$.

\subsection{The space of Rough Riemannian Metrics}

The central theme of this paper is to consider and study the set of all rough Riemannian metrics on $\cM$.
We introduce the following notation to denote this set.
\begin{equation*}
\Met(\mathcal{\cM}):=\{\mg \text{ is a rough Riemannian metric on } \cM\}.
\end{equation*}

We define an extended distance metric on this set.
For that, recall from  \cite{BRough} two metrics $\mg,\mh \in \Met(\cM)$ are said to be close if there exists $C \geq 1$ such that
\begin{equation}\label{ccc}
C^{-1}|u|_{\mh(x)} \leq |u|_{\mg(x)} \leq C|u|_{\mh(x)}
\end{equation}
for all $u \in \tanb_x\cM$ and for almost-every $x \in \cM$. 
For brevity, we write  $\mg \sim_{C} \mh$ to also capture the constant $C \geq 1$ which quantifies the closeness.
With this notation in hand, we formulate the central object of study in this paper:

\begin{defn}\label{DefRoughdistance}
Define $\rmet^\cM: \Met(\cM) \times \Met(\cM) \to [0,\infty]$ by 
\begin{equation}
\rmet^\cM(\mg,\mh)
:=
\begin{cases}
\ \inf \{\log (C): \mg \sim_{C} \mh\} & \exists \tilde{C} \geq 1 \text{ s.t. } \mg \sim_{\tilde{C}}  \mh,\\
\ +\infty \qquad \qquad\qquad \qquad&\text{otherwise}.
\end{cases}
\end{equation}
for $\mg, \mh \in \Met(\cM)$.
\end{defn}

As a first, we obtain the following theorem. 
\begin{thm}
\label{Thm:Complete}
Let $\cM$ be a smooth connected manifold. 
Then, the space $(\Met(\cM), \rmet^\cM)$ is a complete extended length metric space.
\end{thm}
The fact that this is a complete metric space is proved in Proposition~\ref{Prop:Complete}.
The length space property is established in Corollary~\ref{Cor:Length}. 

Since smooth and continuous Riemannian metrics are also RRMs, we consider their closure with respect to $\rmet^\cM$, leading to the following theorem. 
\begin{thm}\label{Thm.limit.smooth.RRMs}
Let $\cM$ be a smooth connected manifold. Then $\overline{\Met_{\Ck{\infty}}(\cM)}^{\rmet^\cM} = \Met_{\Ck{0}}(\cM)$.
In particular, $\Met_{\Ck{0}}(\cM)$ is a closed and, hence,  a complete subset. 
\end{thm}
These are obtained in Section~\ref{S: SmoothRRM}.
The fact that $\Met_{\Ck{0}}(\cM)$ is complete follows directly from its closedness and the completeness of  $(\Met(\cM),\rmet^\cM)$.

Since $\rmet^\cM$ is an extended metric, it is useful to consider  subsets of metrics which sit at a finite distance to each other. 
We define this to be 
\[ 
\Comp(\mg) := \set{\mh \in \Met(\cM): \rmet^\cM(\mg,\mh) < \infty}
\]
for each $\mg \in \Met(\cM)$. 

Let $[\mg]$ be the equivalence classes of $\Met(\cM)$ with respect to the relation $\mg \sim \mh$, by which we mean there is some $C \geq 1$ such that $\mg \sim_{C} \mh$. 
It is easy to see that $[\mg] = \Comp(\mg)$. 
As the notation suggests, these equivalence classes have topological relevance as asserted through the following theorem on path-connectedness.
\begin{thm}\label{Thm:CompConnected}
Let $\cM$ be a smooth connected manifold and $\mg \in \Met(\cM)$. Then the component $\Comp(\mg)$ is path-connected.
\end{thm}

As we show in Subsection~\ref{S:Connectedness}, the space $(\Met(\cM), \rmet^\cM)$ is generally disconnected.
However, as the following theorem shows, compactness of $\cM$ is intimately related to the connectedness of $\Met(\cM)$. 

\begin{thm}\label{Thm_com_iff_conn}
A smooth connected manifold $\cM$ is a compact if and only if the space $(\Met(\cM), \rmet^\cM)$ is connected. 
\end{thm}

The proofs of both these theorems can be found in Section~\ref{S:Connectedness}.

Lastly, we outline an important representational tool for further study of these metrics. 
Define
\begin{equation} 
\label{Eq:Ell} 
\Ell(\cM) := \set{B \in \Lp[loc]{\infty}(\cM;\End(\tanb\cM)): B^{-1} \in \Lp[loc]{\infty}(\cM;\End(\tanb\cM))}.
\end{equation}
This set is clearly not a vector subspace. 
It is,  however a group under composition as shown in Proposition~\ref{Prop:EllIsAGroup}.
This set, in a sense, mirrors the general-linear group $\mathrm{GL}(V)$ over a vector space $V$, although here, it is in infinite dimensions. 
For a given $\mg \in \Met(\cM)$, define 
\[
\mg_{B}(x) [u,v] := \mg(x) [B(x)u,B(x)v].
\]
With this, we have the following theorem.

\begin{thm}
\label{Thm:Stab}
Let $\cM$ be a smooth connected manifold.
Then:
\begin{enumerate}[label=(\roman*)]
\item 
\label{Thm:Stab:i}
for $B \in \Ell(\cM)$ and $\mg \in \Met(\cM)$, we have that $\mg_B \in \Met(\cM)$; 
\item 
\label{Thm:Stab:ii} 
the $\action: \Ell(\cM) \times \Met(\cM) \to \Met(\cM)$ given by $(B,\mg) \mapsto \action(B)\mg = \mg_B$ is a group action of $\Ell$ on $\Met(\cM)$;  and
\item 
\label{Thm:Stab:iii}
given any $\mg,\mh \in \Met(\cM)$, there exists  $B \in \Ell(\cM)$ such that $\mg = \action(B)\mh$.
\end{enumerate}
\end{thm} 
This shows that not only $\Met(\cM)$ is stable under action by elements of $\Ell(\cM)$, but $\Met(\cM)$ can be represented by elements of $\Ell(\cM)$.
As this representation depends on a given metric $\mg$, the coefficients $B$ can be thought of as a global coordinate system with respect to $\mg$.
The $B \in \Ell(\cM)$ which changes $\mg$ to $\mh$ in \ref{Thm:Stab:iii} can be chosen uniquely as a pointwise almost-everywhere $\mh$-self-adjoint transformation.

\section{Applications, consequences and examples}
\label{S:Examples}

\subsection{Rough Riemannian metrics through rescalings and stability}
\label{S:Stab}
Rough metrics can arise on a smooth connected manifold $\cM$ in a variety of natural ways.
If $\mg \in \Met_{\Ck{\infty}}(\cM)$, and $f \in \Lp{\infty}(\cM)$ and $f^{-1} \in \Lp{\infty}(\cM)$, then clearly, $f \mg \in \Met(\cM)$.
In this case, we also have that $f\mg \in \Comp(\mg)$.

This is a special case of the situation where we can consider $B \in \Lp{\infty}(\cM; \End(\tanb\cM),\mg)$ with $B^{-1} \in \Lp{\infty}(\cM;\End(\tanb\cM),\mg)$.
Then, 
\[
\mg_{B}(x)[u,v] = \mg(x)[B(x)u, B(x)v]
\] 
yields $\mg_{B} \in \Comp(\mg) \subset \Met(\cM)$.
The earlier example of conformal rescaling by $f$ is then a special case of $B = \sqrt{f} \ident$.
See Theorem~\ref{Thm:Stab} for a comprehensive description.

\subsection{Pullbacks by local lipeomorphisms and severity of singularities}
\label{S:Pullback}
Let $\cM$ and $\cN$ be smooth connected manifolds and let $\mh \in \Met_{\Ck{\infty}}(\cN)$. 
Further, suppose that $F: \cM \to \cN$ is a lipeomorphism.
Then,
\[
\mg_F (x)[u,v] := (F^\ast \mh)(x)[u,v] = \mh(F(x))[ (\extd F)(x)u, (\extd F)(x) v]
\]
satisfies $\mg_F \in \Met(\cM)$.

Suppose now that $\cM$ has a smooth metric $\mg$ and $\cM'$ is another smooth connected manifold with smooth metric $\mg'$. 
Let $f \in  \Ck{0,1}(\cM;\cM')$ consider $\graph f = \set{(x, f(x)) \in \cM \times \cM')}$.
In this case, $F: \cM \to \graph f$ given by $x \mapsto (x, f(x))$ is a lipeomorphism and we can set $\cN := \graph f$.
In particular, the map $F:\cM \to \cN$ is a homeomorphism so $\cN$ has a smooth structure.
Applying our earlier construction with $\mh = \mg \oplus \mg'$, we are able to pullback the induced geometry of $\graph f$ as an embedded submanifold in $(\cM \times \cM', \mg \oplus \mg')$ isometrically to $(\cM,\mg_F)$.

In the special case of $\cM = \R^2$ and $\cN = \R$, given any set $Z \subset \R^2$ with $\Leb(Z) = 0$, there exists a $f_Z \in \Ck{0,1}(\R^2)$ such that 
\[ 
Z \subset \Sing(f_Z) = \set{x \in \R^2: f_Z \emph{ not differentiable at } x}.
\]
In particular, we can take $Z = \Q \times \Q$, and the resulting rough Riemannian metric, $\mg_{f_Z}$ will be singular on a dense subset of of $\R^2$.
This illustrates the severity of singularities of rough Riemannian metrics.

\subsection{Rough Riemannian metrics and the geometry of divergence form operators}
\label{S:DivForm} 
Let $\cM = \R^n$  with $A \in \Lp{\infty}(\R^n; \R^n)$ real-symmetric.
Moreover, suppose that $A^{-1}$ exists and $A^{-1} \in \Lp{\infty}(\R^n;\R^n)$.
Consider the following energy
\[ 
\cE_{A}[u,v] = \int_{\R^n} A(x)\nabla u(x)\cdot \conj{\nabla v(x)}\ d\Leb(x), 
\]
with domain $\dom(\cE_{A}) = \SobH{1}(\R^n)$. 

From Lax-Milgram theory (see Chapter~IV~Section~2 in \cite{Kato}), we have that there exists self-adjoint $L_{A}: \dom(L_{A}) \subset \SobH{1}(\R^n) \to \Lp{2}(\R^n)$ with $\sqrt{L_A} = \SobH{1}(\R^n)$ such that $\cE_{A}[u,v] = \inprod{ \sqrt{L_A}u, \sqrt{L_A}v}_{\Lp{2}(\R^n)}$.
By definition of $\cE_{A}$, it is easy to see that $L_{A} = \nabla^{\ast} A \nabla = -\divv A \nabla$, a divergence-form operator with measurable coefficients.

This is closely related to a geometry associated to the coefficients $A$.
Fix $B, B^{-1} \in \Lp{\infty}(\R^n; \R^n)$ real-symmetric and define the following inner-product
\[
\mg(x)[u,v] = B(x)u\cdot v.
\]
We want the Laplacian of $\mg$ to be related to $L_A$ up to multiplication by functions.

To that end, let us compute the Laplacian of $\mg$ with respect to the Euclidean inner product.
First, note $d\mu_{\mg} = \sqrt{\det B}\ d\Leb$. 
Then, for $u \in \dom(\Lap_{\mg})$ and $v \in \Ck[c]{\infty}(\R^n)$.
\begin{align*} 
\inprod{\Lap_{\mg}u,v}_{\Lp{2}(\mg)} 
&= 
\inprod{\nabla u, \nabla v}_{\Lp{2}(\mg)} \\ 
&=
 \int_{\R^n} \mg(x)[\nabla u(x), \nabla v(x) ]\ d\mu_{\mg}(x) \\
 &=
 \int_{\R^n} \mg(x)[\nabla u(x), \nabla v(x) ]\ d\mu_{\mg}(x) \\
&= 
\int_{\R^n} B(x) \nabla u(x) \cdot \conj{\nabla v(x)}\ \sqrt{\det B(x)}\ d\Leb(x) \\
&=
\int_{\R^n} -\divv ((\det B(x))^{\frac12} B(x) \nabla u)(x)\  \conj{v(x)}\ \ d\Leb(x) \\
&= 
\int_{\R^n} -(\det B(x))^{-\frac12} \divv( (\det B(x))^{\frac12} B(x) \nabla u)(x)\  \conj{v(x)}\ \ d\mu_{\mg}(x).
\end{align*}
That is, $\Lap_{\mg} = -(\det B(x))^{-\frac12} \divv (\det B)^{\frac12} B \nabla$. 

We want $\Lap_{\mg}= -(\det B(x))^{-\frac12} \divv A \nabla$.
As an ansatz, set $B = f A$, for $f, f^{-1} \in \Lp{\infty}(\R^n)$.
Then,  $\det B = f^n \det A$ so that $(\det B)^{\frac12} = f^{\frac n2} (\det A)^{\frac12}$.
We solve for $f$ by requiring 
\[ 
A = (\det B)^{\frac12} B =  f^{\frac n2} (\det A)^{\frac12} f A = f^{\frac{n+2}{2}} (\det A)^{\frac12} A, 
\]
which implies $1 = f^{\frac{n+2}{2}} (\det A)^{\frac12}$.
That is, 
\[ 
f = (\det A)^{-\frac{1}{n+2}}.
\]
With this choice of $f$, we have 
\[ 
\det B = (\det A)^{-\frac{n}{n+2}} \det A = (\det A)^{\frac{2}{n+2}}.
\]
Therefore 
\[
\Lap_{\mg} = -(\det A)^{-\frac{1}{n+2}}  \divv A \nabla, 
\]
which is the operator $\divv A\nabla$ up to multiplication by an $\Lp{\infty}$ function.

This can be repeated on a manifold.
Suppose now that $\cM$ is an $n$-dimensional smooth connected manifold with $\mg \in \Met(\cM)$. 
If  $L_{\mg,A} = -\divv_{\mg} A \nabla$ with $A, A^{-1} \in \Lp{\infty}(\End \cotanb \cM)$ and $A^{\ast,\mg} = A$ almost-everywhere, define 
\[
\mh(x)[u,v] :=  \mg(x)[ (\det A(x)^{-\frac{1}{n+2}} A(x) u, v].
\]
Clearly, $\mh \in \Met(\cM)$.
Furthermore, since $A, A^{-1} \in \Lp{\infty}(\Sym\End \cotanb \cM)$, we have that $\rmet^\cM(\mg,\mh) < \infty$, or equivalently, $\mh \in \Comp(\mg)$.
Mimicking our previous calculation in $\R^n$, we find that
\[
\Lap_{\mh} = -(\det A)^{-\frac{1}{n+2}}  \divv_{\mg} A \nabla = -(\det A)^{-\frac{1}{n+2}}  L_{\mg,A}
\]
Again, we see that  $L_{\mg,A}$ is the Laplacian of $\mh \in \Comp(\mg)$ up to multiplication by uniformly lower and upper bounded measurable functions.
Furthermore, if $\det A = 1$, then we can see that $\Lap_{\mh} = L_{\mg,A}$.

\subsection{A counterexample to smooth approximation}
\label{S:NonApprox}
Theorem~\ref{Thm.limit.smooth.RRMs} shows that the closure of smooth metrics give us continuous metrics.

We provide a concrete example to demonstrate why general RRMs cannot be smoothly approximated.
Let $\cM = \R^n$ and define $f: \R^n \to [1,\infty)$ such that
\[
f(x) 
:= 
\begin{cases}
10^{43}& \text{ when } \modulus{x} \geq 1,\\
1 & \text{ when } \modulus{x}< 1,\\
\end{cases}.
\]
Define  $\mg(p) := f(p) \delta_{\R^n}$ where $\delta_{\R^n}$ is the Euclidean metric.
From the discussion in Subsection~\ref{S:Stab}, it is clear that $\mg \in \Met(\R^n)$ and, in fact, $\mg \in C(\delta_{\R^n})$.

We show that $\mg$ cannot be approximated smoothly. 
Note that $\mg_n = B_n \delta_{\R^n}$ with $B_n, B_n^{-1} \in \Lp{\infty}(\R^n; \R^n)$.
Therefore, if a sequence $\mg_n \in \Met_{\Ck{\infty}}(\R^n)$ such that $\mg_n \to \mg$, then there exists $\epsilon_0 > 0$ such that for $ \epsilon \in (0, \epsilon_0)$, there exists $N(\epsilon) > 0$ such that $n \geq N(\epsilon)$ yields $\norm{B_n - f\Id}_{\Lp{\infty}(\R^n)} < \epsilon$.
See the proof of Theorem~\ref{Thm.limit.smooth.RRMs} to obtain this estimate from $\mg_n \to \mg$. 
Exploiting the equivalence between the Frobenius and operator norms and invoking Lemma~\ref{bundlemetriclem} with $\cE = \R^n \tensor \R^n$ and $\mgt = \delta_{\R^{n}} \tensor \delta_{\R^n}$, we deduce that $f \Id \in \Ck{0}(\R^n;\R^n)$.
That is, $f \in \Ck{0}(\R^n)$, which is a contradiction. 

\subsection{Preservation of spaces on components} 
\label{S:Preserve} 

It was already established in \cite{BRough}, although not in the notation of this paper, that important spaces, as listed below, remain fixed (as sets) with equivalent norms across components $\Comp(\mg)$.
That is, for all $\mh \in \Comp(\mg)$:
\begin{enumerate}[label=(\roman*)] 
\item  For  $x$-almost-everywhere and  for all $u \in \tanb_x\cM$, 
\[
\e^{-\rmet^\cM(\mg,\mh)} \modulus{u}_{\mg(x)} \leq \modulus{u}_{\mh(x)} \leq \e^{\rmet^\cM(\mg,\mh)} \modulus{u}_{\mg(x)}.
\]
\item For all $x, y \in \cM$, 
\[
\e^{-\rmet^\cM(\mg,\mh)} \met_{\mg}(x,y) \leq \met_{\mh}(x,y) \leq \e^{\rmet^\cM(\mg,\mh)} \met_{\mg}(x,y).
\]
\item For $p \in [1, \infty)$, $\Lp{p}(\cM;\Tensors[r,s]\cM,\mg) = \Lp{p}(\cM;\Tensors[r,s]\cM,\mg)$ as sets and  
\begin{align*}
\e^{-(r+s+\frac{n}{2p} \rmet^\cM(\mg,\mh))} \norm{u}_{\Lp{p}(\Tensors[r,s]\cM),\mg)}
&\leq
\norm{u}_{\Lp{p}(\cM;\Tensors[r,s]\cM),\mh)}\\
&\qquad\leq
\e^{(r+s+\frac{n}{2p} \rmet^\cM(\mg,\mh))} \norm{u}_{\Lp{p}(\cM;\Tensors[r,s]\cM),\mg)}.
\end{align*}
\item For $p \in [1,\infty)$,  $\Sob{1,p}(\cM,\mg) = \Sob{1,p}(\cM,\mh)$ as sets and
\[
\e^{-(1+\frac{n}{2p} \rmet^\cM(\mg,\mh))} \norm{u}_{\Sob{1,p}(\cM,\mh)}.
\leq
\norm{u}_{\Sob{1,p}(\cM,\mg)}
\leq
\e^{(1+\frac{n}{2p} \rmet^\cM(\mg,\mh))} \norm{u}_{\Sob{1,p}(\cM,\mh)}.
\]
\end{enumerate}

\subsection{Preservation of Poincaré inequalities}
\label{S:Poincare}

Let $\mh \in \Met(\cM)$.
For a measurable subset $X \subset \cM$ with $\mu_{\mh}(X) < \infty$, the average $\Av_{\mh, X}:\Lp{0}(X) \to [0,\infty]$ is defined by 
\[ 
\Av_{\mh,X}u  = \fint_{X} u\ d\mu_{\mh} = \frac{1}{\mu_{\mh}(X)} \int_{X} u\ d\mu_{\mh}. 
\]

From \eqref{eq:d_g}, we have a well-defined distance $\met_{\mh}: \cM \times \cM \to \R$ associated to $\mh$. 
Let $\Ball_{\mh}(x,r)$ be the $\met_{\mh}$-ball centred at $x \in \cM$ of radius $r > 0$. 
We assume the that $\mu_{\mh}(\Ball_{\mh}(x,r)) < \infty$ for all $x \in \cM$ and $r > 0$ and say that $\mh$ satisfies a $(p,q)$-\emph{generalised local Poincaré inequality} at $x \in \cM$ for $p, q \in [1,\infty]$ if there exist constants $\eta(x,\mh), C_1(x,r,\mh) \in (0, \infty)$ and $C_2(x,r,\mh) \in [0,\infty)$ such that for all $r < \infty$ and $u \in \Ck{\infty}(\cM)$, 
\begin{equation} 
\label{Def:Poin}
\begin{aligned}
&\norm{{u - \Av_{\mh, \Ball_{\mh}(x,r)} u}}_{\Lp{p}(\Ball_{\mh}(x,r), \mh)} \\
&\qquad\qquad\leq C_1(x,r,\mh) \norm{{\nabla u}}_{\Lp{q}(\Ball_{\mh}(x,\eta(x,\mh) r),\mh)} + C_2(x,r,\mh) \norm{ {u}}_{\Lp{q}(\Ball_{\mh}(x, \eta(x,\mh)),\mh)}.
\end{aligned}
\end{equation}
The case of $C_2(x,r,\mh) = 0$ corresponds to a homogeneous Poincaré inequality.

Let $\mg \in \Comp(\mh)$ or equivalently $\rmet^\cM(\mg,\mh) < \infty$.
We show that $\mg$ satisfies a $(p,q)$-generalised local Poincaré inequality at $x \in \cM$ if $\mh$ satisfies a $(p,q)$-generalised local Poincaré inequality at $x \in \cM$.
The new constants depend on $C_1(x,r,\mh), C_2(x,r,\mh), \eta(x,\mh)$ and $\rmet^\cM(\mg,\mh)$.

Set $\theta > 1$ to be chosen later  and  note that 
\begin{equation} 
\label{Eq:Poin1} 
\begin{aligned}
\norm{{u - \Av_{\mg, \Ball_{\mg}(x,r)} u}}_{\Lp{p}(\Ball_{\mg}(x,r), \mg)}
&\leq
\norm{{u - \Av_{\mh, \Ball_{\mh}(x,\theta r)} u}}_{\Lp{p}(\Ball_{\mg}(x,r), \mg)} \\
&\qquad + 
\norm{{\Av_{\mh,\Ball_{\mh}(x,\theta r)}u - \Av_{\mg, \Ball_{\mg}(x,r)} u}}_{\Lp{p}(\Ball_{\mg}(x,r), \mg)} 
\end{aligned}
\end{equation}
We consider the second quantity, noting that the average produces a constant: 
\begin{align*} 
&\norm{{\Av_{\mh,\Ball_{\mh}(x,\theta r)}u - \Av_{\mg, \Ball_{\mg}(x,r)} u}}_{\Lp{p}(\Ball_{\mg}(x,r), \mg)}  \\
&\qquad= \cbrac{\int_{\Ball_{\mg}(x,r)} \modulus{\Av_{\mh,\Ball_{\mh}(x,\theta r)}u - \Av_{\mg, \Ball_{\mg}(x,r)} u}^p\ d\mu_{\mg}}^{\frac1p} \\
&\qquad= \modulus{\Av_{\mh,\Ball_{\mh}(x,\theta r)}u - \Av_{\mg, \Ball_{\mg}(x,r)} u} \mu_{\mg}(\Ball_{\mg}(x,r))^{\frac1p} \\ 
&\qquad= \modulus{ \Av_{\mg, \Ball_{\mg}(x,r)}(u -  \Av_{\mh,\Ball_{\mh}(x,\theta r)}u)} \mu_{\mg}(\Ball_{\mg}(x,r))^{\frac1p} \\ 
&\qquad= \mu_{\mg}(\Ball_{\mg}(x,r))^{-1} \modulus{ \int_{\Ball_{\mg}(x,r)} (u - \Av_{\mh,\Ball_{\mh}(x,\theta r)}u)\cdot 1} \mu_{\mg}(\Ball_{\mg}(x,r))^{\frac1p} \\
&\qquad\leq \mu_{\mg}(\Ball_{\mg}(x,r))^{\frac1p-1} \norm{u - \Av_{\mh,\Ball_{\mh}(x,\theta r)}u}_{\Lp{p}(\Ball_{\mg}(x,r),\mg)} \cbrac{\int_{\Ball_{\mg}(x,r)} 1^{p'}}^{\frac1p'} \\
&\qquad= \mu_{\mg}(\Ball_{\mg}(x,r))^{\frac1p-1} \norm{u - \Av_{\mh,\Ball_{\mh}(x,\theta r)}u}_{\Lp{p}(\Ball_{\mg}(x,r),\mg)} \mu_{\mg} (\Ball_{\mg}(x,r))^{\frac1p'} \\
&\qquad=  \mu_{\mg}(\Ball_{\mg}(x,r))^{\frac1p + \frac1p' -1}\norm{u - \Av_{\mh,\Ball_{\mh}(x,\theta r)}u}_{\Lp{p}(\Ball_{\mg}(x,r),\mg)} \\
&\qquad=  \norm{u - \Av_{\mh,\Ball_{\mh}(x,\theta r)}u}_{\Lp{p}(\Ball_{\mg}(x,r),\mg)},
\end{align*}
since $\frac1p + \frac1p' = 1$ in the invocation of the Cauchy-Schwartz inequality.
On noting that $\Ball_{\mg}(x,r) \subset \Ball_{\mh}(x, \e^{\rmet^\cM(\mg,\mh)})$, we choose $\theta := \e^{d^\cM(\mg,\mh)}$ and upon making that choice, combining this with \eqref{Eq:Poin1}, 
\begin{align*} 
\norm{{u - \Av_{\mg, \Ball_{\mg}(x,r)} u}}_{\Lp{p}(\Ball_{\mg}(x,r), \mg)}
&\leq 2  \norm{{u - \Av_{\mh, \Ball_{\mh}(x,\theta r)} u}}_{\Lp{p}(\Ball_{\mg}(x,r), \mg)} \\
&\leq 2  \theta^{\frac{n}{2p}} \norm{{u - \Av_{\mh, \Ball_{\mh}(x,\theta r)} u}}_{\Lp{p}(\Ball_{\mh}(x,\theta r), \mh)} \\
&\leq 2 C_1(x,r,\mh) \theta^{\frac{n}{2p}} \norm{\nabla u}_{\Lp{q}(\Ball_{\mh}(x,\eta(x,\mh) \theta r), \mh)} \\
&\qquad+ 2 C_2(x,r,\mh) \theta^{\frac{n}{2p}} \norm{u}_{\Lp{q}(\Ball_{\mh}(x,\eta(x,\mh) \theta r), \mh)} \\
&\leq 2 C_1(x,r,\mh) \theta^{\frac{n}{2p} + \frac{n}{2q} + 1} \norm{\nabla u}_{\Lp{q}(\Ball_{\mh}(x,\eta(x,\mh) \theta r), \mg)} \\
&\qquad+ 2 C_2(x,r,\mh) \theta^{\frac{n}{2p} + \frac{n}{2q}} \norm{u}_{\Lp{q}(\Ball_{\mh}(x,\eta(x,\mh) \theta r), \mg)} \\
 &\leq 2 C_1(x,r,\mh) \theta^{\frac{n}{2p} + \frac{n}{2q} + 1} \norm{\nabla u}_{\Lp{q}(\Ball_{\mg}(x,\eta(x,\mh) \theta^2 r), \mg)} \\
&\qquad+ 2 C_2(x,r,\mh) \theta^{\frac{n}{2p} + \frac{n}{2q}} \norm{u}_{\Lp{q}(\Ball_{\mg}(x,\eta(x,\mh) \theta^2 r), \mg)} \\ 
\end{align*}
Therefore, $\mg$ satisfies a $(p,q)$-generalised Poincaré inequality with constants
\begin{align*} 
&C_1(x,r,\mg) =  2C_1(x,r,\mh) \e^{\cbrac{\frac{n}{2p} + \frac{n}{2q} + 1}\rmet^\cM(\mg,\mh)}\\
&C_2(x,r,\mg) =  2C_2(x,r,\mh) \e^{\cbrac{\frac{n}{2p} + \frac{n}{2q}}\rmet^\cM(\mg,\mh)} \\
&\eta(x,\mg) = \eta(x,\mh) \e^{2 \rmet^\cM(\mg,\mh)}.
\end{align*}

Let us now restrict to $\mh \in \Met_{\Ck{\infty}}(\cM)$ and complete. 
This means that $\Ball_{\mh}(x,r)$ are compact by Hopf-Rinow and for any $\mg \in \Comp(\mh)$, we again have the same property.
Furthermore,  assume that there exists $K \geq 0$ such that 
\[
\Ric(\mh) \geq -K \mh.
\]
From Theorem~9.2 in \cite{SC}, we can set $p = q \in [1,\infty)$ to find \eqref{Def:Poin} satisfied with 
\begin{align*}
C_1(x,r,\mh) = A_p \e^{B_n \sqrt{K} r} r \quad
C_2(x,r,\mh) = 0 \quad \text{and} \quad
\eta(x,\mh) = 1,
\end{align*}
where $A_p$ is a constant dependent on $p$ and $B_n$ is a constant dependent on $n = \dim \cM$. 
Therefore, any $\mg \in \Comp(\mh)$ then satisfies
\[ 
\norm{ u - \Av_{\mg, \Ball_{\mg}(x,r)}u}_{\Lp{2}(\Ball_{\mg}(x,r),\mg)} \leq 2 A_p \e^{B_n \sqrt{K} r + (\frac{n}{p} + 1)\rmet^\cM(\mg,\mh)} r \norm{\nabla u}_{\Lp{p}(\Ball_{\mg}( \e^{2 \rmet^\cM(\mg,\mh)r},\mg)}.
\]

\subsection{Ricci curvature lower bounds and heat kernel regularity}
\label{S:Ricci}

Let $\mg \in \Met(\cM)$ and  $u \in \SobH{1}([0, \infty); \Lp{2}(\cM,\mg))$ be a solution to the heat equation
\[ 
(\partial_t u)(t,x) = (\Lap_{\mg} u)(t,x).
\]
It is well known that $u(t,x) = (\e^{-t \Lap_{\mg}} u)(x)$ and therefore, $t \mapsto u(t,\cdot)$ is analytic.
In \cite{BanBry}, it was shown that $(t,x) \mapsto u(t,x) \in \Ck{0}((0, \infty)\times \cM)$, though on any given precompact set $U$, there exists $\alpha > 0$ such that $(t,x) \mapsto u(t,x) \in \Ck{\alpha}((0,\infty) \times \cM)$.

Assume the following: 
\begin{enumerate}[label=(\roman*)] 
\item 
there is a metric  $\mh \in \Comp(\mg) \cap \Met_{\Ck{\infty}}(\cM)$ complete; 
\item 
there exists a constant $K \geq 0$ such that $\Ric(\mh) \geq -K \mh$.
\end{enumerate}

From Subsection~\ref{S:Preserve},  $\Lp{2}(\cM,\mg) = \Lp{2}(\cM,\mh)$ and $\SobH{1}(\cM,\mg) = \SobH{1}(\cM,\mh)$ as sets with equivalent norms.
Letting $\mg[u,v] = \mh[Bu,v]$ for $B \in \Ell(\cM)$ from Theorem~\ref{Thm:Stab} and choosing this to be the unique self-adjoint endomorphism, a routine calculation mimicking that in Subsection~\ref{S:DivForm} leads to  
\[
\Lap_{\mg} = -(\det B)^{-\frac12} \divv_{\mh} (\det B)^{\frac12}B \nabla.
\]
We have that 
\[ 
\e^{-\rmet^\cM(\mg,\mh)}\modulus{u}_{\mh} \leq \modulus{\sqrt{B}u}_{\mh}  \leq \e^{\rmet^\cM(\mg,\mh)}\modulus{u}_{\mh}
\]
and that 
\[
\e^{-\frac{n}{2} \rmet^\cM(\mg,\mh)} \leq (\det B)^{\frac12} \leq \e^{\frac{n}{2} \rmet^\cM(\mg,\mh)}.
\]
Applying Corollary~5.5 in \cite{SC} with a choice of $\delta = \frac12$, $r = 1$ and  noting that $\alpha = \e^{\rmet^\cM(\mg,\mh)}$ and $\mu = \e^{\frac{n}{2}\rmet^\cM(\mg,\mh)}$ are valid for our divergence-form equation, we obtain $\gamma$ dependent on dimension, $\alpha$ and $\mu$, or in other words dependent on $\rmet^\cM(\mg,\mh)$, such that 
\[
(t,x) \mapsto u(t,x) \in \Ck{\gamma}((0,\infty) \times \cM).
\]
Quantitatively, for $s > 1$ and $(t,y), (t',y') \in Q := (s - 1, s) \times \Ball_{\mh}(x,1)$ and with $Q' := (s - \frac12, s) \times \Ball_{\mh}(x, \frac12)$, we have a constant $C < \infty$ dependent only on dimension and $\rmet^\cM(\mg,\mh)$ such that 
\[ 
\modulus{u(t,y) - u(t',y')} \leq C (1 + \sqrt{K})^\gamma (\max\set{ \modulus{t - t'}, \met_{\mh}(y, y')})^\gamma \norm{u}_{\Lp{\infty}(Q)}
\]
In other words, the presence of a smooth complete metric in the component $\Comp(\mg)$ immediately improves the regularity of solutions to the heat equation with respect to any rough  Riemannian metric $\mg' \in \Comp(\mg)$.

More general divergence form operators of the form 
\[
\mathscr{L}u = -m^{-1} \divv_{\mg} (m A \nabla u  + m u X) + Yu + b u,
\]
where $m, A, X, Y, b$ are allowed to be $(t,x)$ dependent, not necessarily symmetric, equipped with natural upper and lower bounds to ensure the ellipticity of $\mathscr{L}$.
Solutions $u$ can be asserted to be Hölder regularity under the assumptions we have made on $\mh$.
This again follows immediately from Corollary~5.5 in \cite{SC}.

\subsection{Hodge theory for rough Riemannian metrics}
Let $\Forms[k]\cM \to \cM$ be the bundle of $k$-forms $0 \leq k \leq n$ and $\Forms\cM  = \oplus_{k=0}^n \Forms[k]\cM$.
On $k$-forms, the exterior derivative  $\extd: \Ck{\infty}(\cM; \Forms[k]\cM) \to \Ck{\infty}(\cM; \Forms[k+1]\cM)$, determined by the underlying differentiable structure on $\cM$.
With respect to $\mh \in \Met_{\Ck{\infty}}(\cM)$, $\extd$ has a formal adjoint $\extd^{\dagger,\mh}:\Ck{\infty}(\cM; \Forms[k]\cM) \to \Ck{\infty}(\cM; \Forms[k-1]\cM)$. 
The Hodge-Dirac operator is then given by $\Dir_{\mh,\infty} := \extd + \extd^{\dagger,\mh}: \Ck{\infty}(\cM;\Forms\cM) \to \Ck{\infty}(\cM;\Forms\cM)$. 
This is an elliptic first-order differential operator.

When $\cM$ is compact, $\Dir_{\mh,\infty}$ has a unique closure $\Dir_{\mh}$. 
The classic Hodge theorem can be paraphrased as
\[ 
\ker(\Dir_{\mh} \rest{\Forms[k]\cM}) \cong \Hom[k]_{\dR}(\cM) \cong \Hom[k]_{\Sing}(\cM),
\] 
where 
\[
\Hom[k]_{\dR}(\cM) = \faktor{\ker (\extd\rest{\Ck{\infty}(\cM;\Forms[k]\cM)})}{ \ran (\extd\rest{\Ck{\infty}(\cM;\Forms[k+1] \cM)})}
\] 
is the $k$-th de Rham cohomology (since $\extd^2 = 0$) and $\Hom[k]_{\Sing}(\cM)$ is the $k$-th singular cohomology.
A particular and important consequence of the Hodge theorem is that for another  $\mh' \in \Met_{\Ck{\infty}}(\cM)$,
\[ 
\ker(\Dir_{\mh'} \rest{\Forms[k]\cM}) \cong \ker(\Dir_{\mh} \rest{\Forms[k]\cM}).
\]

In \cite{BH}, the authors dispense with the smoothness assumption on $\mh$, considering a general $\mg \in \Met(\cM)$.
In the smooth setting, $\close{\Dir_{\mh}} = \close{\extd} + \extd^{\ast, \mh}$. 
The fact that this is a well-defined self-adjoint operator is readily asserted using the fact that $\Ck{\infty}(\cM;\Forms\cM) \subset \dom(\close{\extd}) \cap \dom(\extd^{\ast,\mh})$.
However, for a general $\mg \in \Met(\cM)$, we have  $\Ck{\infty}(\cM;\Forms\cM)  \not\subset \dom(\extd^{\ast,\mg})$ which complicates asserting  $\Dir_{\mg} := \close{\extd} + \extd^{\ast,\mg}$ is densely-defined.

Nevertheless, this can be overcome using techniques emanating from \cite{AKMC} and  the operator $\Dir_{\mg}$ is indeed self-adjoint.
The authors in \cite{BH} show that for compact $\cM$ and $\mg, \mg' \in \Met(\cM)$
\[ 
\ker(\Dir_{\mg} \rest{\Forms[k]\cM}) \cong \ker(\Dir_{\mg'} \rest{\Forms[k]\cM}).
\]
Since we can always choose $\mg' = \mh \in \Met_{\Ck{\infty}}(\cM)$, 
\[ 
\ker(\Dir_{\mg} \rest{\Forms[k]\cM}) \cong \Hom[k]_{\dR}(\cM) \cong \Hom[k]_{\Sing}(\cM).
\]
This, in particular, justifies the assertion that $\Met(\cM)$ contain metrics of \emph{geometric singularities} due to the fact that fundamental topological properties relating solutions of the Hodge-Dirac operator to the cohomology remain unchanged.

The authors \cite{BH} also consider the non-compact situation. 
Let $\cM$ now be a general smooth connected non-compact manifold and $\mg \in \Met(\cM)$.
Then, $\Met(\cM)$ will necessarily be disconnected by Theorem~\ref{Thm_com_iff_conn}. 

Although the exterior derivative $\extd$, as it is induced by the topology of $\cM$, is a well-defined object, it is more complicated to consider  $\extd$ as an unbounded operator on $\Lp{2}(\cM;\Forms\cM, \mg)$ for $\mg \in \Met(\cM)$.
This is due to the fact that $\extd$ might admit many closed extensions.
In fact, this even arises for smooth but incomplete metrics.
Therefore, the authors in \cite{BH} consider two closed extensions of $\extd$. The first is denoted by  $\extd_2$,  obtained by taking the $\Lp{2}$-closure of $\extd$ on $\set{ u \in \Ck{\infty}(\cM;\Forms\cM): \extd u \in \Lp{2}(\cM;\Forms\cM,\mg)}$. The second is $\extd_0$, taking the $\Lp{2}$-closure of $\Ck[cc]{\infty}(\cM;\Forms\cM) := \set{u \in \Ck{\infty}(\cM;\Forms\cM): \text{ and compact}}$.

Both of these extensions are \emph{nilpotent}, meaning that $\ran(\extd_2) \subset \ker(\extd_2)$ and $\ran(\extd_0) \subset \ker(\extd_0)$.
It is clear, however, that $\extd_0 \subset \extd_2$ in the sense of containment of domains. 
There are closed extensions containing $\extd_0$ and contained in $\extd_2$ which fail to be nilpotent, so fix $\extd_e$  closed and nilpotent (again meaning that $\ran(\extd_e) \subset \ker(\extd_e)$) such that $\extd_0 \subset \extd_e \subset \extd_2$.

If $\mh \in \Comp(\mg)$ is another rough Riemannian metric, then $\Lp{2}(\cM;\Forms\cM,\mh) = \Lp{2}(\cM;\Forms\cM,\mg)$ as sets and with equivalent norms by Subsection~\ref{S:Preserve}.
From the vantage point of $\Lp{2}(\cM;\Forms\cM,\mh)$, due to the equivalence of norms $\norm{\cdot}_{\Lp{2}(\mg)} \simeq \norm{\cdot}_{\Lp{2}(\mh)}$, the operator $\extd_e$ remains a closed operator with the same domain.  
However, the adjoints depend on $\mg$ and $\mh$ and therefore, we consider
\[ 
\Dir_{e,\mh} := \extd_e + \extd_e^{\ast, \mh} \quad\text{and}\quad \Dir_{e,\mh} := \extd_e + \extd_e^{\ast, \mh}.
\]
While these operators no longer satisfy elliptic regularity, the nilpotency of $\extd_e$, via techniques developed by \cite{AKMC} are self-adjoint operators (with respect to the norm  respectively with the norm induced by $\mg$ or $\mh$).
Without referencing the metric, they are $0$-bisectorial operators.
In \cite{BH}, the authors show 
\[ 
\ker(\Dir_{e,\mg} \rest{\Forms[k]\cM}) \cong \ker(\Dir_{e,\mh} \rest{\Forms[k]\cM}),
\]
even though these spaces are generally infinite-dimensional.
In other words, kernels of the Hodge-Dirac operators of nilpotent extension, remain isomorphic across connected components of $\Met(\cM)$.

We remark that the extensions $\extd_0$ and $\extd_2$ have meaning in the situation of a compact manifold with boundary. 
To fit our framework, we can consider the interiors of these manifolds and consider the boundary trace map as a limit.
In that setting, these extensions respectively compute the \emph{relative} and \emph{absolute} cohomology of such a manifold.
In that setting, the work of \cite{BH} tells us that the regularity (or lack thereof) of a rough Riemannian metric $\mg \in \Met(\cM)$ retains the underlying cohomological information.

\subsection{Functional calculus and the Kato square root problem} 
\label{S:FCKato} 
Let $A \in \Lp{\infty}(\R^n; \R^n)$ be a complex matrix-valued function such that there exist $\kappa, \Lambda \in (0,\infty)$ such that
\[ 
\Re Au\cdot u \geq \kappa\ \text{a.e.}\quad \text{and}\quad \norm{A}_{\Lp{\infty}(\R^n)} \leq \Lambda.
\]
Define $\cE_{A}: \SobH{1}(\R^n) \times \SobH{1}(\R^n) \to \C$ by $\cE_{A}[u,v] = \inprod{A \nabla u, \nabla v}$.
This is a densely-defined sesquilinear form and by the Lax-Milgram theorem, we obtain $L_{A} := - \divv A \nabla$ as a $\omega$-sectorial operator, where $\omega < 2\pi$.
As such, using the branch cut along the negative real-axis, functional calculus yields a square root $\sqrt{L_{A}}: \dom(\sqrt{L_{A}}) \subset \SobH{1}(\R^n) \to \Lp{2}(\R^n)$, a densely-defined $\omega/2$-sectorial operator.

The conjecture was initially due to Kato in 1961 in  \cite{Kato61} in a more abstract setting.
It was refined by McIntosh in 1972 in\cite{McIntosh72} through providing to a counterexample to Kato's initial conjecture. 
The reformulation was to assert that $\dom(\sqrt{L_{A}}) = \SobH{1}(\R^n)$.
In dimension $1$, this has connections to the boundedness of the Cauchy integral operator on a Lipschitz curve. 
This was solved via the use of functional calculus with links to harmonic analysis by Coifman-Meyer-McIntosh in 1982 \cite{CMMc}.
The general problem was resolved affirmatively by Auscher-Hofmann-Lacey-McIntosh-Tchamitchian in 2002 in \cite{AHLMcT}, which relied on scalar-valued real-variable harmonic analysis techniques.
A first-order version of the proof, closely mirroring the $1$-dimensional setting, was given by Axelsson(Rosén)-Keith-McIntosh in 2005 in \cite{AKMC}.
This featured the first version of the Kato square root problem on manifolds, albeit in the compact setting.

On a manifold $\cM$ with metric $\mg \in \Met(\cM)$, consider the unbounded nilpotent operator $\Gamma: \Lp{2}(\cM,\mg) \oplus \Lp{2}(\tanb\cM,\mg) \to \Lp{2}(\cM,\mg) \oplus \Lp{2}(\tanb\cM,\mg)$, with adjoint $\Gamma^{\ast,\mg}$ and the perturbed operator $\Pi_{B}$ by
\[ 
\Gamma := \begin{pmatrix} 
0 & 0  \\
\nabla & 0 
\end{pmatrix}, 
\quad
\Gamma^{\ast,\mg} 
= \begin{pmatrix}
0 & -\divv \\
0 & 0 
\end{pmatrix}, 
\quad
\text{and}
\quad
\Pi_{B} := \Gamma + B_1 \Gamma^{\ast,\mg} B_2. 
\]
The coefficients $B$ can be carefully constructed from a given $A \in \Lp{\infty}(\cM; \End(\cotanb\cM))$ with $\Re \mg[Au,u] \geq \kappa$ and $\norm{A}_{\Lp{\infty}(\mg)} < \Lambda$ such that
\[ 
\Pi_{B}^2 
= 
\begin{pmatrix} 
-\divv A \nabla & 0 \\
0 & -\nabla \divv A
\end{pmatrix}.
\]
If $\sqrt{\Pi_{B}^2}$ exists and  $\dom(\sqrt{\Pi_{B}^2}) = \dom(\Pi_{B})$ holds, then projecting to the first coordinate, an affirmative answer can be obtained for the geometric Kato square root problem.
Under boundedness and ellipticity hypothesis on $A$, $\Pi_{B}$ is $\omega$-bisectorial for $\omega < \pi/2$. 
The key is to prove quadratic estimates of the form 
\begin{equation} 
\label{Eq:Qest} 
\int_{0}^\infty \norm{\psi(\Pi_{B}) u}^2_{\Lp{2}(\mg)} \frac{dt}{t} \simeq \norm{u}_{\Lp{2}(\mg)}
\end{equation}
for all $u \in \ran(\Pi_B)$. 
Here $\psi: \interior{S}_{\mu} \subset \C \to \C$ is a holomorphic function on th open bisector $\interior{S}_{\mu}$ with $\mu > \omega$, which decays polynomially to $0$ at $0$ and $\infty$.
This provides a functional calculus $f(\Pi_{B})$ for bounded functions $f:S_{\mu} \to \C$ holomorphic on $\interior{S}_{\mu}$, providing $\dom(\sqrt{\Pi_{B}^2}) = \dom(\Pi_{B})$ and also yielding $B \mapsto f(\Pi_{B})$ is holomorphic.

In \cite{BRough}, in the language of this paper, it is shown that if \eqref{Eq:Qest} is valid for $\mg \in \Met(\cM)$, then it is valid for all $\mh \in \Comp(\mg)$ (with possibly different constants). 
In particular, if $\cM$ is compact, then the geometric Kato square root problem has a solution for every $\mg \in \Met(\cM)$.

Also in \cite{BRough}, an inhomogeneous version of the Kato square root problem is also considered, where the space in question is no longer $\Lp{2}(\cM,\mg) \oplus  \Lp{2}(\cotanb\cM,\mg)$ but rather $\Lp{2}(\cM,\mg) \oplus  \Lp{2}(\cM,\mg) \oplus \Lp{2}(\cotanb\cM,\mg)$ and $\nabla$ is replaced by $(\ident, \nabla)$.
In geometric settings, the inhomogeneous problem is more applicable. 
On a compact manifold, the two are equivalent which can be seen easily from the fact that the Laplacian exhibits a spectral gap around $0$.
The inhomogeneous problem was solved by Morris in \cite{Morris} for submanifolds of $\R^{n+k}$ with second fundamental form bounds, later improved to the intrinsic setting with bounded Ricci curvature and injectivity radius bounds by Bandara-McIntosh in \cite{BMc}, and more recently to uniform lower bounded Ricci curvature and bounded injectivity radius by Auscher-Morris-Rosén \cite{AMR}.
The results in \cite{BRough} apply to all these contexts and and so in particular, if $\mh \in \Comp(\mg)$ and $\mg \in \Met_{\Ck{\infty}}(\cM)$, complete, and has uniform positive lower bounds on injectivity radius bounds  and its its Ricci curvature is uniformly bounded below (possibly by a negative number), then the inhomogeneous Kato problem can be solved for $\mh$.

\subsection{Convergence and regularity}
\label{S:ConvReg} 
Let $\cM$ and $\cN$ be a smooth connected manifolds and $\mh \in \Met(\cN)$.
Fix an immersion $t \mapsto \Phi_t \in \Ck{0}([0,1); \Ck{0,1}(\cM, \cN))$, by which we mean that $\extd \Phi_t \neq 0$ for every $t \in [0,T)$.
We further assume that it is uniformly bounded: there exists $C \geq 1$ such that 
\begin{equation} 
\label{Eq:ConvReg}
\frac1C \modulus{(\extd \Phi_0)(x)u}_{\mh(\Phi_0(x))} \leq \modulus{(\extd \Phi_t)(x)u}_{\mh(\Phi_t(x))} \leq C  \modulus{(\extd \Phi_0)(x)u}_{\mh(\Phi_0(x))}
\end{equation} 
for $x$-almost-everywhere in $\cM$ and for all $u \in \tanb_x \cM$.

Clearly, this allows us to obtain metrics $\mg_t := \Psi_t^\ast (\mh \rest{ \Phi_t(\cM)})$ on $\cM$ since $\extd \Phi_t \neq 0$ for all $t \in [0,T)$.
Moreover, condition~\eqref{Eq:ConvReg} yields that $\mg_t \in \Comp(\mg_0)$.

Suppose now that we obtain $\rmet^\cM(\mg_t, \mg_s) \to 0$ as $t \to T$.
Theorem~\ref{Thm:Complete} then guarantees that $\mg_T \in \Comp(\mg_0) \subset \Met(\cM)$ and that $\mg_t \to \mg_T$.

If $\mh \in \Met_{\Ck{k}}(\cN)$ for $k \geq 0$ and $t\mapsto \Phi_t  \in \Ck{0}([0,T); \Ck{l}(\cM;\cN))$ for $l \geq 1$, then Theorem~\ref{Thm.limit.smooth.RRMs} yields that $\mg_T \in  \Met_{\Ck{0}}(\cM)$.

This could potentially be applied to understand regularity properties in geometric flows. 
For instance, such immersions $\Phi_t$ could arise from an extrinsic geometric flow, such as the mean curvature flow.
Such flows typically admit singularities at extinction time, which are severe and result in the change of topology. 
Nevertheless, if it is possible to obtain this $\Lp{\infty}$-bound as found in \eqref{Eq:ConvReg} along with the Cauchy property (up to subsequence), then although the immersed limit object might seem to have topological singularities, it will in fact be a smooth manifold with the singularities confined to the geometry. 

\subsection{The set of induced metric-measure spaces}
\label{S:InducedMMs}

Consider the following set
\begin{equation}
\label{Eq:MMset} 
\Spa(\cM) := \set{ (\met_{\mg}, \mu_{\mg}): \mg \in \Met(\cM) }.
\end{equation} 
From Subsections~\ref{S:Meas} and \ref{S:Dist}, we have that 
\[ 
\Met(\cM) \ni \mg \mapsto (\met_{\mg}, \mu_{\mg})
\]
defines a surjection $\Xi: \Met(\cM) \to \Spa(\cM)$.
A natural question to ask is whether information is lost when studying $\Xi (\Met(\cM))$ rather than $\Met(\cM)$.
Concretely, this is to ask whether the map $\Xi$ is an injection.

We answer this question in the negative.
For that, fix $\R^4$ which we write as the product space $\R^2 \times \R^2$.
By invoking Theorem~2 in \cite{Sturm} with $\delta = \frac12$ and $a = \delta_{\R^4}$, we obtain a new metric $\mg_1(x) := \Psi(x) \delta^{R^4}$ where $\frac12 < \Psi < 1$ almost-everywhere but with $\met_{\mg_1} = \met_{\delta_{\R^4}}$.

To disprove injectivity, we need to produce metrics $\mg, \mg' \in \Met(\cM)$ such that $\met_{\mg'} = \met_{\mg} = \met_{\delta_{\R^4}}$ with $\mu_{\mg} = \mu_{\mg'}$.
For that, we tweak Theorem~2 in \cite{Sturm} slightly for our purposes.

In the construction of $\Psi$, we choose $\set{x_k} \subset \R^4$ a countable dense subset and curves $\gamma_{k,l,m} \in \Ck{0,1}([0,1],\R^4)$ between $x_k$ and $x_l$ such that  $\len_{\delta_{\R^4}}(\gamma_{k,l,m}) \leq ( 1 + \frac 1m) \modulus{x_k - x_l}$. 
Using the same procedure, there are functions  $\psi_{k,l,m} \in\Ck{0}(\R^4)$ such that 
\[
\Leb( \set{x \in\R^4: \psi_{k,l,m}(x) > \alpha}) < (1 - \alpha) 2^{-k-l-m} \epsilon
\]  and where $\psi_{k,l,m} = 1$ on $\gamma_{k,l,m}([0,1])$.
The function $\Psi(x) = \frac12 + \frac12 \Psi_0$ where $\Psi_0 := \sup\limits_{k,l,m \in \Na} \psi_{k,l,m}$.

Fix $\beta \in (0,1]$ and note that $\Psi_0^\beta = \sup\limits_{k,l,m \in \Na} \psi_{k,l,m}^\beta$.
Define the metric $\mg_{\beta}(x) := \Psi(x)^\beta \delta_{\R^4}$.
Now, by the continuity of $\psi_{k,l,m}$, we have that 
\[
U_{k,l,m} := \set{x \in \R^4:  \psi_{k,l,m}(x) >  2 \cbrac{ \frac{1}{(1 + \frac1m)^2} - \frac12}}
\]
is an open set.
For a regular point $x \in U^\beta_{k,l,m}$,
\begin{align*} 
\frac12 + \frac12 \Psi_0(x)
>
\frac12 + \frac12 \psi_{k,l,m}(x) 
> 
\frac{1}{(1 + \frac1m)^2}.
\end{align*}
Therefore,  $\Psi^\beta(x)  > \frac{1}{(1 + \frac1m)^{2\beta}}$ and $(1 + \frac1m)^{2\beta} \delta_{\R^4} \leq (1 + \frac1m)^{2\beta} \mg_\beta(x)$ in the sense of forms and 
\begin{align*}
\frac{1}{\cbrac{1 + \frac1m}^{\beta}} \modulus{x_k - x_l} 
\leq 
\frac{1}{\cbrac{1 + \frac1m}^{\beta}} \len_{\delta_{\R^4}}(\gamma_{k,l,m}) 
&\leq  
\len_{\met_{\mg_\beta}}(\gamma_{k,l,m}) \\
&\leq \len_{\delta_{\R^4}}(\gamma_{k,l,m}) 
\leq 
\cbrac{1 + \frac1m} \modulus{x_k - x_l},
\end{align*} 
where the penultimate inequality follows from the fact that $\Psi_0^\beta \leq 1$ and $\mg_\beta \leq \delta_{\R^4}$ and the ultimate one from the choice of $\gamma_{k,l,m}$.
By Proposition~\ref{Prop:Intrinsic}, the distance $\met_{\mg_\beta}$ is a length metric and since there are curves $\gamma_{k,l,m}$ for arbitrarily large $m \in \Na$, and $\set{x_k} \subset \R^4$ is dense, we see that $\met_{\mg_\beta} = \met_{\delta_{\R^4}}$ for any choice of $\beta \in (0,1]$. 

Fix $\theta \in (0,1]$ and define the metric $\mh_{\theta}(x) := (\Psi_0^\theta\ \delta_{\R^2}) \oplus \delta_{\R^2}$ on decomposing $\R^4 = \R^2 \times \R^2$.
By the same construction as above, we see that  for almost-every $x \in U_{k,l,m}$, fixing $u = (u_1, u_2) \in \tanb_x \R^2 \oplus \tanb_x \R^2 = \tanb_x \R^4$, 
\begin{align*}
\mh_{\theta}(x)[u,u] 
= 
\Psi_0(x)^\theta \modulus{u_1}^2_{\delta_{\R^2}} + \modulus{u_2}^2_{\delta_{\R^2}} 
&> 
\frac{1}{(1 + \frac1m)^{2\theta}} \modulus{u_1}^2_{\delta_{\R^2}} + \modulus{u_2}^2_{\delta_{\R^2}}  \\ 
&> 
\frac{1}{(1 + \frac1m)^{2\theta}} (\modulus{u_1}^2_{\delta_{\R^2}} + \modulus{u_2}^2_{\delta_{\R^2}})
=
\frac{1}{(1 + \frac1m)^{2\theta}} \modulus{u}^2_{\delta_{\R^4}}
\end{align*}
Trivially, $\mh_{\theta} \leq \delta_{\R^4}$ almost-everywhere as forms, and by the same argument as above, we deduce that $\met_{\mh_{\theta}} = \met_{\delta_{\R^4}}$ for all $\theta \in (0,1]$.

Now, lets compute the volume element of each of these metrics. 
First, we note that $\det \mg_{\beta} = \det ( \Psi_0^\beta \ident_{\R^4} ) = \Psi_0^{4\beta}$ and  $\det \mh_{\theta} = \det ( (\Psi_0^\theta \ident_{\R^2}) \oplus \ident_{\R^2} ) = \Psi^{2\theta}$. 

Therefore, choosing $\theta = 1$ and $\beta = \frac12$, we find that $\det \mh_{1} = \det \mg_{\frac12}$.
Let $\mg := \mh_{1}$ and $\mg' := \mg_{\frac12}$ and we find that 
\[ 
\Xi(\mg) = (\met_{\mg}, \mu_{\mg}) = (\met_{\delta_{\R^4}}, \det \Psi_0\ d\Leb) =   (\met_{\mg'}, \mu_{\mg'}) = \Xi(\mg').
\]
This shows that $\Xi: \Met(\cM) \to \Spa(\cM)$ is not injective.
Therefore, the study we conduct in this paper on $\Met(\cM)$ is a richer space which carries more information than $\Spa(\cM)$. 

In contrast, we note that if we restrict to $\Met_{\Ck{0}}(\cM)$, then by Proposition~4 in \cite{Sturm}, $\Xi\rest{\Met_{\Ck{0}}(\cM)}:  \Met_{\Ck{0}}(\cM) \to \Spa(\cM)$ is injective. 
In fact, an even stronger statement holds: the map $\Met_{\Ck{0}}(\cM) \to \set{ \met_\mg }$ given by $\mg \mapsto \met_{\mg}$ is an injection.

\section{Rough Riemannian metrics}
\label{S:RRMs} 
The space of rough Riemannian metrics is the  primary subject of study in paper.
However, before we embark on considering this set, recall the definition of the elements of this space, \emph{rough Riemannian metrics} (or RRMs) from Definition~\ref{Def:RRM}.

RRMs generalise smooth as well as continuous Riemannian metrics.
Conceptually speaking, a RRM allowed to be low-regular, degenerate on sets that have measure zero, but retain local boundedness and ellipticity.
The example in Subsection~\ref{S:Pullback}  shows how RRMs can capture geometric singularities of metrics obtained through Lipschitz pullbacks, including Lipschitz graphs and Subsection~\ref{S:Stab} shows how singularities can arise through conformal rescalings.

Definition~\ref{Def:RRM} can be equivalently formulated by requiring a RRM $\mg$ to be a symmetric $(2,0)$ tensorfield such that $\mg, \mg^{-1} \in \Lp[loc]{\infty}(\Sym \Tensors[2,0]\cM)$, where by $\mg^{-1}$ we denote the inverse matrix corresponding to $\mg$ in a given smooth frame.
See Lemma~\ref{Lem:RoughMetEquiv}.

Since rough Riemannian metrics carry singularities, we define the singular and regular set associated to a rough Riemannian metric $\mg$: 
\begin{align*} 
\Sing(\mg) &= \set{x \in \cM: \exists u \in \tanb_x\cM\ \text{s.t.}\  \modulus{u}_{\mg(x)} = 0\ \text{or}\ \mg(x)\ \text{is not defined}} \\ 
\Reg(\mg) &= \cM \setminus \Sing(\mg).
\end{align*} 
By definition $\mu_{\mg}(\Sing(\mg)) = 0$.

\subsection{The induced measure}
\label{S:Meas}

As aforementioned in Section~\ref{S:Results}, the metric-independent measure structure can be captured by an induced measure $\mu_{\mg}$ from a metric $\mg \in \Met(\cM)$.
This measure is given by the local expression 
\begin{equation}\label{roughmeasure}
d\mu_g(x)= \sqrt{\det \mg(x)} d\psi^{*}\mathcal{L}(x), 
\end{equation}
inside a coordinate patch $(U,\psi)$.
Verifying this yields a well-defined measure is readily argued by covering $\cM$ by locally comparable charts (see Definition~\ref{Def:RRM}).

In Proposition~6~\cite{BRough}, $\mu_{\mg}$  has been shown to be Borel and finite on a compact sets. 
The following proposition expands the properties of the measure $\mu_{\mg}$ and shows that it is actually a Radon measure. 
\begin{prop}
\label{Prop:Radon}
On a smooth connected manifold $\cM$ and for $\mg \in \Met(\cM)$, the induced volume measure $\mu_{\mg}$ is Radon. 
\end{prop}
\begin{proof}
To show that  $\mu_{\mg}$ is Radon, we need to show that it is inner-regular and that it is locally finite.
Local finiteness follows from Proposition~\cite{BRough} since $\mu_{\mg}$ is finite on compact sets and there is always a pre-compact open set around a given point $x \in \cM$.

Therefore, we prove that  $\mu_{\mg}$ is inner-regular. 
Recall this means that for every open set $A \subset \cM$,
\[
\mu_{\mg}(A)= \sup \{\mu_{\mg}(K):K \subset A, \text{ K is compact set }\}.
\]
For convenience, let us define this quantity as
\[\mu_{\mg_*}(A):= \sup \{\mu_{\mg}(K):K \subset A, \text{ K is compact set }\}.\] 
To prove inner-regularity, we show that $ \mu_{\mg_*}(A) \leq \mu_{\mg}(A)$ and $\mu_{\mg}(A)\leq \mu_{\mg_*}(A)$. 
Let $K$ be any compact subset such that $K \subset A$. 
Since $\mu_{\mg}(K)\leq \mu_{\mg}(A)$, 
\[
\mu_{\mg_*}(A)= \sup \{\mu_{\mg}(K):K \subset A, \text{ K is compact set }\} \leq \sup\{\mu_{\mg}(A)\}=\mu_{\mg}(A).
\]
This proves $\mu_{\mg_{\ast}}(A) \leq \mu_{\mg}(A)$.

To prove the remaining inequality, we consider the two cases:  $\mu_{\mg}(A) < \infty$ and $\mu_{\mg}(A)= \infty$.
Assume that $\mu_{\mg}(A) < \infty$. 
Since $\cM$ is a manifold, it is locally compact and Hausdorff.
Moreover, as $A$ is an open subset of $\cM$, it is itself is a locally compact space.
Therefore, for all $x \in A$, there exists a neighbourhood $V_x$ which is diffeomorphic to a Euclidean ball such that its closure contained in $A$. 
Thus, the collection 
\begin{equation}
\label{Eq:Cover} 
\mathcal{C}=\{V_x: x\in A \text{ such that } \overline{V_x}\subseteq A \text{ and diffeomorphic to a Euclidean ball }\}
\end{equation}  
is an open cover for $A$. 
Furthermore, since $\cM$ is second-countable so $A$ inherits this property.
Every second-countable space is Lindelöf and therefore, the collection $\mathcal{C}$ has a countable subcover for $A$, say $\tilde{\mathcal{C}} = \{V_{x_i}\}_{i=1}^{\infty}$. 
Define a compact subset of $A$ as follows 
\[K_N= \cup^{N}_{i=1} \overline{V_{x_i}}.\]
Clearly, from construction $K_N \subset A$.
% since V_xi \subset of A for all i, then K_N subset of A. Moreover, K_N is compact since the finite union of compact sets is compact.
Since $\mu_{\mg}(A) < \infty$ is finite, we have that
\[
\mu_{\mg}(K_N) = \int_M \chi_{K_N}(y) \,d\mu_{\mg}(y) < \int_M \chi_{A}(y) \,d\mu_{\mg}(y) = \mu_{\mg}(A) < \infty.
\] 
Note also that $\cup_{N=1}^\infty K_N = A$ by definition, and therefore, $\lim_{N \to \infty} \chi_{K}(y) = \chi_{A}(y)$.
That is, $\chi_{K_N} \to \chi_A$ point-wise almost-everywhere. 
Coupling this with  $\chi_{K_N}\leq \chi_A$ and  applying the dominated convergence theorem, 
\[
\lim\limits_{N \to \infty} \int_M \chi_{K_N}(y) \,d\mu_{\mg}(y)= \int_M \lim\limits_{N\to \infty}\chi_{K_N}(y) \,d\mu_{\mg}(y)= \int_M \chi_{A}(y) \,d\mu_{\mg}(y), 
\] 
which yields $\lim\limits_{N\to \infty}\mu_{\mg}(K_N)=  \mu_{\mg}(A)$.
Therefore, for a given $\epsilon > 0$, there exists $N_{\epsilon} \in \Na$ such that 
\begin{align*}
\mu_{\mg}(A)&\leq \mu_{\mg}(K_N) + \epsilon
\leq \sup \{\mu_{\mg}(K): K \text{ compact and } K \subset A\} + \epsilon 
\leq  \mu_{\mg_*}(A) + \epsilon.
\end{align*}
Since $\epsilon > 0$ was arbitrary, we conclude $\mu_{\mg}(A) \leq \mu_{\mg_*}(A)$.
 
Now we consider $\mu_{\mg}(A) = \infty$.
For that, we first show that $A$ can be written as a countable union of disjoint precompact subsets of $A$. 
Consider the collection $\tilde{\mathcal{C}}$ from \eqref{Eq:Cover} which we utilised earlier. 
As we have seen, this is a subcover for $A$.
Since $\overline{V_{x_i}} \subset A$, $\cup_{i=1}^\infty \overline{V_{x_i}} \subseteq A$ and so  $A= \cup_{i=1}^\infty \overline{V_{x_i}}$. 

For notational simplicity, let $B_i := V_{x_i}$. 
%Therefore, it is easy to see that \[\cup_{i=1}^\infty \overline{V_{x_i}} \subseteq A \cup_{i=1}^\infty V_{x_i} \subseteq  \cup_{i=1}^\infty \overline{V_{x_i}},\]
Define $\{X_i\}$ by:
\[
X_1 := \overline{B_1} \text{ and } 
X_i := \overline{B_i}\setminus \cup_{k=1}^{i-1} \overline{B_k}.
\] 
These are mutually disjoint Borel sets.
To see this, suppose  $x \in X_i \cap X_j$ for $i \neq j$. 
Without the loss of generality, we can assume that $i < j$.
From $x \in X_i$, we have that $x \in B_i$ but from $x \in X_j$, we have that $x \not \in \cup_{k=1}^{j-1} \overline{B_k}$. 
Since $i \leq j-1$, we conclude  $x \not \in X_i$.
This is a contradiction so $X_i \cap X_j = \emptyset$ for $i \neq j$.

We assert $A = \union_{i=1}^\infty X_i$.
The inclusion $\union_{i=1}^\infty X_i \subset A$ is easy since $X_i \subset A$.
For the converse, suppose  $x\in A$ and let $i_0 = \min\set{j: x \in \close{B_j}}$. 
Note that $i_0 \geq 1$ since $A = \cup_{i=1}^{\infty} \overline{B_i}$ implies that  there is an $i$ such that  $x\in \overline{B_i}$.
Now, $x \in \close{B_{i_0}}$ and by choice of $i_0$, $x \not \in \close{B_j}$ for all $j < i_0$.
In particular, $\close{X_k} \subset \close{B_k}$ for  all $k$, and hence, $x \not \in \close{X_j}$ for $j \leq i_0 - 1$.
That is, $x \in \close{B_{i_0}} \setminus \union_{k=1}^{i_0 -1} \close{B_k}$.
This shows the reverse inclusion $\union_{i=1}^\infty X_i \subset A$  and hence, $A = \union_{i=1}^\infty X_i \subset A$.

Since $X_i$ are mutually disjoint Borel subsets, we have that 
\begin{equation}
\label{Eq:MeasA}
\mu_{\mg}(A) 
=
\mu(\union_{i=1}^\infty X_i)
=
\sum_{k=1}^\infty \mu_{\mg}(X_i).
\end{equation}

Now, define the compact set $F_N= \cup_{i=1}^{N} \overline{B_i}$. Rewriting $F_N$ as a union of disjoint sets $\{X_i\}$'s. 
\begin{align*}
F_N
= 
\overline{B_{N}} \union (\union_{j=1}^{N-1} \overline{B_{j}})
= 
\overline{B_N} \cup F_{N-1}
= 
(\overline{B_N} \setminus F_{N-1}) \sqcup F_{N-1}.
\end{align*}
Iterating this procedure, we find 
\begin{align*}
F_N
&= (\overline{B_N} \setminus F_{N-1}) \sqcup  (\overline{B_{N-1}} \setminus F_{N-2}) \sqcup ... \sqcup (\overline{B_2} \setminus \overline{B_1}) \sqcup \overline{B_1}.
\end{align*}
Noting that $\close{B_N}  \setminus F_{N-1} = X_N$,
\begin{align*}
\mu_{\mg}(F_N)&= \mu_{\mg}(\overline{B_N} \setminus F_{N-1})+\mu_{\mg}(\overline{B_{N-1}} \setminus F_{N-2})+...+\mu_{\mg}(\overline{B_2} \setminus \overline{B_1})+\mu_{\mg}( \overline{B_1})\\
&= \sum_{i=1}^{N}\mu_{\mg}(\overline{B_i} \setminus F_{i-1})\\
&= \sum_{i=1}^{N} \mu_{\mg}(X_i).
\end{align*}
From equation~\eqref{Eq:MeasA}, we have that $\mu_{\mg}(F_N) \to \infty$ as $N \to \infty$.
Since $F_N$ compact and $F_N \subset A$, we have that
\[
\mu_{\mg}(F_N) \leq  \sup \set{ \mu_{\mg}(K): K \text{ is compact and } K \subset A} = \mu_{\mg_{*}}(A). 
\]
Since this is true for any $N$, we conclude that $\mu_{\mg_{*}}(A) = \infty = \mu_{\mg}(A)$. 
%Moreover, the outer regularity follows from Proposition 1 \cite{gruenhage1978inner}. <-- NO! 
\end{proof}

\subsection{The induced distance}
\label{S:Dist}

The goal of this section is to present a notion of distance for rough Riemannian metrics which generalise the corresponding notion for  smooth Riemannian metrics.
Before considering the general picture, we recall the smooth case.
First consider $\mh \in \Met_{\Ck{\infty}}(\cM)$.
If $\gamma: [a,b]\to \cM$ is a piecewise smooth curve, its $\mh$-length is given by
\begin{equation} 
\label{Eq:Len} 
\len_{\mh}(\gamma) := \int_{a}^{b} \modulus{\gamma'(t)}_{\mh(\gamma(t))}\ dt.
\end{equation} 
The induced Riemannian distance between $x, y \in \cM$ with respect to $\mh$ is then given by optimising over such curves: 
\begin{align} 
\label{Eq:RMet}
%\label{d_r}
\met_{\Riem, \mh}(x, y) := \inf\set{ \len_{\mh}(\gamma): \gamma:[a,b]\to\cM \text{ piecewise smooth }, \gamma(a) = x, \gamma(b) = y }.
\end{align}

A metric $\mg \in \Met(\cM)$ is only defined up to sets of measure zero. 
Therefore, a naïve generalisation for $\mg$-length by using \eqref{Eq:Len} is not possible as this expression is not well-defined. 
We are, therefore, required to approach the definition of distance through an alternative vantage point.  
Following \cite{Norris}, we define: 
\begin{equation}
\label{Eq:IMet}
\met_{\mg}(x,y)=\sup \{w(y)-w(x): w \in \Ck{0,1}(\cM),\quad |\nabla w|_{\mg^{-1}}\leq 1 \text{ a.e.}\}. 
\end{equation}
Firstly, we recall that this is a good notion of distance.
\begin{lem}
\label{Lem:GoodDist} 
On a smooth connected manifold $\cM$ with $\mg \in \Met(\cM)$, $\met_{\mg}: \cM\times \cM \to [0,\infty)$ is a metric.
\end{lem}
\begin{proof}
For $x,y \in \cM$, by the connectedness and local path-connectedness of $\cM$, we have a continuous path $\gamma \in \Ck{0}([0,1];\cM)$ with $\gamma(0) = x$ and $\gamma(1) = y$. 
By compactness of $\gamma([0,1])$, we can cover this by a finite number of locally comparable charts $(U_i,\psi_i)$. 
Taking points $z_i,z_{i+1} \in U_i$ along this path, we find 
\begin{align*} 
\met_{\mg}(x,y) 
\leq \sum_{i=1}^K \met_{\mg}^{U}(z_i, z_{i+1}) 
\leq \sum_{i=1}^K C_i |\psi_(z_i) - \psi(z_{i+1})| 
< \infty, 
\end{align*}  
since inside $U_i$, $C_i^{-1} \modulus{u}_{\psi^{\ast}_i\delta} \leq \modulus{u}_{\mg^{-1}} \leq C_i \modulus{u}_{\psi^{\ast}_i\delta}$.
This shows that $\met_{\mg}$ is finite. 

From definition, it is immediate that $\met_{\mg}(x,y) = \met_{\mg}(y,x)$ for all $x, y \in \cM$ and that $\met(x,x) = 0$.
Also, the triangle inequality follows from the definition, by writing $w(x) - w(y) = (w(x) - w(z)) + (w(z) - w(y))$ and using the subadditivity of the supremum.

It remains to assert that for $x \neq y$, $\met_{\mg}(x,y) > 0$. 
Fix such $x,y \in \cM$ and let $U$ be a pre-compact open set near $x$ with $y \neq U$ and $\phi:\Ck[cc]{\infty}(\cM)$ any function with $\phi(x) = 1$ and $\spt \phi \subset U$.
Clearly $0 < \norm{\nabla \phi}_{\Lp{\infty}(\cM,\mg)} < \infty$ and so $w = \phi/ \norm{\nabla \phi}_{\Lp{\infty}(\cM,\mg)} \in \Ck[cc]{\infty}(\cM) \subset \Ck{0,1}(\cM)$ satisfying $\modulus{\nabla w(y)} \leq 1$ for all $y \in \cM$.
Therefore, $\met_{\mg}(x,y) \geq w(x) - w(y) = w(x) > 0$.
\end{proof} 

Let us first show that this distance agrees with the usual Riemannian distance in the smooth setting.
It is a result which is alluded to \cite{Norris}, although its proof is terse.
In \cite{CeccoPalmieri}, a proof of this fact is also given but its proof relies on facts they establish on Lipschitz manifolds and is less straightforward.
Here, we give a straightforward and simple proof.
In what follows, $\nabla = \extd$ and $\mg^{-1}$ is the induced metric, from $\mg$, on $\cotanb\cM$.
\begin{prop}
\label{Prop:Distance} 
Let $\cM$ be a smooth connected manifold with smooth Riemannian metric $\mh$.
Then $\met_{\Riem,\mh} =  \met_{\mh}$.
\end{prop}
\begin{proof}
Firstly, we consider the case where $\cM=\mathbb{R}^n$, but still allow $\mh$ to be a smooth metric possibly different to the Euclidean one.
Fix $x, y \in \R^n$.

We show that $\met_{\mh}(x,y) \leq \met_{\Riem,\mh}(x,y)$.
Suppose that $\gamma :[a,b] \to  \mathbb{R}^n$ is a piecewise smooth curve and  $w:\mathbb{R}^n \to \mathbb{R}$ be a locally Lipschitz function.
Letting $\modulus{\xi}_{\mh^{-1}}= \mh^{ij} \xi_i \xi_j$ where $(\mh^{ij}) = (\mh_{ij})^{-1}$ inside a frame, calculate $\frac{d}{dt}(w\circ \gamma)$ at a differentiable point $t$ of $w \circ \gamma$: 
\begin{align*}
\frac{d}{dt}(w\circ \gamma) (t) &= \sum_{j} (\partial_{j}w)(\gamma(t)) \dot{\gamma}_j(t) = \nabla w(\gamma(t))^{T}\ \ident\ \dot{\gamma}(t) \\
&= \nabla w(\gamma(t))^{T}\mh^{-1}(\gamma(t))\mh(\gamma(t)) \dot{\gamma}(t) 
= [{\mh^{-1}}(\gamma(t))\nabla w(\gamma(t))]^{T}]\mh(\gamma(t)) \dot{\gamma}(t) \\
&= [\mh^{-1}(\gamma(t))\nabla w(\gamma(t))]\cdot [\mh(\gamma(t)) \dot{\gamma}(t)]
%the transpose has been replaced by the dot product
\leq | [\mh^{-1}(\gamma(t))\nabla w(\gamma(t))]\cdot [\mh(\gamma(t)) \dot{\gamma}(t)]| \\
&\leq | \mh^{-1}(\gamma(t))\nabla w(\gamma(t))| |\mh(\gamma(t)) \dot{\gamma}(t)|
% since |a \cdot b| \leq |a||b|
= |\nabla w|_{\mh^{-1}(\gamma(t))}|\dot{\gamma}(t)|_{\mh(\gamma(t))} 
\leq |\dot{\gamma}(t)|_{\mh(\gamma(t))}
\end{align*}
Further choosing $\gamma(0) = x$ and $\gamma(1) = y$,  and $\modulus{\nabla w}_{\mh^{-1}} \leq 1$ almost-everywhere, by the fundamental theorem of calculus, we have 
\begin{align}
\label{FTCd_g}
w(y)-w(x)
&= \int_{0}^{1}  \frac{d}{dt}(w\circ \gamma)(t) \, dt 
\leq \int_{0}^{1}|\dot{\gamma}(t)|_{\mh(t)} \ dt. 
\end{align}
Since \ref{FTCd_g} is true for all such curves joining $x$ and $y$, by taking an infimum over all such curves and using \eqref{Eq:RMet}, 
\[
w(y)-w(x)
\leq \met_{\Riem,\mh}(x,y).
\]
Since \eqref{FTCd_g} holds for any locally Lipschitz $w$ with $\modulus{\nabla w}_{\mh^{-1}} \leq 1$ almost-everywhere,  by taking a supremum over the left hand side and \eqref{Eq:IMet}, we obtain 
\begin{equation*}
\met_{\mh}(x,y) \leq  \met_{\Riem,\mh}(x,y).
\end{equation*} 

We now prove the reverse inequality. 
Assume again that $\gamma :[a,b] \to  \mathbb{R}^n$ is a piecewise smooth path.
Fix $x\in \mathbb{R}^n$ and define $\tilde{w}(y)=\met_{\Riem,\mh}(x,y)$.  
We will  prove that $\tilde{w}:\R^n \to \R$ is locally Lipschitz and $|\nabla \tilde{w}|_{\mh}\leq 1$ almost-everywhere.
It suffices to show that $\tilde{w}$ is Lipschitz on $r$-Euclidean balls $\Ball(r)$ as these are precompact sets exhausting $\R^n$.

Fix $r > 0$ and $\Ball(r)$.
Since $\mh$ is a smooth metric, we have a constant $C_r \geq 1$ such that $\mh \sim_{C_r} \delta_{\R^n}$
Let $x,z \in \Ball(r)$ and note
\begin{align*}
|\tilde{w}(y)-\tilde{w}(z)|&=|\met_{\Riem, \mh}(x,y)-\met_{\Riem,\mh}(x,z)|
\leq \met_{\Riem,\mh}(y,z) %since d(x,y)\leq d(x,z)+d(z,y) then d(x,y)-d(x,z)\leq d(y,z) 
\leq C_r |y-z|.
\end{align*}
Therefore, $\tilde{w} \in \Ck{0,1}(\R^n)$. 

Now we show that $|\nabla \tilde{w}|_{\mh^{-1}}\leq 1$ almost-everywhere.
By the smoothness of $\mh$, we have for locally Lipschitz $w$ and regular point $z \in \R^n$,
\[
|\nabla w(z)|_{\mh^{-1}(z)}
=
\limsup_{y\to z}\frac{|w(y)-w(z)|}{\met_{\Riem,\mh}(y,z)}.
\]
Since $|\tilde{w}(y)-\tilde{w}(z)| =|\met_{\Riem, \mh}(x,y)-\met_{\Riem,\mh}(x,z)| \leq \met_{\Riem,\mh}(y,z)$, we conclude $|\nabla \tilde{w}(z)|_{\mh^{-1}(z)}\leq 1$.
As this is true for all regular points $z \in \R^n$, we have that $|\nabla \tilde{w}|_{\mh^{-1}}\leq 1$ almost-everywhere. 
Now, from \eqref{Eq:IMet},  we have 
\begin{align*} 
\met_{\mh}(x,y) 
&= \sup\set{ w(x) - w(y): w \in \Ck{0,1}(\R^n),\quad  \modulus{\nabla w}_{\mh^{-1}} \leq 1 \text{ a.e }} \\
&\geq \modulus{\tilde{w}(x) - \tilde{w}(y)}
= \modulus{\met_{\Riem,\mh}(x,x) - \met_{\Riem,\mh}(x,y)}
= \met_{\Riem,\mh}(x,y).
\end{align*}
Together, we obtain that $\met_{\mh}(x,y) = \met_{\Riem,\mh}(x,y)$ for all $x,y \in \R^n$. 

Now let us reconsider the case of a general smooth connected manifold $\cM$.
First, let $(U,\psi)$ be a smooth chart for $\cM$.
Suppose that $f \in \Ck{0,1}(U)$. 
Clearly, $f \circ \psi^{-1} \in \Ck{0,1}(\psi(U))$.
Furthermore, noting that $\nabla^\cM = \extd^\cM$, for $x \in U \cap \Reg(f)$, 
\begin{align*}
\modulus{(\extd^{\cM}f)(x)}_{\mg^{-1}(x)}^2
&= 
\mg^{-1}(x) [ (\extd^{\cM})(x), (\extd^{\cM})(x) ] \\
&=
\mg^{-1}( \psi^{-1} \circ \psi(x)) [ \psi^\ast (\psi^{-1})^\ast(\extd^{\cM}f)(x), \psi^\ast (\psi^{-1})^\ast(\extd^{\cM}f)(x)] \\
&=
\psi_\ast \mg^{-1}(\psi(x)) [ (\psi^{-1})^\ast(\extd^{\cM}f)(\psi(x)),  (\psi^{-1})^\ast(\extd^{\cM}f)(\psi(x))] \\
&=
\psi_\ast \mg^{-1}(\psi(x)) [ (\extd^{\R^n} f \circ \psi^{-1})(\psi(x)),  ((\extd^{\cM} f \circ \psi^{-1})(\psi(x))] \\
&= 
\modulus{(\extd^{\R^n}f \circ \psi^{-1} )(\psi(x))}_{\psi_\ast \mg^{-1}(\psi(x))}^2.
\end{align*} testing an
Therefore,
$\modulus{(\extd^{\cM}f)}_{\mg^{-1}}^2 \leq 1$ almost-everywhere in $U$ if and only if  
$\modulus{(\extd^{\R^n}f \circ \psi^{-1} )(\psi(x))}_{\psi_\ast \mg^{-1}(\psi(x))} \leq 1$ almost-everywhere in $\psi(U)$.

Fix $x,y \in \cM$. 
We first prove that $\met_{\mh}(x,y) \leq \met_{\Riem, \mh}(x,y)$. 
To that end, let $\gamma:[a,b] \to \cM$ be a piecewise smooth curve. 
Let $\{(U_i,\psi_i)\}_{i=1}^K$ be a collection of charts such that $\psi_i: U_i \to \R^n$ and $\gamma([a,b]) \subset \union_{i=1}^K U_i$.
Furthermore, let $a = t_0 < t_1 < \dots < t_K = b$ be a sequence of points such that $x_{i-1}, x_{i} \in U_i$, where $x_i := \gamma(t_i)$.
Noting that $w: \cM \to \R$ locally Lipschitz implies that $w\rest{U_i} \circ \psi^{-1}:\R^n \to \R$ locally Lipschitz, we obtain that 
\begin{align*} 
\met_{\mh}(x_{i-1},x_{i}) 
	&= \sup\set{w(x_{i-1}) - w(x_{i}): w \in \Ck{0,1}(\cM),\quad \modulus{\nabla{w}}_{\mh^{-1}} \leq 1}\\
	&= \sup\left\{w\rest{U_i} \circ \psi_i^{-1} (\psi_i(x_{i-1})) -  w\rest{U_i} \circ \psi_i^{-1} (\psi_i(x_{i})): \right. \\
	&\qquad\qquad  \left. w \in \Ck{0,1}(\cM),\quad \modulus{\nabla{w\rest{U_i} \circ \psi_i^{-1}}}_{{\psi_i}_{\ast} \mh^{-1}} \leq 1\right\} \\
	&\leq \sup\set{\tilde{w}(\psi_i(x_{i-1})) - \tilde{w}(\psi_i(x_{i})): \tilde{w} \in \Ck{0,1}(\R^n),\quad \modulus{\nabla{\tilde{w}}}_{{\psi_i}_{\ast} \mh^{-1}} \leq 1} \\
	&= \met_{(\psi^{-1}_i)^{*} \mh}(\psi_i(x_{i-1}), \psi_i(x_{i})) \\ 
	&\leq  \len_{(\psi_i^{-1})^* \mh\rest{U_i}} (\psi\circ \gamma\rest{[t_{i-1},t_{i}]}) \\
	&=  \len_{\mh}(\gamma\rest{[t_{i-1},t_{i}]})
\end{align*}
since $(\psi^{-1}_i)^\ast \mh$ is a smooth metric on $\R^n$ and we invoked the $\R^n$ equality of metrics $\met^{\R^n}_{(\psi^{-1}_i)^\ast \mh} = \met^{\R^n}_{\Riem, (\psi^{-1}_i)^\ast \mh}$ on the third equality, and the fact that $\psi_i \circ \gamma\rest{[t_{i-1}, t_i]}: [t_{i-1},t_i] \to \R^n$ is a curve between $\psi_i(x_{i-1})$ and $\psi_i(x_{i})$ for the penultimate inequality, and where the final equality follows from a change of variable using $\psi_{i}$.
Hence, 
\[ 
\met_{\mh}(x,y) 
\leq \sum_{i=1}^K \met_{\mh}(x_{i-1},x_i)
\leq \sum_{i=1}^K \len_{\mh}(\gamma\rest{[t_{i-1},t_{i}]})
= \len_{\mh}(\gamma).
\]
Taking an infimum over then yields the desired inequality $\met_{\mh}(x,y)\leq \met_{\Riem,\mh}(x,y)$.

To prove the reverse inequality, we follow the construction for the $\R^n$ case. 
For $x,y \in \cM$ fixed, define $w(z) := \met_{\Riem,\mh}(x,z)$.
Through localisation and using the $\R^n$ argument, we can see that $w: \cM \to \R$ is locally Lipschitz.
Furthermore, $\modulus{w(z) - w(z')} = \modulus{\met_{\Riem,\mh}(x,z) - \met_{\Riem,\mh}(x,z')} \leq \met_{\Riem,\mh}(z,z')$.
Using the fact that for a smooth metric $\mh$ on $\cM$
\[ 
\modulus{\nabla w(z)}_{\mh^{-1}} = \limsup_{z' \to z} \frac{\modulus{w(z) - w(z')}}{\met_{\Riem, \mh}(z,z')},
\]
we obtain that $\modulus{\nabla w}_{\mh^{-1}} \leq 1$ almost-everywhere.
Therefore, 
\[ 
\met_{\Riem,\mh}(x,y) = \modulus{w(x) - w(y)} \leq \met_{\mh}(x,y).
\]
Together, this shows that $\met_{\Riem,\mh} = \met_{\mh}$.
\end{proof}

We return now to the general situation and fix $\mg \in \Met(\cM)$.
The length function for $\sigma \in \Ck{0}([a,b]; \cM)$ with respect to the induced metric $\met_{\mg}$ is given by 
\begin{equation}
\label{Eq:dLen}
\len_{\met_{\mg}}(\sigma) := \sup\limits_{0 = t_0 < \dots < t_M = 1} \sum_{k=1}^M \met_{\mg}(\gamma(t_{k-1}),\gamma(t_k)) 
\end{equation}
The induced \emph{intrinsic distance} from $\met_{\mg}$ is then 
\begin{equation}
\label{Eq:InMet}
\met_{\ic,\mg}(x,y) 
:= 
\inf\limits_{\substack{\sigma \in \Ck{0}([a,b],\cM)\\\gamma(0) = x,\ \gamma(1) = y}} \len_{\met_{\mg}}(\sigma)
\end{equation}
It is immediate that $\met_{\mg} \leq \met_{\ic, \mg}$.

For a smooth metric $\mh \in \Met_{\Ck{\infty}}(\cM)$, we have by Proposition~\ref{Prop:Distance} that $\met_{\mh} = \met_{\Riem,\mh}$. 
Noting that $\len_{\mh}(\sigma) = \len_{\met_{\mh}}(\sigma)$ and using Corollary~2.3 in \cite{Burtscher}, it readily follows that $\met_{\mh} = \met_{\ic,\mh}$.
That is, $\met_{\mh}$ is an intrinsic (length) metric.
This is an important property which holds even for the general case as demonstrated below.
\begin{prop}
\label{Prop:Intrinsic}
On a smooth connected manifold $\cM$ with $\mg \in \Met(\cM)$, $\met_{\mg}(x,y) = \met_{\ic, \mg}(x,y)$.
That is, $\met_{\mg}$ is an intrinsic metric and $(\cM,\met_{\mg})$ is a length space.
\end{prop}
\begin{proof}
Fix $\mg \in \Met(\cM)$ and let $\cU = \set{U \subset \cM:\ \psi:U \to \R^n, \text{ locally comparable}}$.
We automatically have that $\met_{\mg} \leq \met_{\ic,\mg}$ and so we prove the reverse inequality.
For that, first note the following distance from \cite{Norris}: 
\begin{equation} 
\label{Eq:NorDist} 
\met_{0}(x,y) := \inf \set{ \sum_{i=1}^K \met_{\mg}^{U_i}(z_{i-1},z_i): U_i \in \cU \text{ and } z_{i-1}, z_{i} \in \Lambda(U_i)},
\end{equation}
where $\Lambda(U_i)$ is a certain neighbourhood of the diagonal on $U_i \times U_i$ and $\met_{\mg}^{U_i}: U_i \times U_i \to \R$ is the metric on $U_i$ induced by $\mg\rest{U_i}$.

Furthermore, $U_i$ is an oriented manifold and moreover, there is a global comparability constant in $U_i$ so it satisfies the hypotheses of Theorem~3.4 in \cite{DP90} and we conclude that $\met_{\mg}^{U_i}$ an intrinsic distance and hence,
\[
\met_{\mg}^{U_i}(z,z') 
= 
\inf\limits_{\substack{\gamma \in \Ck{0}([a,b];U_i) \\ \gamma(a) = z, \gamma(b) = z'}} \sum_{a = t_0 < \dots < t_K = b} \met_{\mg}^{U_i}(\gamma(t_{i-1}), \gamma(t_i)).
\]
Since every $w \in \Ck{0,1}(\cM)$ restricts to $w \rest{U_i} \in \Ck{0,1}(U_i)$, it is easily seen that  $\met_{\mg}^{U_i}(w, w') \geq  \met_{\mg}(w,w')$ for all $w, w' \in U_i$ and therefore, each sum in the expression~\eqref{Eq:NorDist} satisfies
\begin{align*}
\sum_{i=1}^K \met_{\mg}^{U_i}(z_{i-1},z_i)
&\geq 
\sum_{i=1}^K
\inf\limits_{\substack{\gamma \in \Ck{0}([a_i,b_i];U_i) \\ \gamma(a) = z_{i-1}, \gamma(b) = z_i}} \sum_{a = t_0 < \dots < t_K = b} \met_{\mg}(\gamma(t_{i-1}), \gamma(t_i))\\
&=
\sum_{i=1}^K
\met_{\ic, \mg}(z_{i-1}, z_i)\\
&\geq 
\met_{\ic, \mg}(x,y)
\end{align*}
Combining this with \eqref{Eq:NorDist} and noting $\met_{0} = \met_{\mg}$, by Theorem~3.6 in \cite{Norris}, we obtain
\[ 
\met_{\mg}(x,y) = \met_{0}(x,y) \geq \met_{\ic,\mg}(x,y)
\]
Therefore, $\met_{\mg} = \met_{\ic,\mg}$ and is, therefore, an intrinsic metric.
\end{proof}

\begin{rem}
This result is mentioned in a number of places in the literature, including \cite{Norris}, \cite{DP90}, \cite{DP91} and \cite{DP95}. 
In these papers, even Lipschitz manifolds (with charts which are lipeomorphisms) are considered.
However, the proofs in \cite{DP90, DP91} seem to require extra hypotheses (such as orientability) and \cite{DP95} cites the earlier results.
In \cite{Norris} also mentions \cite{DP90,DP91,DP95} but does not provide a detailed argument. 
However, using the construction of $\met_{0}$ in \cite{Norris}, the results in \cite{DP91} can be utilised to dispense with extra assumptions.
\end{rem} 

\begin{rem}
Following the proofs of Proposition~\ref{Prop:Distance}, it is tempting to write $w(z) := \met_{\mg}(x,z)$, which is clearly locally Lipschitz,  and then attempt to show that $\modulus{\nabla w(z)}_{\mg}  \leq 1$ by alluding to
\begin{equation}
\label{Eq:Grad}  
\modulus{\nabla f(z)}_{\mg^{-1}}  = \lim_{z' \to z} \frac{f(z) - f(z')}{\met_{\mg}(z,z')}. 
\end{equation}
However, this equality is not in general true if the metric $\mg$ is not smooth!
A counterexample is constructed by Sturm in \cite{Sturm},  where it is shown that there exists a symmetric matrix $0 < A < 1$ (as bilinear forms) on $\R^n$ such that 
\[
\met_{\mg_{A}}(x,y) 
	= \sup \set{w(x) - w(y): w \in \Ck{0,1}(\R^n),\quad \modulus{A^{-1} \nabla w} \leq 1 \text{a.e}}
	= \modulus{ x - y}
	= \met_{\delta}(x,y).
\]
If \eqref{Eq:Grad} was true for all $\mg \in \Met(\cM)$, then one would obtain 
\begin{align*}  
\modulus{A^{-1} \nabla f(z)}_{\mg_{A}}
= 
\modulus{\nabla f(z)}_{\mg_{A}^{-1}} 
= \lim_{z' \to z} \frac{f(z) - f(z')}{\met_{\mg_{A}}(z,z')} 
= 
\lim_{z' \to z} \frac{f(z) - f(z')}{\modulus{z - z'}} 
= \modulus{\nabla f(z)}_{\delta}
\end{align*}
Since for any $u \in \cotanb_z\R^n$, we can choose $f\in \Ck{\infty}(\R^n)$ such that $\nabla f (z) = u$, this would force $A = 1$ (as a bilinear form) contradicting the counterexample in \cite{Sturm}.
\end{rem} 

Let us finally mention the connection between distance and the small-time asymptotics of heat kernels associated to the structures we have seen.
Fix $\mg \in \Met(\cM)$ and recall that  $\SobH{1}(\cM,\mg) \subset \Lp{2}(\cM,\mg)$ is the Sobolev space obtained by taking the closure of $\set{ f\in \Ck{\infty} \cap \Lp{2}(\cM): \nabla f \in \Lp{2}(\cM;\cotanb\cM,\mg)}$ with respect to the Sobolev norm $\norm{\cdot}_{\SobH{1}}^2 = \norm{\cdot}^2 + \norm{\nabla \cdot}^2$.
The space $\SobH[0]{1}(\cM),\mg)$ is then the closure of $\Ck[cc]{\infty}(\cM)$ in $\SobH{1}(\cM,\mg)$.
The associated Laplacian to each  space, by consider the energy $\cE: \SobH{1}(\cM,\mg) \times \SobH{1}(\cM,\mg) \to \C$ or $\cE_0: \SobH[0]{1}(\cM,\mg) \times \SobH[0]{1}(\cM,\mg) \to \C$ given by the map $(u,v) \mapsto \inprod{\nabla u, \nabla v}_{\Lp{2}(\mg)}$ yields a Laplacian $\Lap_{\mg}$ and $\Lap_{\mg,0}$.
Let $\hker_{\mg}: (0,\infty) \times \cM \times \cM \to \R$ and $\hker_{\mg,0}: (0,\infty) \times \cM \times \cM \to \R$ be the heat kernels associated to heat equations $\partial_t - \Lap_{\mg}$ and $\partial_t - \Lap_{\mg,0}$ respectively.

These Laplacians are natural - if the manifold is a compact manifold with boundary (where we would consider its interior to fit our boundaryless framework), $\Lap_{\mg}$ would be the Neumann Laplacian and $\Lap_{\mg,0}$ the Dirichlet Laplacian.
If $\mg$ is smooth and complete, then $\Lap_{\mg} = \Lap_{\mg,0}$.
Theorem~1.1 in \cite{Norris} asserts that 
\begin{align*} 
&t \log \hker_{\mg}(t,x,y) \to -\frac14 \met_{\mg}(x,y)^2 \\
&t \log \hker_{\mg,0}(t,x,y) \to -\frac14 \met_{\mg}(x,y)^2 
\end{align*}
uniformly on compact sets as $t \to 0$.
This is complemented by the fact that $\hker_{\mg}, \hker_{\mg,0} \in \Ck{0}( (0,\infty) \times \cM \times \cM)$, which can be obtained by results in \cite{ERS} and also \cite{BanBry}. 

This regularity and the so-called Varadhan asymptotics show that, even for our class of highly irregular geometries, diffusion sees the distance structure, although it is important to stress that by results in \cite{Sturm}, the distance is not determined by diffusion.
Moreover, both the Neumann problem and the Dirichlet problem see the same distance.

\subsection{The space of rough Riemannian metrics}
\label{S:RRMSpace}
We now consider the collection $\Met(\cM)$ of all rough Riemannian metrics  and equip it with a topology. 
This is motivated by the myriad of examples given in Section~\ref{S:Examples} which illustrates analytic and geometric properties are stable under $\Lp{\infty}$-perturbations.
This topology restricts to ``components'' which effectively capture this notion of $\Lp{\infty}$.

The candidate extended metric $\rmet^\cM: \cM \times \cM \to [0,\infty]$ is defined in Definition~\ref{DefRoughdistance}.
For two metrics $\mg, \mh$ with $\rmet^{\cM}(\mg,\mh) < \infty$, given $\epsilon > 0$, we have $C_{\epsilon} \geq 1$ with $\mg \sim_{C_{\epsilon}} \mh$ such that  $\log(C_{\epsilon}) \leq \rmet^{\cM}(\mg,\mh)  + \epsilon$. 
Therefore,
\[ 
\e^{-(\rmet^{\cM}(\mg,\mh)  + \epsilon)}\modulus{u}_{\mg} \leq \modulus{u}_{\mh} \leq \e^{(\rmet^{\cM}(\mg,\mh)  + \epsilon)}\modulus{u}_{\mg}
\] 
Letting $\epsilon \to 0$, we obtain that
\[ 
\e^{-\rmet^{\cM}(\mg,\mh)}\modulus{u}_{\mg} \leq \modulus{u}_{\mh} \leq \e^{\rmet^{\cM}(\mg,\mh)}\modulus{u}_{\mg}. 
\]
Therefore, the candidate distance provides the optimal bound in the characterising condition of closeness.
Below, we establish that this is indeed an extended metric. 

\begin{prop}
\label{Prop:RoughExtendedMetric}
The map $\rmet^\cM:\cM\times \cM \to [0,\infty] $ is an extended metric.
\end{prop}
\begin{proof}
It is immediate that $\rmet^{\cM}(\mg_1,\mg_2) = \rmet^{\cM}(\mg_2,\mg_1)$ for all $\mg_1,\mg_2 \in \Met(\cM)$ by definition.
Moreover, $\rmet^{\cM}(\mg,\mg) = 0$ since $(1 + \epsilon)^{-1} \modulus{u}_{\mg} \leq \modulus{u}_{\mg} \leq (1 + \epsilon)\modulus{u}_{\mg}$ and we can let $\epsilon \to 0$.

For three metrics $\mg_1,\mg_2,\mg_3 \in \Met(\cM)$ with $\rmet^{\cM}(\mg_i, \mg_j) < \infty$
Choose $C, C' \geq1$ such that $C^{-1} \modulus{u}_{\mg_1} \leq \modulus{u}_{\mg_2 } \leq C \modulus{u}_{\mg_1}$ and ${C'}^{-1} \modulus{u}_{\mg_2} \leq \modulus{u}_{\mg_3} \leq C' \modulus{u}_{\mg_2}$, from which its clear that 
\[ 
\frac{1}{CC'} \modulus{u}_{\mg_1} \leq \modulus{u}_{\mg_3} \leq CC' \modulus{u}_{\mg_1}.
\]
This implies $\rmet^{\cM}(\mg_1,\mg_3) \leq \log(CC') = \log(C) + \log(C')$.
Optimising over $C$ and $C'$ yields $\rmet^{\cM}(\mg_1,\mg_3)  \leq \rmet^{\cM}(\mg_1,\mg_2) + \rmet^{\cM}(\mg_2,\mg_3)$.
\end{proof}

Throughout, by the term \emph{the space of rough Riemannian metrics}, we will refer to the pair $(\Met(\cM),\rmet^\cM)$ or equivalently, the set $\Met(\cM)$ equipped with the topology given by $\rmet^{\cM}$. 
This will be the central object of our study. 

As described in the Section~\ref{S:Results}, we can consider $\Comp(\mg)$ of metrics $\mh$ with $\rmet^{\cM}(\mg,\mh) < \infty$, which is also realised by the equivalence relation $\mg \sim \mh$.
The space $(\Comp(\mg), \rmet^\cM\rest{\Comp(\mg)})$ is clearly a metric space by construction.

\subsection{The locally elliptic group and representations}
\label{S:RRMAction}

Let us now consider the set of locally elliptic transformations $\Ell(\cM)$ as defined in \eqref{Eq:Ell}.

\begin{prop}
\label{Prop:EllIsAGroup}
The set $(\Ell(\cM),\circ)$ is a group, where $\circ$ is pointwise composition on fibres. 
\end{prop}
\begin{proof}
Associativity of $\circ$ is clear and $\Id \in \Ell(\cM)$ which is the identity element.
Furthermore, for $B \in \Ell(\cM)$, clearly $B^{-1} \in \Ell(\cM)$ (in fact, by definition).

Let $B_1, B_2 \in \Ell(\cM)$.
Then, clearly for almost-every $x$, $B_1 B_2(x) \in \mathrm{GL}(\tanb_x\cM)$.
Therefore, $B_1 B_2$ is measurable since $B_1$ and $B_2$ are measurable.

Fix a compact subset $K \subset \cM$ and an auxiliary metric $\mg$.
Then, 
\[
\norm{B_1B_2}_{\Lp{\infty}(K,\mg)} 
\leq \norm{B_1}_{\Lp{\infty}(K,\mg)} \norm{B_2}_{\Lp{\infty}(K,\mg)}. 
\]
A similar argument shows that $(B_1 B_2)^{-1} = B_2^{-1} B_1^{-1} \in \Lp{\infty}(K; \End(\tanb\cM), \mg)$. 
Therefore, $(\Ell(\cM),\circ)$ is a group.
\end{proof}

We now prove Theorem~\ref{Thm:Stab}. 
For that, we first establish this lemma. 
\begin{lem}
\label{Lem:RoughMetEquiv}
We have that $\mg \in \Met(\cM)$ if and only if $\mg \in \Lp[loc]{\infty}(\cM; \Tensors[2,0]\cM)$ and $\mg^{-1} \in \Lp[loc]{\infty}(\cM; \Tensors[2,0]\cM)$, where $\mg^{-1} = (\mg_{ij})^{-1}$ in a local frame.
\end{lem}
\begin{proof}
If the condition is satisfied, then we can choose charts $(U,\psi)$ such that $(\close{U},\psi))$ sits inside a chart and $\close{U}$ is compact.
Then, the $\Lp{\infty}$ bound on $\mg$ and $\mg^{-1}$ by choosing the pullback metric $\psi^{\delta}$ with which to compute the $\Lp{\infty}(U)$ norm shows that $\mg \in \Met(\cM)$.

Conversely, given a compact set $K \subset \cM$, we have a finite number of local comparability charts $(U_i, \psi_i)$ such that $K \subset \union_{i=1}^N U_i$.
We patch together a metric induced by those local comparability charts and via a partition of unity $\set{\eta_i}_{i=1}^N$ and  patch together a smooth metric $\mh = \sum_{i=1}^N \eta_i \psi_i^\ast \delta_{\R^n}$.
The $\Lp{\infty}$-bound of $\mg$ with respect to this metric $\mh$ uniformly bounded above and below as we can take an appropriate maximum arising in the constants of the local comparability charts.
\end{proof}

\begin{proof}[Proof of Theorem~\ref{Thm:Stab}]
To prove Theorem~\ref{Thm:Stab}~\ref{Thm:Stab:i}, we fix $\mg \in \Met(\cM)$ and $B \in \Ell(\cM)$.
Then, for almost-every point on a compact subset $K$, 
\[ 
\mg_B[u,u] 
= \mg[Bu,Bu]
= \norm{\sqrt{B^\ast B}u}^2_{\mg}.
\]
By Lemma~\ref{Lem:RoughMetEquiv}, since $\mg \in \Met(\cM)$ and optimising over $u \in \tanb_x \cM$, we have $C_K \geq 1$ such that 
\[
\frac{1}{C_K} \leq \norm{\sqrt{B^\ast B}}_{\Lp{\infty}(K,\mg)} \leq C_K.
\] 
Therefore, $\sqrt{B^\ast B} \in \Ell(\cM)$ and also $B \in \Ell(\cM)$.

The conclusion Theorem~\ref{Thm:Stab}~\ref{Thm:Stab:ii} then follows immediately, from the definition of  $\action: \Ell(\cM) \times \Met(\cM) \to \Met(\cM)$.

To prove \ref{Thm:Stab:iii}, fix $\mg$ and $\mh$ and choose $x \in \Reg(\mg) \cap \Reg(\mh)$.
Since $\mg(x)$ and $\mh(x)$ are two inner products on the vector space $\tanb_x \cM$, it is a routine fact in linear algebra that there exists $B_x \in \Sym\Tensors[1,1]_x\cM$ such that 
\[ 
\mg(x)[u,v] = \mh(x)[B_x u,B_x v]
\] 
for all $u, v \in \tanb_x\cM$. 
Define $B(x) := B_x$ and since $\mg$ and $\mh$ are measurable, it follows that $B$ is also measurable.
By the same argument which shows Theorem~\ref{Thm:Stab}~\ref{Thm:Stab:i}, we deduce that $B \in \Ell(\cM)$.

We see that $\mg[u,v] = \mh[ B^\ast B u, v] = \mh[ \sqrt{B^\ast B} u, \sqrt{B^\ast B} v]$.
Uniqueness of a $\mh$-self-adjoint is readily argued. 
Let $B' := B^\ast B$ and $B'' \in \Ell(\cM)$ with $(B'')^{\ast, \mh} = B''$ such that $\mg(x)[u,v] = \mh(x)[B''(x)u,v]$, then $0 = \mh(x)[B'(x)-B''(x))u,v]$ and therefore $B'(x) = B''(x)$ for almost-every $x \in \cM$.
Clearly, $\sqrt{B'(x)} = \sqrt{B''(x)}$ which finishes the proof. 
\end{proof} 

\begin{rem}
Proposition~10 in \cite{BRough} expresses, in the language of this paper, that if $\mg$ and $\mh$ satisfy $\mh \in \Comp(\mg)$, we have that $B \in \Lp{\infty}(\cM; \Sym\End(\tanb\cM),\mg)$ and $B^{-1} \in \Lp{\infty}(\cM; \Sym\End(\tanb\cM),\mg)$. 
\end{rem}

\section{Smooth metrics}
\label{S: SmoothRRM} 

Although we allow for less than continuous objects in $\Met(\cM)$, by keeping $\cM$ smooth, we are afforded the possibility to study the structure of smooth and continuous objects in $\Met(\cM)$.
In Subsection~\ref{S:NonApprox}, we have already seen that there are rough Riemannian metrics which cannot be smoothly approximated in  $(\Met(\cM),\rmet^\cM)$.
Therefore, the goal of this section is to identify the closure of smooth metrics in $\Met(\cM)$ with respect to $\rmet^\cM$.
We will first consider the closure of smooth metrics and show that this lies in the subspace of continuous metrics. 
After that, we will consider the converse. 

\subsection{The closure of smooth metrics}

We begin with the following technical lemma.
\begin{lem}\label{bundlemetriclem}
Let $\cE \to \cM$ be a real vector bundle of rank $N$ and  $\widetilde{\mg} $ a smooth bundle metric on $\cE$. 
Suppose that  $\{\xi_n\} \subset \Ck{\infty} \cap \Lp{\infty} (\cM;\cE,\widetilde{\mg})$  and $\xi \in \Lp{\infty}(\cM;\cE,\widetilde{\mg})$ such that as $n \to \infty$, 
\[
\lVert\xi_n -\xi \rVert_{\Lp{\infty}(\cM;\cE,\widetilde{\mg})}=\esssup\limits_{x\in \cM}|\xi_n(x)-\xi(x)|_{\widetilde{\mg}(x)}.
\]
Then, $\xi \in \Ck{0}(\cM;\cE)$.
\end{lem}

\begin{proof}
Let $(U_{\alpha},\phi_{\alpha})$ be charts that cover $\cM$ which coincide with local trivialisations $(U_{\alpha},\Phi_{\alpha})$ for $\cE$.
Letting  $\pi:\cE \to \cM$ denote the projection map, we have that $\Phi_{\alpha}:\pi^{-1}(U_{\alpha}) \to U_{\alpha} \times \mathbb{R}^N$ is a local trivialisation. 
Moreover, by applying Gram-Schmidt, we can find an orthonormal frame $U_\alpha \ni x \mapsto \{t_1(x),...,t_N(x)\}$ with respect to $\widetilde{\mg}$. 
Clearly, $x \mapsto t_j(x)$ is a smooth map since $\widetilde{\mg}$ is smooth. 

Fix  $\{\xi_n\}$  and $\xi$ as in the hypothesis of the lemma. 
Inside $U_\alpha$, we write 
\[
\xi_n(x)= \sum_j C_n^j(x)t_j(x),\quad 
\xi(x) =\sum_jC^j(x)t_j(x),
\]
where $C_n^j \in \Ck{\infty}(U_\alpha, \R)$ and $C^j \in \Lp{0}(U_\alpha,\R)$. 
At $x\in U_\alpha$, 
%\begin{equation} 
%\label{bundle_metric_locally}
\begin{align*}
|\xi_n(x) -\xi(x)|^2_{\widetilde{g}(x)}&=\widetilde{g}(x)[\xi_n(x) -\xi(x),\xi_n(x) -\xi(x)]\\
&=\widetilde{g}(x)[\sum_j (C_n^j(x)-C^j(x))t_j(x),\sum_r (C_n^r(x)-C^r(x))t_r(x)]\\
&=\sum_{j,r} (C_n^j(x)-C^j(x))(C_n^r(x)-C^r(x))\widetilde{g}(x)[t_j(x),t_r(x)]\\
&=\sum_{j,r} (C_n^j(x)-C^j(x))(C_n^r(x)-C^r(x))\delta^j_r\\
&=\sum_{j} (C_n^j(x)-C^j(x))^2\\
&\simeq 
\cbrac{\sum_{j} \modulus{C_n^j(x)-C^j(x)}}^2,
\end{align*}
%\end{equation}
since any two norms are equivalent on a finite dimensional vector space with constant only dependent on dimension.
Moreover, since $\norm{\cdot}_{\Lp{\infty}(U_\alpha)} \leq \norm{\cdot}_{\Lp{\infty}(\cM)}$, we have that  for each $j$, 
the sequence  $\set{C_n^j}$ of functions  satisfy :w
i
\begin{equation*}\label{conv_fun_sec}
\lVert C^j_n-C^j \rVert_{\Lp{\infty}(U_\alpha)}=\esssup_{x\in U_\alpha}|C^j_n(x)-C^j(x)| \to 0,
\end{equation*}
as $n \to \infty$.
This means that $\set{C^j_n}$ is a Cauchy sequence in $\Lp{\infty}(U_\alpha)$ and by smoothness, it is Cauchy sequence in $\Ck{0}(U_\alpha)$. 
For any compact subset $K \subset U_\alpha$, we obtain $Y^{j}_K \in \Ck{0}(K)$ such that $C^j_n\rest{K} \to Y^{j}_K$ as $n \to \infty$.
Since $C^j_n\rest{K} \to C^j\rest{K}$ in $\Lp{\infty}(K)$, which is a metric space and hence Hausdorff, we have that $C^j\rest{K} = Y^j_{K}$ almost-everywhere in $K$.

For every $x \in U_\alpha$ we can choose $K_x$ compact with nonempty interior.
Since  $C^j\rest{K_x} \in \Ck{0}(K_x)$ we conclude that $^^j \in \Ck{0}(U_\alpha)$ and therefore, $\xi\rest{U_\alpha} \in \Ck{0}(U_\alpha; \cE)$.
As the $U_\alpha$ are a cover for $\cM$, we conclude $\xi \in \Ck{0}(\cM;\cE)$.
\end{proof}

With the aid of the previous Lemma~\ref{bundlemetriclem}, we show the closure of the set of smooth metrics contained in the set of continuous metric in the rough space.  

\begin{proof}[Proof of Theorem~\ref{Thm.limit.smooth.RRMs} ``$\subset$'']
Let $\mg \in \Met(\cM)$ and  $\{\mg_n\}\subseteq \Met_{\Ck{\infty}}(\cM)$ such that $\rmet^\cM(\mg_n, \mg) \to 0$.
We prove that $\mg \in \Met_{\Ck{0}}(\cM)$.

Fix $\epsilon > 0$.
By convergence of $\mg_n \to \mg$ in $\rmet^\cM$,  there is a $N = N(\epsilon) \in \mathbb{N}$ such that whenever $n \geq N$, we have $\rmet^\cM(\mg_n,\mg)<\epsilon$. 
Therefore, $\mg \sim_{\e^\epsilon}\mg_n$ which yields
\begin{equation} 
\label{Eq:epsilon}
\e^{-\epsilon}|u|_{\mg(x)}
\leq 
|u|_{g_n(x)}
\leq 
\e^{\epsilon}|u|_{\mg(x)}.
\end{equation} 
Since we can restrict to $\epsilon < 1$ and rearrange the sequence, without loss of generality, assume that $\rmet^\cM(\mg_1, \mg_n) < 1$ for all $n \in \Na$.
This in particular means that $\mg_n \in \Comp(\mg_1)$ and by Theorem~\ref{Thm:Stab}~\ref{Thm:Stab:iii},  there exists $\mg_1$-self-adjoint $B_n \in \Ell(\cM)$ such that 
\[
\mg_n(x)[u,v]
=
\mg_1(x)[B_n(x)u,B_n(x)u].
= 
\mg_1(x)[B_n^{2}(x)u,u]
\]
Similarly, there exists  $B \in \Ell(\cM)$ such that  $\mg(x)[u,v] = \mg_1(x)[B(x)u,B(x)v] = \mg_1(x)[B^2(x)u,v]$.

To prove the theorem, we show that that $\norm{B_n^2 - B^2}_{\Lp{\infty}(\Tensors[2,0]\cM,\widetilde{\mg}_1)} \to 0$ as $n \to \infty$, where $\widetilde{\mg}_1$ is the metric from $\mg_1$ on $\Tensors[2,0]\cM$.
Note that $\mg_1 \in \Met_{\Ck{\infty}}(\cM)$ by choice and therefore, $\widetilde{\mg}_1$ is a smooth bundle metric.
This allows us to invoke Lemma~\ref{bundlemetriclem} with a choice of $\cE = \Tensors[2,0]\cM$ to conclude that $B^2 \in \Ck{0}(\cM;\Tensors[2,0]\cM)$.
Therefore, $\mg$ is continuous since $\mg[u,v] = \mg_1[Bu,Bv]$.
 
First note that from \eqref{Eq:epsilon}, for a regular point $x \in \Reg(\mg)$,
\[
\e^{-\epsilon} |B(x)u|_{\mg_1(x)}
\leq 
|B_n(x)u|_{\mg_1(x)}
\leq 
\e^{\epsilon }|B(x)u|_{\mg_1(x)}.
\]
By transforming $u=B^{-1}(x)v$, we obtain 
\[
\e^{-\epsilon} |v|_{\mg_1(x)}
\leq 
|B_n(x)B^{-1}(x)v|_{\mg_1(x)}
\leq 
\e^{\epsilon }|v|_{\mg_1(x)}
\] 
for all $v \in \tanb_x\cM$.
Normalising by $\modulus{v}_{\mg_1(x)} \neq 0$ and taking a supremum, 
\begin{equation} 
\label{Eq:MatrixNorm} 
\e^{-\epsilon}
\leq 
\norm{B_n(x)B^{-1}(x)}_{\op,\mg_1(x)}
\leq 
\e^{\epsilon }
\end{equation}
and as $n \to \infty$, $\norm{B_n(x)B^{-1}(x)}_{\op,\mg_1(x)} \to 1$.  
Moreover, recall that the matrix norm is equal to the spectral norm
\begin{align*} 
\lVert B_n(x)B^{-1}(x)\rVert_{\op,\mg_1(x)}^2
&=
\lambda_{\max}((B_n(x)B^{-1}(x))^{\ast}(B_n(x)B^{-1}(x))) \\
&= 
\lambda_{\max}(B^{-1}(x)B_n^2(x)B^{-1}(x)),
\end{align*}
where by $\lambda_{\max}(X)$, we denote the largest eigenvalue for a positive symmetric matrix $X$. 
Using \eqref{Eq:MatrixNorm}, this leads to
\begin{equation*}
\e^{-2\epsilon} \leq \lambda_{\max}(B^{-1}(x)B_n^2(x)B^{-1}(x)) \leq \e^{2\epsilon},
\end{equation*}
Replacing  $B_n(x)B^{-1}(x)$ by $B(x)B^{-1}_n(x)$ and using the spectral mapping theorem to assert $\lambda_{\min}(X) = \lambda_{\max}(X^{-1})$ when $X$ is positive, symmetric and invertible, we get 
\begin{equation}
\begin{split}
\lVert B(x)B^{-1}_n(x)\rVert^2_{\op,\mg_1(x)}&= \lambda_{\max}((B(x)B^{-1}_n(x))^{\ast}(B(x)B^{-1}_n(x)))\\
& = \lambda_{\min}((B_n(x)B^{-1}(x))^{\ast}(B_n(x)B^{-1}(x)))  
\end{split}
\end{equation}
This leads to the two-sided bounds
\begin{equation}
\begin{split}
e^{-2\epsilon}\leq \lambda_{\min}(B^{-1}(x)B_n^2(x)B^{-1}(x))\leq \lambda(B^{-1}(x)B_n^2(x)B^{-1}(x)) \\
\qquad \leq \lambda_{\max}(B^{-1}(x)B_n^2(x)B^{-1}(x)) \leq e^{2\epsilon},
\end{split}   
\end{equation}
The matrix $B^{-1}(x)B_n^2(x)B^{-1}(x)$ is symmetric as it is the conjugation of the symmetric matrix $B_n^2(x)$  by the symmetric matrix $B^{-1}(x)$.
Therefore, $B^{-1}(x)B_n^2(x)B^{-1}(x)$ is diagonalisable as
 \[
D_n(x)=P_n(x)(B^{-1}(x)B_n^2(x)B^{-1}(x))P_n^{-1}(x),
\]
where $P_n(x)$ is orthogonal and where $D_n(x)$ is the diagonal matrix of the eigenvalues of the matrix $B^{-1}(x)B_n^2(x)B^{-1}(x)$. 

We compute the matrix norm of $B_n^2(x)-B^2(x)$:
\begin{align*}
\lVert B_n^2(x)-B^2(x)\rVert_{\op,\mg_1(x)}&=\lVert[B_n^2(x)B^{-1}(x)-B(x)]B(x)\rVert_{\op,\mg_1(x)}\\
&=\lVert B(x)[B^{-1}(x)B_n^2(x)B^{-1}(x)-\Id]B(x)\rVert_{\op,\mg_1(x)}\\
&\leq \lVert B(x)\rVert^{2}_{\op,\mg_1(x)}\ \lVert B^{-1}(x)B_n^2(x)B^{-1}(x)-\Id\rVert_{\op,\mg_1(x)}\\
&= \norm{B(x)}_{\op,\mg_1(x)}^2\ \norm{P_n^{-1}(x) D_n(x) P_n(x) -\Id}_{\op,\mg_1(x)} \\
&=\lVert B(x)\rVert^{2}_{\op,\mg_1(x)}\  \lVert P_n^{-1}(x) D_n(x) P_n(x) -P_n^{-1}(x)P_n(x)\rVert_{\op,\mg_1(x)}\\
&=\lVert B(x)\rVert^{2}_{\op,\mg_1(x)}\ \lVert P_n^{-1}(x)[D_n(x)-\Id]P_n(x)\rVert_{\op,\mg_1(x)}\\
&\leq \lVert B(x)\rVert^{2}_{\op,\mg_1(x)}\ \lVert P_n^{-1}(x)\rVert_{\op,\mg_1(x)} \lVert D_n(x)-\Id \rVert_{\op,\mg_1(x)} \lVert P_n(x)\rVert_{\op,\mg_1(x)} \\ 
&= \lVert B(x)\rVert^{2}_{\op,\mg_1(x)}\  \lVert D_n(x)-\Id \rVert_{\op,\mg_1(x)} \\
&\leq  \e^2  \lVert D_n(x)-\Id \rVert_{\op,\mg_1(x)}
\end{align*}
since $\lVert P_n(x)\rVert_{\op,\mg_1(x)}=\lVert P_n^{-1}(x)\rVert_{\op,\mg_1(x)}=1$ and $\lVert B(x)\rVert _{\op,\mg_1(x)}\leq \e$ since $\rmet^\cM(\mg_1, \mg) \leq 1$.

We bound $\norm{D_n(x)-\Id}_{\op,\mg_1(x)}$ from above, independent of $x$.
For that, note since $D_n(x)$ is the eigenvalue matrix of $B^{-1}(x)B_n^2(x)B^{-1}(x)$,
\begin{equation}
\label{Eq:D_nEig} 
\e^{-2\epsilon}
\leq 
\lambda_i(D_n(x))
\leq 
\e^{2\epsilon},
\end{equation}
for $i\in \{1,...,(\dim\cM)\}$.
Furthermore, since $D_n(x) -\Id$ is symmetric and positive by \eqref{Eq:D_nEig}, we have that  $\norm{D_n(x) -\Id}_{\op,\mg_1(x)} = \lambda_{\max}(D_n(x) -\Id)$.
Clearly, this is a diagonal matrix $D_n(x) - \Id = \diag( D_n(x))_{1} - 1, \dots, (D_n(x)_{\dim\cM} - 1)$.
Thus, 
\[ 
\lambda_i(D_n(x)-\Id) 
= \lambda_i(D_n(x)) - 1
\leq   \e^{2\epsilon} - 1
\]
and therefore, $\rVert B^2_n(x)-B^2(x)\lVert_{\op,\mg_1(x)} \leq \e^{2}(\e^{2\epsilon} - 1)$.

Our computations so far have been in the operator norm. 
However, the metric $\mgt_1$ of $u = \sum_{ij} u_{ij} e^{i} \tensor e^{j} \in  \Tensors[2,0]_x\cM$, in an orthonormal frame $\set{e_j}$ at $x$ is given by 
\[ 
\modulus{u}_{\mgt_1(x)}^2 = \sum_{ij} \modulus{u_{ij}}^2.
\]
It is easy to see that this is the Frobenius norm and that 
\[ 
 \norm{\cdot}_{\op,\mg_1(x)} \leq \norm{\cdot}_{\mgt_1(x)} \leq (\dim \cM) \norm{\cdot}_{\op,\mg_1(x)}. 
\]

Therefore, 
\[ 
\modulus{B^2_n(x) - B^2(x)}_{\mgt_1(x)} \leq (\dim \cM) \rVert B^2_n(x)-B^2(x)\lVert_{\op,\mg_1(x)} \leq (\dim \cM) \e^{2}(\e^{2\epsilon} - 1).
\] 
The right hand side is independent of $x \in \Reg(\mg)$ and since $\cM \setminus \Reg(\mg)$ is a set of measure zero, we obtain that $\norm{B^2_n - B^2}_{\Lp{\infty}(\cM;\Tensors[2,0]\cM,\widetilde{\mg}_1)} \to 0$ as $n \to \infty$.
By application of  Lemma \ref{bundlemetriclem} on choosing $\cE = \Tensors[2,0]\cM$ and the smooth bundle metric $\widetilde{\mg}_1$ induced by $\mg_1$, we deduce that  $B^2 \in \Ck{0}(\cM;\Tensors[2,0]\cM)$. 
Since $\mg(x)[u,v] = \mg_1(x) [B(x)u,B(x)v] = \mg_1(x)[B(x)^2u,v]$, we conclude that $\mg \in \Met_{\Ck{0}}(\cM)$.
\end{proof}

\subsection{Smooth approximations of continuous metrics} 

In this subsection, we assert that every $\mg \in \Met_{\Ck{0}}(\cM)$ can be approximated by a sequence in $\Met_{\Ck{\infty}}(\cM)$.
This, together with the statement we obtained in the previous section, completes the proof of Theorem~\ref{Thm.limit.smooth.RRMs}.

First, we note the useful technical lemma which we utilise in the proof. 
\begin{lem}
\label{Lem:AlternativeClose} 
Suppose $\mg_t \in \Met(\cM)$ for $t \in (0, \delta)$ for some $\delta > 0$ and  $f_1, f_2: (0,\delta) \to (0,\infty)$ are functions satisfying: 
\begin{enumerate}[label=(\roman*)] 
\item \label{Lem:Alt:1} 
$f_1(t) \leq 1 \leq f_2(t)$ for all $t \in (0, \delta)$, 
\item \label{Lem:Alt:2} 
$t\mapsto f_1(t)$ is  decreasing as  $t$ increases,  
\item \label{Lem:Alt:3}
$t \mapsto f_2(t)$  increases as $t$ increases, and
\item \label{Lem:Alt:4} 
$f_1(t) \to 1$ and $f_2(t) \to 1$ as $t \to 0$. 
\end{enumerate}
If there exists $\mg_0 \in \Met(\cM)$ such that 
\[ 
f_1(t) \modulus{u}_{\mg_t(x)} \leq \modulus{u}_{\mg_0(x)} \leq f_2(t) \modulus{u}_{\mg_t(x)}
\] 
for all $u \in \tanb_x\cM$ for almost-every $x$, then $\rmet^\cM(\mg_t,\mg_0) \to 0$. 
\end{lem}
\begin{proof}
Note that by the condition~\ref{Lem:Alt:1}, we have that $1/f_1(t) \geq 1$ and $1/f_2(t) \leq 1$.
By~\ref{Lem:Alt:2}, $t \mapsto 1/f_1(t)$ is increasing in $t$ and by~\ref{Lem:Alt:3}, $t \mapsto 1/f_2(t)$ is decreasing in $t$.
Define $C:(0,\delta) \to (0, \infty)$ by
\[
C(t) := \frac{f_2(t)}{f_1(t)}.
\]
Clearly, $C(t) \geq 1$ and since it is the multiple of two increasing functions, it is also increasing.
\begin{align*} 
C(t)^{-1} \modulus{u}_{\mg_t(x)} =  \frac{f_1(t)}{f_2(t)} \modulus{u}_{\mg_t(x)} 
\leq f_1(t) \modulus{u}_{\mg_t(x)} 
&\leq \modulus{u}_{\mg_0(x)}  \\ 
&\leq f_2(t) \modulus{u}_{\mg_t(x)}
\leq
 \frac{f_2(t)}{f_1(t)} \modulus{u}_{\mg_t(x)} 
=
C(t) \modulus{u}_{\mg_t(x)}. 
\end{align*} 
Therefore,  $\rmet^\cM(\mg_t,\mg_0) \leq \log(C(t)) \to 0$ as $t \to 0$. 
\end{proof}

With this, we prove the following. 

\begin{proof}[Proof of Theorem~\ref{Thm.limit.smooth.RRMs} ``$\supset$'']
Let  $\mg \in \Met_{\Ck{0}}(\cM)$ and fix $x \in \cM$. 
By the application of the Gram-Schmidt process, we find an orthonormal basis $\{s_i\} \subseteq \tanb_x\cM$ such that 
\begin{equation*}
\mg(x)[s_i,s_j]=\delta_{ij}.
\end{equation*}
Let $(U,\psi)$ be a chart around $x$ and define a metric $\widetilde{\mg} = (\psi^{-1})^\ast \mg$ on $\psi(U) \subset \R^n$. 
It is worth to note that since $\mg \in \Met_{\Ck{0}}(\cM)$, $\widetilde{\mg} \in \Met_{\Ck{0}}( \psi(U))$. 
By pushing forward $\set{s_i} \subset \tanb_x \cM$, we obtain $\widetilde{s}_i = \psi_{\ast} s_i$ at $\tanb_{\psi(x)} \psi(U)$.
From construction,
\begin{align*}
\widetilde{\mg}(\psi(x))[\widetilde{s_i},\widetilde{s_j}]
&= (\psi^{-1})^*\mg(\psi(x))[(\psi)_*s_i,(\psi)_*s_j]\\
&=\mg(x)[(\psi^{-1})_*(\psi)_*s_i,(\psi^{-1})_*(\psi)_*s_j]\\
&=\mg(x)[s_i,s_j]\\
&=\delta_{ij}.
\end{align*}

Note that $\set{\widetilde{s}_i}$ is a basis at the point $\psi(x)$. 
We now extend this to a smooth basis $\set{\hat{s}_i}$ in all of $\tanb\psi(U)$ by parallel transporting $\set{\widetilde{s}_i}$ inside $\psi(U)$.
To explicitly construct this, fix the standard orthonormal basis $\set{e_k}$ of $\R^n$. 
Clearly, $\set{e_k} \subset \tanb\psi(U) \subset \tanb\R^n \cong \R^n$ defines a frame inside $\tanb\psi(U)$.
At $\psi(x) \in \psi(U)$, write $\widetilde{s}_i = \sum_{j=1}^k a_i^j e_j$.
We define the extended frame $\set{\hat{s}_i} \subset \tanb\psi(U)$ simply by
\[ 
\hat{s}_i(y)  := \sum_{j=1}^k a_i^j e_j. 
\]
This is clearly an extension as $\hat{s}_i(\psi(x)) = \widetilde{s}_i$ and $y \mapsto \hat{s}_i(y)$ is smooth. 

Now we construct a smooth metric $\hat{\mg}$ on $\psi(U)$ simply by asking for it to be orthonormal with respect to the frame $\hat{s}_i$.
That is, for $y \in \psi(U)$, define 
\begin{align*}
\hat{\mg}(y)[\hat{s}_i(y),\hat{s}_j(y)]
:=
\delta_{ij}.
\end{align*}
Since $y \mapsto \hat{s}_i(y)$ is smooth, so is $\hat{\mg}$.
Define the metric 
\[
\mh^U := \psi^*\hat{\mg},
\]
inside $U$. 
Clearly, since $\hat{\mg}$ is smooth, $\psi$ is smooth, so is $\mh^U$.

Let $s_k^U := \psi^{-1}_\ast \hat{s}_k$ and write $\mg_{ij}(x) := \mg(x)[ s_i^U(x), s_j^U(x)]$ and $\mh^U_{ij}(x) := \mh^U(x)[ s_i^U(x), s_j^U(x)]$.
By construction 
\begin{equation}
\label{Eq:Coeff}
\mh_{ij}(x) = \mh^U(x)[e_i, e_j]  = \hat{\mg}(\psi(x))[\hat{s}_i(\psi(x)), \hat{s}_j(\psi(x))] = \delta_{ij} = \mg(x)[e_i, e_j] = \mg_{ij}(x).
\end{equation}
Fixing $\epsilon > 0$, by continuity of $\mg$ and of $\mh^U$, we find a $\delta_1(x) > 0 $ such that for $z \in U_{\delta_1(x)}^x := \psi^{-1}(\Ball^{\R^n}(\psi(x),\delta_1(x)))$, 
\[
|\mg_{ij}(x) -\mg_{ij}(z)|
<\frac{\epsilon}{2}
\quad \text{ and }  \quad
|\mh_{ij}(x) -\mh_{ij}^{U} (z)| = 0 < \frac{\epsilon}{2}
\]
for all $z\in  U_{\delta_1(x)}^{x}$. 
For such $z\in  U_{\delta_1(x)}^{x}$, 
\begin{align*}
|\mg_{ij}(z)- \mh_{ij}^{U}(z)|
&=
|\mg_{ij}(z)- \mg_{ij}(x)+ \mg_{ij}(x)-\mh_{ij}^{U}(z)|\\
&
\leq 
|\mg_{ij}(z)- \mg_{ij}(x)|+|\mh_{ij}(x)-\mh_{ij}^{U}(z)|\\
&<
\frac{\epsilon}{2}+\frac{\epsilon}{2}
=
\epsilon,
\end{align*}
where the second line used $\mh_{ij}^U(x) = \mg_{ij}(x)$ from \eqref{Eq:Coeff}. 

Clearly $\set{U_{\delta_{1}(x)}^x}$ is a cover for $\cM$ and so fix a countable subcover $\set{U_i}$.
Let $\set{\eta^i}$ be a smooth partition of unity subordinate to $\set{U_i}$. 
Define  $\mg^{\epsilon}$ by
\[
\mg^{\epsilon}(x)
:=
\sum_{i=1}^{\infty} \eta^i(x) \mh^{U^i_\epsilon}(x),
\]
where $\mh^{U^i_\epsilon}$ is the metric we built for the larger set $U^i$ before. 
Clearly, $\mh^\epsilon \in \Met_{\Ck{\infty}}(\cM)$.

We claim that  $\rmet^\cM(\mh^\epsilon, \mg) \to 0$ as $\epsilon \to 0$. 
Fix $\epsilon < 1$ and fix $z \in \cM$. 
Suppose that $z \in U_{a}$ for some $a$.
Then, for $u \in \tanb_z \cM$ with respect to the pullback orthonormal basis of $\mh^{U^a_\epsilon}$,
\begin{align*} 
\modulus{u}^2_{\mh^{U^a_\epsilon}(z)}
&=
\sum_{ij} \mh^{U^a_\epsilon}_{ij}(z) u_i u_j \\
&= 
\sum_{ij} [\mh^{U^a_\epsilon}_{ij}(z) - \mg_{ij}(z)] u_i u_j  + \sum_{ij} \mg_{ij}(z) u_i u_j \\
&\leq 
\epsilon \sum_{ij} u_i u_j +  \modulus{u}_{\mg(z)}^2 \\
&\leq
 \epsilon \cbrac{\sum_{i} u_i^2}^{\frac12} \cbrac{\sum_{j} u_j^2}^{\frac12} + \modulus{u}_{\mg(z)}^2 \\ 
&= 
\epsilon \modulus{u}^2_{\mh^{U^a_\epsilon}(z)} + \modulus{u}_{\mg(z)}^2,
\end{align*}
since $\sum_{i} u_i^2 = \modulus{u}^2_{\mh^{U^a_\epsilon}(z)}$.
Therefore,  $(1 - \epsilon) \modulus{u}^2_{\mh^{U^a_\epsilon}(z)}\leq \modulus{u}_{\mg(z)}^2$ and
\[
(1 - \epsilon) \modulus{u}_{\mg^{\epsilon}(z)}^2
= \sum_{a} \eta^a(z) (1 - \epsilon) \modulus{u}^2_{\mh^{U^a_\epsilon}(z)}
\leq  
\modulus{u}_{\mg(z)}^2.
\]
Also, for the same $z \in U_a$,
\begin{align*}
\modulus{u}_{\mg(z)}^2 
&= \sum_{ij} \mg_{ij}(z) u_i u_j \\
&\leq \sum_{ij} \modulus{\mg_{ij}(z) - \mh_{ij}^{U^a_\epsilon}}(z) u_i u_j + \mh_{ij}^{U^a_\epsilon}(z) \\
&\leq \epsilon \sum_{ij} u_i u_j +  \modulus{u}^2_{\mh^{U^a_\epsilon}(z)} \\ 
&\leq (1 + \epsilon) \modulus{u}^2_{\mh^{U^a_\epsilon}(z)}
\end{align*} 
Then, 
\[ 
\modulus{u}_{\mg(z)}^2 
= \sum_{a} \eta^a(z) \modulus{u}_{\mg(z)}^2
\leq (1 + \epsilon) \sum_{a} \eta^a(z) \modulus{u}^2_{\mh^{U^a_\epsilon}(z)}
= \modulus{u}_{\mg^{\epsilon}(z)}^2. 
\]
Together, this proves that
\[ 
(1 - \epsilon) \modulus{u}_{\mg^{\epsilon}(z)}^2 
\leq  
\modulus{u}_{\mg(z)}^2 \leq 
(1 + \epsilon)
 \modulus{u}_{\mg^{\epsilon}(z)}^2. 
\]
Defining $f_1, f_2: (0,1) \to (0,\infty)$ by $f_1(t) = (1 - t)$ and $f_2(t) = (1 + t)$  and invoking Lemma~\ref{Lem:AlternativeClose}, we deduce that $\rmet^\cM(\mg^\epsilon, \mg) \to 0$ as $\epsilon \to 0$.  
\end{proof}

\section{Completeness}
\label{S:Completeness}
We have so far identified an important closed subset, the subset of continuous metrics $\Met_{\Ck{0}}(\cM)$ of $(\Met(\cM),\rmet^\cM)$.
Now, we proceed to show that $(\Met(\cM),\rmet^\cM)$ is a complete metric space. 

We do so by reducing the question to the completeness of $\Lp{\infty}(\cM;\Tensors[r,s]\cM,\mgt)$ for $\mg \in \Met(\cM)$ where $\mgt$ is the induced metric on $\Tensors[r,s]\cM$ from $\mg$.First, we verify this holds for $\mg \in \Met_{\Ck{\infty}}(\cM)$.
\begin{lem}
\label{vbundle-smooth-complete}
Let $\mg \in \Met_{\Ck{\infty}}(\cM)$ and let  $\mgt$ be the smooth bundle  metric induced by $\mg$  on $\Tensors[r,s]\cM$. 
Then the space $\Lp{\infty}(\cM;\Tensors[r,s]\cM,\mgt)$ is a Banach space. 
\end{lem}
\begin{proof}
Fix an open cover of chart$(\psi_i, U_i)$ for $\cM$.
These are also local trivialisations for $\Tensors[r,s]\cM$.
Let $\set{\xi_n}$ be a Cauchy sequence in $\Lp{\infty}(\cM;\Tensors[r,s]\cM,\mgt)$. 
Inside each $U_j$, noting $\R^M = (\psi_i^{-1})^\ast \Tensors[r,s]\cM$ where $M = \rank \Tensors[r,s]\cM$,
\begin{align*}
\norm{(\psi_i^{-1})^\ast \xi_n\rest{U_i} - (\psi_i^{-1})^\ast \xi_m\rest{U_i}}_{\Lp{\infty}(\psi_i(U_i); \R^M, (\psi_i^{-1})^\ast\mgt)}
&=  
\norm{\xi_n - \xi_m}_{\Lp{\infty}(U_i; \Tensors[r,s]\cM, \mgt)}  \\ 
&\leq
\norm{\xi_n - \xi_m}_{\Lp{\infty}(\cM; \Tensors[r,s]\cM, \mgt)}. 
\end{align*}
Therefore, $\set{(\psi_i^{-1})^\ast \xi_n\rest{U_i}}$ is a Cauchy sequence in  $\Lp{\infty}(\psi_i(U_i); \R^M, (\psi_i^{-1})^\ast\mgt)$.
Since $(\R^M, (\psi_i^{-1})^\ast\mgt)) \cong (\R^M, \delta_{\R^M})$ by choosing an orthonormal basis for $(\psi_i^{-1})^\ast\mgt$, we have $\tilde{\xi}^i \in \Lp{\infty}(\psi_i(U_i); \R^M, (\psi_i^{-1})^\ast\mgt)$ such that $(\psi_i^{-1})^\ast \xi_n\rest{U_i} \to \tilde{\xi}_i$.
Clearly, $\xi^i = \psi_i^\ast \tilde{\xi}_i$ satisfies $\xi_n\rest{U_i} \to \xi^i$ inside $\Lp{\infty}(U_i; \Tensors[r,s]\cM, \mgt)$.

We claim that $\xi^i = \xi^j$ on $U_i \cap U_j \neq \emptyset$. 
For that, repeat the same procedure on $U_j \cap U_j$ to find $\xi_n\rest{U_i \cap U_j} \to \xi^{ij}$. 
Since $\norm{\cdot }_{\Lp{\infty}(U_i \cap U_j; \Tensors[r,s]\cM, \mgt)} \leq \norm{\cdot}_{\Lp{\infty}(U_i; \Tensors[r,s]\cM, \mgt)}$ and $\Lp{\infty}(U_i \cap U_j; \Tensors[r,s]\cM, \mgt)$ is Hausdorff, we have that $\xi^i = \xi^{ij} = \xi^j$.
Therefore, 
\[ 
\xi(x) := \xi^i (x)\quad x \in U_i
\]
is well-defined. 
Moreover, $\xi\rest{U_i} = \xi^i$, measurable in $U_i$, and therefore, $\xi$ is itself measurable.

Now we show that $\xi \in\Lp{\infty}(\cM; \Tensors[r,s]\cM, \mgt)$.
To that end, note that $\xi(x) = \lim_{n \to \infty} \xi_n(x)$ almost-everywhere pointwise.
Since $\set{\xi_n}$ is Cauchy in $\Lp{\infty}(\cM; \Tensors[r,s]\cM, \mgt)$, there exists $N \in \Na$ such that $\norm{\xi_n - \xi_m}_{\Lp{\infty}(\cM)} \leq 1$ for all $m,n \geq N$.
In particular, this means that for almost-every $x \in \cM$, 
\begin{equation} 
\label{Eq:Linfbd} 
\modulus{\xi_n(x)} \leq 1 + \modulus{\xi_N(x)} \leq 1 + \norm{\xi_N}_{\Lp{\infty}(\cM;\Tensors[r,s]\cM,\mgt)}. 
\end{equation} 
Let $x \in \cM$ be a regular point where $\modulus{\xi(x) - \xi_i(x)} \to 0$ as $i \to \infty$.
This means that for each $\epsilon > 0$, there exists $N'(\epsilon,x)$ such that $i \geq N'(\epsilon,x)$ implies $\modulus{\xi(x) - \xi_i(x)} < \epsilon$. 
The expression \eqref{Eq:Linfbd} is valid for all $n \geq N$ and hence, certainly for all $n \geq \max\set{N'(\epsilon,x), N(\epsilon)}$.
That means we can take the limit as $n \to \infty$ on the right hand side to obtain 
\begin{equation*} 
\modulus{\xi(x)} \leq 1 + \modulus{\xi_N(x)} \leq 1 + \norm{\xi_N}_{\Lp{\infty}(\cM; \Tensors[r,s]\cM,\mgt)}. 
\end{equation*}
This shows that $\xi \in \Lp{\infty}(\cM;\Tensors[r,s]\cM,\mgt)$. 

It remains to show that $\xi_n \to \xi$ uniformly almost-everywhere.
Fix $\epsilon >0$. 
It suffices to show that for $\epsilon > 0$, there exists $N = N(\epsilon) > 0$ such that for all $n \geq N$, $\modulus{\xi(x) - \xi_n(x)}_{\mgt(x)} < \epsilon$ for $x$-a.e.
Let $N = N(\epsilon) > 0$ be given from the Cauchy condition where for $m, n \geq N$, we have that $\norm{\xi_m - \xi_n}_{\Lp{\infty}(\cM;\Tensors[r,s]\cM,\mgt)} < \epsilon/2$.
At a regular point $x$ where we have pointwise convergence $\xi_m(x) \to \xi(x)$, choose $N'(\epsilon,x)$ such that for $m \geq N'$, we have $\modulus{\xi(x) - \xi_m(x)} < \epsilon/2$.
Letting $N'' := \max\set{N'(\epsilon,x), N(\epsilon)}$, for all $n \geq N$,
\begin{align*}
\modulus{\xi(x) - \xi_n(x)} 
\leq
\modulus{\xi(x) - \xi_{N''}(x)}  + \modulus{\xi_{N''}(x) - \xi_n(x)}  
< 
\frac\epsilon2 + \frac\epsilon2
 = \epsilon.
\end{align*}
Since $x \in \cM$ was chosen to be a regular point and the complement of such points is a set of measure zero, the proof is complete.
\end{proof}

We now upgrade this result so that it holds for general rough Riemannian metrics $\mg \in \Met(\cM)$.
\begin{lem}
\label{vbundle-rough-complete}
Let $\mg \in \Met(\cM)$ and $\mgt$ the induced metric from $\mg$ on $\Tensors[r,s]\cM$.
Then, $\Lp{\infty}(\cM;\Tensors[r,s]\cM,\mgt)$ is a Banach space.
\end{lem}
\begin{proof}
Suppose that $\{\xi_n\} \subset \Lp{\infty}(\cM;\Tensors[r,s]\cM,\mgt)$ is a Cauchy sequence of sections. 
Let $K_1 \subseteq K_2 \subseteq ...\subseteq K_k\subseteq ... $ be an exhaustion of $\cM$ by compact sets $K_j$ with nonempty interior.
Clearly,
\begin{equation}
\label{vbundle-rough-complete-Cau-loc}
\rVert \xi_n -\xi_m \lVert_{\Lp{\infty}(K_j,\Tensors[r,s]\cM,\mgt)} \leq \rVert \xi_n-\xi_m\lVert_{\Lp{\infty}(\cM;\Tensors[r,s]\cM,\mgt)} \to 0.    
\end{equation}
as $m, n \to \infty$ and therefore, we see that $\set{\xi_n\rest{K_j}}$ is a Cauchy sequence in $\Lp{\infty}(K_j;\Tensors[r,s]\cM,\mgt)$. 

Since $K_j$ is compact, there exists $\mh_j^\infty \in \Met_{\Ck{\infty}}(K_j)$ and from  Theorem \ref{Thm:Stab} a $\mh_j^\infty$-self-adjoint  $B^2_j \in \Ell(K_j)$ such that 
\[
\mg(x)[u,v]=\mh_j^{\infty}(x)[B^2_j(x)u,v]=h_j^{\infty}(x)[B_j(x)u,B_j(x)v].
\]
From  Proposition~10 in \cite{BRough}, we see that  $B_j \in \Lp{\infty}(U_j;\End(\tanb\cM),\mg)$ and $B_j^{-1} \in \Lp{\infty}(U_j;\End(\tanb\cM),\mg)$. 
Denoting the induced $\tilde{B}_j \in \Lp{0}(U_j;\End(\Tensors[r,s]\cM))$ satisfying $\mgt[u,v] = \tilde{\mh}_j^\infty[\tilde{B}_j u, \tilde{B}_j v]$, we have that $\tilde{B}_j \in \Lp{\infty}(U_j;\End(\Tensors[r,s]\cM),\mgt)$ and $\tilde{B}_j^{-1} \in \Lp{\infty}(U_j;\End(\Tensors[r,s]\cM),\mgt)$.

On $K_j$, as $m,n\to \infty$, 
\[
\rVert \tilde{B}_j\xi_n-\tilde{B}_j\xi_m\lVert_{\Lp{\infty}(K_j,\Tensors[r,s]\cM,\mht_j^{\infty})}
=
\rVert \xi_n-\xi_m \lVert_{\Lp{\infty}(K_j;\Tensors[r,s]\cM,\mgt)} \to 0.
\]
That is, $\set{\tilde{B}_j\xi_n}$ is a Cauchy sequence of in $\Lp{\infty}(K_j,\Tensors[r,s]\cM,\mht_j^{\infty})$. 
Since $\mgt_j^\infty$ is smooth, by Lemma~\ref{vbundle-smooth-complete}, there exists $\tilde{\xi}^j \in  \Lp{\infty}(K_j,\Tensors[r,s]\cM,\mht_j^{\infty})$ such that $\tilde{B}_j \xi_n \to \tilde{\xi}^j$.

Let $\xi^j := B_j^{-1} \tilde{\xi}^j$. 
Clearly, $\xi^j \in \Lp{\infty}(K_j;\Tensors[r,s]\cM,\mgt)$ and we see that $\xi_n\rest{K_j} \to \xi^j$. 
We show that $\xi^{j+1}\rest{K_j} = \xi_j$.
Fix $\epsilon > 0$.
Since $\xi_n \to \xi^j$ inside $K_j$, we have that there is an $N$ such that $\norm{\xi^{a} - \xi_n}_{\Lp{\infty}(K_{a}; \Tensors[r,s]\cM, \mgt)} < \epsilon/2 $ for $a = j, j+1$ and $n \geq N$.
Therefore, 
\begin{align*}  
\norm{\xi^{j+1} - \xi^{j}}_{\Lp{\infty}(K_{j}; \Tensors[r,s]\cM, \mgt)}
&\leq 
 \norm{\xi^{j+1} - \xi_N}_{\Lp{\infty}(K_{j}; \Tensors[r,s]\cM, \mgt)} 
+
 \norm{\xi^{j} - \xi_N}_{\Lp{\infty}(K_{j}; \Tensors[r,s]\cM, \mgt)} \\
&\leq
 \norm{\xi^{j+1} - \xi_N}_{\Lp{\infty}(K_{j+1}; \Tensors[r,s]\cM, \mgt)} 
+
 \norm{\xi^{j} - \xi_N}_{\Lp{\infty}(K_{j}; \Tensors[r,s]\cM, \mgt)} \\ 
&\leq 
\frac{\epsilon}{2} +   \frac{\epsilon}{2} < \epsilon
\end{align*}
Since $\epsilon > 0$ was arbitrary, we conclude that $\xi^{j+1}\rest{K_j}  = \xi^j$.
Therefore,
\[
\xi(x) := \xi^j(x) \quad x \in U^j.
\]
is well-defined.

Now, we show that $\xi \in {\Lp{\infty}(\cM; \Tensors[r,s]\cM, \mgt)}$.
By the Cauchy criterion for $\set{\xi_n}$, we can find  $N = N(1) > 0$ such that for all $m,n \geq N$, we have $\modulus{\xi_n(x) - \xi_m(x)}_{\mgt(x)} \leq 1$ for almost-every $x \in \cM$.
Noting that  $\xi_n(x) \to \xi(x)$ pointwise $x$-almost-everywhere since $\xi_n \to \xi^j$ in $\Lp{\infty}(U_j; \Tensors[r,s]\cM,\mgt)$, let  $x \in \cM$ a regular point (where $\xi_n(x) \to \xi(x)$ pointwise) and let $K_j \subset \cM$ such that $x \in K_j$. 
Then, 
\begin{align*} 
\modulus{\xi_n(x)}_{\mgt(x)} 
&\leq 
\modulus{\xi_n(x) - \xi_N(x)} + \modulus{\xi_N(x)} \\
&\leq 
\modulus{\xi_n(x) - \xi_N(x)} +\norm{\xi_N}_{\Lp{\infty}(\cM; \Tensors[r,s]\cM, \mgt)}  \\
&\leq 
1 + \norm{\xi_N}_{\Lp{\infty}(K_{j+1}; \Tensors[r,s]\cM, \mgt)}
\end{align*}
for all $n \geq N$.
Therefore, we can take the limit a $n \to \infty$ to conclude that $\modulus{\xi(x)}_{\mgt(x)} \leq 1 + \norm{\xi_N}_{\Lp{\infty}(K_{j+1}; \Tensors[r,s]\cM, \mgt)}$ at every regular point. 
Hence, $\xi \in {\Lp{\infty}(\cM; \Tensors[r,s]\cM, \mgt)}$.

To finish the proof, we show that $\xi_n \to \xi$ in $\Lp{\infty}(\cM; \Tensors[r,s]\cM, \mgt)$ by showing that $\xi_n(x) \to \xi(x)$ uniformly almost-everywhere in $x \in \cM$.
Fix $\epsilon > 0$ and let $N = N(\epsilon)$ from the Cauchy criterion satisfying for all $m, n \geq N$,  $\modulus{\xi_n(x) - \xi_m(x)}_{\mgt(x)} <  \epsilon/2$ for almost-every $x \in \cM$. 
By pointwise convergence, we can find $N'(\epsilon, x)> 0$ such that for $m \geq N'$ implies $\modulus{\xi(x) - \xi_m(x)}_{\mgt(x)} <  \epsilon/2$
Again, fix a regular point $x \in \cM$ where this holds.
Then, for $n \geq N$ and choosing $m := \max\set{N, N'}$, 
\[
\modulus{\xi(x) - \xi_n(x)}_{\mgt(x)} 
\leq 
\modulus{\xi(x) - \xi_m(x)}_{\mgt(x)} + \modulus{\xi_m(x) - \xi_n(x)}_{\mgt(x)}
\leq \frac\epsilon2 + \frac\epsilon2 
< 
\epsilon.  
\]
This is a uniform estimate for almost-every $x \in \cM$ and therefore, $\Lp{\infty}(\cM; \Tensors[r,s]\cM, \mgt)$ is a Banach space.
\end{proof}

As expected, the $\Lp{\infty}$-spaces induced by rough Riemannian metrics of smooth subbundles of $\Tensors[r,s]\cM$ also are Banach spaces.
\begin{cor} 
Let $\cE \subset \Tensors[r,s]\cM$ be a smooth subbundle of $\cM$.
Then for $\mg \in \Met(\cM)$ and $\mgt$ the induced bundle metric on $\Tensors[r,s]\cM$, the space $\Lp{\infty}(\cM; \cE, \mgt)$ is a Banach space.
\end{cor}
\begin{proof}
Let $\pi_x: \Tensors[r,s]_x\cM \to  \cE_x$ be the fibrewise projection and let $\pi: \Tensors[r,s]\cM \to \cE$ be the total projection.
Clearly, $\norm{\pi u}_{\Lp{\infty}(\cM;\Tensors[r,s]\cM,\mgt)} \leq \norm{u}_{\Lp{\infty}(\cM;\Tensors[r,s]\cM,\mgt)}$ and $\pi$ is an idempotent on $\Lp{\infty}(\cM;\Tensors[r,s]\cM,\mgt)$ as $\pi_x$ is a fibrewise idempotent.  
This defines a bounded projection in $\Lp{\infty}(\cM;\Tensors[r,s]\cM,\mgt)$ with $\Lp{\infty}(\cM;\cE,\mgt) = \pi \Lp{\infty}(\cM;\Tensors[r,s]\cM,\mgt)$.
Since bounded projections in Banach spaces have closed ranges, we conclude that $\Lp{\infty}(\cM;\cE,\mgt)$ is a closed subspace of a Banach space and hence, complete.
\end{proof} 

Via this, we prove the following. 
\begin{prop}
\label{Prop:Complete} 
The space $(\Met(\cM),\rmet^\cM)$ is a complete metric space. 
\end{prop}
\begin{proof}
Let $\{\mg_n\}$ be a Cauchy sequence in $\Met(\cM)$. 
That is, for any $\epsilon>0$, there exists $N = N(\epsilon)\in \mathbb{N}$ such that for all $m, n \geq N$, we have $\rmet^\cM(\mg_n,\mg_m)<\epsilon$. 
Rearranging this sequence by throwing away a finite number of initial terms, without loss of generality, we can assume that for all $m, n \geq 1$, 
\[
\rmet^\cM(\mg_n,\mg_m)<1.
\]
Therefore $\mg_n \in \Comp(\mg_1)$ and by Theorem~\ref{Thm:Stab} and Proposition~10~\cite{BRough}, there exists  
a $\mg_1$-self-adjoint endomorphism $B_m^2 \in \Lp{\infty}(\cM;\End(\tanb\cM),\mg_1)$  such that $B_m^{-2} \in \Lp{\infty}(\cM;\End(\tanb\cM),\mg_1)$  and 
\[
\mg_m(x)(u,u)=\mg_1(x)(B_m(x)u, B_m(x)u).
\]
Since $\rmet^\cM(\mg_m,\mg_1) < 1$, $\e^{-1} \rvert u \lvert_{\mg_1(x)}\leq \rvert u \lvert_{\mg_m(x)} \leq e \rvert u \lvert_{\mg_1(x)}$ and therefore, 
\[
\e^{-1} \rvert u \lvert_{\mg_1(x)}\leq \rvert B_m(x)u \lvert_{\mg_1(x)} \leq e \rvert u \lvert_{\mg_1(x)}.
\]
This provides us with the following two-sided uniform bound 
\begin{equation*} 
\label{Eq:2sidedbd} 
\e^{-1} \leq \norm{B_m(x)}_{\op,\mg_1(x)}\leq \e.
\end{equation*} 
on the matrix norm with respect to the metric $\mg_1$.

For $0< \epsilon <1,$ there exists $N = N(\epsilon)\in \mathbb{N}$ such that $\rmet^\cM(\mg_m,\mg_n)<\epsilon,$ for all $n,m \geq N.$ Thus, we have the following local comparability inequality
\[\e^{-\epsilon} \rvert u\lvert_{\mg_m(x)}\leq \rvert u \lvert_{\mg_n(x)}\leq \e^{\epsilon} \rvert u\lvert_{\mg_m(x)}, \] 
and therefore  
\[
\e^{-\epsilon}
\leq 
\lVert B_n(x)B_m^{-1}(x) \rVert_{\op,\mg_1(x)}
\leq \e^{\epsilon}.
\]
By employing the same argument that has been in Section~\ref{S: SmoothRRM} to prove Theorem~\ref{Thm.limit.smooth.RRMs}, we compute 
\begin{equation*}\label{est.norm.Bn-Bm}
\lVert B_n^2(x)- B_m^2(x) \rVert_{\op, \mg_1(x)}\leq \e^2 (\e^{2\epsilon} - 1).
\end{equation*}
Identify $\End(\tanb\cM) = \cotanb\cM \tensor \tanb\cM = \Tensors[1,1]\cM$. 
The induced norm $\mgt_1$ on $\cotanb\cM \tensor \tanb\cM$ is readily seen to be the Frobenius norm by performing a calculation in an orthonormal frame at a regular point $x \in \cM$. 
Therefore, as in the proof of Theorem~\ref{Thm.limit.smooth.RRMs}, we obtain that
\begin{equation*}
\label{Eq:RealNormEst}
\modulus{B_n^2(x)- B_m^2(x)}_{\mgt_1(x)} \leq (\dim \cM) \e^2 (\e^{2\epsilon} - 1).
\end{equation*}
This shows that $\set{B_n^2}$ is a Cauchy sequence in  $\Lp{\infty}(\cM;\Tensors[1,1]\cM,\mgt_1)$.
Invoking Lemma~\ref{vbundle-rough-complete-Cau-loc}, we obtain some $B' \in \Lp{\infty}(\cM;\Tensors[1,1]\cM,\mg_1)$ such that $B_n^2 \to B'$ in $\Lp{\infty}(\cM;\Tensors[1,1]\cM,\mgt_1)$.
Also, since $(B_n^2)^{\ast,\mgt} = B_n^2$ almost-everywhere pointwise, we obtain that $(B')^{\ast,\mg_1} = B'$ almost-everywhere pointwise.

On replacing $B_m$ by $B_m^{-1}$ in \eqref{Eq:2sidedbd} and replicating this argument, we find that $\set{B_n^{-2}}$ is a Cauchy sequence in $\Lp{\infty}(\cM;\Tensors[1,1]\cM,\mgt_1)$ and therefore, $(B')^{-1} \in \Lp{\infty}(\cM;\Tensors[1,1]\cM,\mgt_1)$.
Therefore, we let $B := \sqrt{B'}$ pointwise-almost everywhere and we have shown that $B_n^2 \to B^2$ and $B_n^{-2} \to B^{-2}$ in $\Lp{\infty}(\cM;\Tensors[1,1]\cM,\mgt_1)$.
 
Define 
\[
\mg_\infty(x)[u,u]
:=
\mg_1(x)[B(x)u,B(x)u]
=
\mg_1(x)[B^2(x)u,u],
\]
which is readily verified to be a rough Riemannian metric. 
We claim that the sequence $\rmet^\cM(\mg_n,\mg_\infty) \to 0$ as $n \to \infty$. 

First, note at a regular point $x \in \cM$, and all $u \in \tanb_x\cM$, 
\begin{align*}
\modulus{u}_{\mg_n(x)}^2 
&= \mg_1(x)[B_n^2(x)u, u] \\
&= \mg_1(x)[(B_n^2(x) - B^2(x))u, u] +\mg_1(x)[B^2(x)u, u] \\
&\leq  \norm{B_n^2(x) - B^2(x)}_{\op,\mg_1(x)} \modulus{u}_{\mg_1(x)}^2 + \modulus{u}_{\mg_\infty(x)}^2 \\ 
&\leq  \norm{B_n^2(x) - B^2(x)}_{\op,\mg_1(x)} \modulus{B^{-1}(x) B(x)u}_{\mg_1(x)}^2 + \modulus{u}_{\mg_\infty(x)}^2 \\ 
&\leq \cbrac{\e^2(\e^{2\epsilon} - 1) \norm{B^{-1}(x)}_{\op,\mg_1(x)}^2 + 1}\ \modulus{u}_{\mg_\infty(x)}^2 \\ 
&\leq \cbrac{(\dim \cM)\e^2(\e^{2\epsilon} - 1) \norm{B^{-1}}^2_{\Lp{\infty}(\cM;\Tensors[1,1]\cM,\mgt_1)} + 1} \modulus{u}_{\mg_\infty(x)}^2,
\end{align*}
where we obtain a dimensional constant from changing from the operator norm on $B^{-1}$ to the Frobenius norm before taking a supremum. 
Similarly,
\begin{align*}
\modulus{u}_{\mg_\infty(x)}^2 
&= \mg_1(x)[(B^2(x) - B_n^2(x))u, u] +\mg_1(x)[B_n^2(x)u, u] \\
&\leq \norm{B_n^2(x) - B^2(x)}_{\op,\mg_1(x)} \modulus{B^{-1}(x) B(x) u}_{\mg_1(x)}^2 +\modulus{u}_{\mg_n(x)}^2  \\ 
&\leq \cbrac{(\dim \cM)\e^2(\e^{2\epsilon} - 1) \norm{B^{-1}}^2_{\Lp{\infty}(\cM;\Tensors[1,1]\cM,\mgt_1)}} \modulus{u}_{\mg_\infty(x)}^2 + \modulus{u}_{\mg_n(x)}^2
\end{align*}
Subtracting the first term from both sides and factoring out $\modulus{u}_{\mg_\infty(x)}^2$, and taking square roots, we obtain 
\begin{align*}
&\cbrac{1 - (\dim \cM)\e^2(\e^{2\epsilon} - 1) \norm{B^{-1}}^2_{\Lp{\infty}(\cM;\Tensors[1,1]\cM,\mgt_1)}}^{\frac12} \modulus{u}_{\mg_\infty(x)}
\leq 
\modulus{u}_{\mg_n(x)} \\ 
&\qquad\qquad\leq 
\cbrac{(\dim \cM)\e^2(\e^{2\epsilon} - 1) \norm{B^{-1}}^2_{\Lp{\infty}(\cM;\Tensors[1,1]\cM,\mgt_1)} + 1}^{\frac12} \modulus{u}_{\mg_\infty(x)}.
\end{align*}
%HERE%
Since the term $(\dim \cM)\e^2(\e^{2\epsilon} - 1) \norm{B^{-1}}^2_{\Lp{\infty}(\cM;\Tensors[1,1]\cM,\mgt_1)} \to 0$ as $\epsilon \to 0$ continuously, we can find a $\delta > 0$ to define functions $f_1,f_2:(0,\delta) \to (0,\infty)$ by: 
\begin{align*} 
f_1(\epsilon) &:= \cbrac{1 - (\dim \cM)\e^2(\e^{2\epsilon} - 1) \norm{B^{-1}}^2_{\Lp{\infty}(\cM;\Tensors[1,1]\cM,\mgt_1)}}^{\frac12} \\
f_2(\epsilon) &:= \cbrac{(\dim \cM)\e^2(\e^{2\epsilon} - 1) \norm{B^{-1}}^2_{\Lp{\infty}(\cM;\Tensors[1,1]\cM,\mgt_1)} + 1}^{\frac12}. 
\end{align*}
Clearly $f_1$ and $f_2$ satisfy the hypothesis of Lemma~\ref{Lem:AlternativeClose} (rewriting the sequence to fit into the form of this lemma), we can conclude that $\rmet^\cM(\mg_{n},\mg_{\infty}) \to 0$ as $n \to \infty$.
Therefore, $(\Met(\cM),\rmet^\cM)$ is a complete metric space. 
\end{proof}

\section{Connectedness}
\label{S:Connectedness}

Recall the subset $\Comp(\mg)$ for each $\mg \in \Met(\cM)$.
In this section, we provide a proof of Theorem~\ref{Thm:CompConnected} and Theorem~\ref{Thm_com_iff_conn}.
The first states, as the notation suggests, that $\Comp(\mg)$ is a connected component.
In fact, we show a stronger property that $\Comp(\mg)$ is path-connected.
The latter theorem  characterises compactness of $\cM$ in terms of connectedness of $(\Met(\cM),\rmet^\cM)$.

\subsection{Path-connectedness}
Clearly, we can see that $\Comp(\mg)$ arises as equivalence classes with respect to the relation $\mg_1 \sim \mg_2$ if $\rmet^\cM(\mg_1,\mg_2) < \infty$.
We begin our discussion around connectedness by first illustrating that each $\Comp(\mg)$ is both open and closed simultaneously.
\begin{lem}
\label{lem_closed_open}
For  $\mg \in \Met(\cM)$, the subset $\Comp(\mg)$ is open and closed. 
\end{lem}
\begin{proof}
Fix $\mg \in \Met(\cM)$. 
Firstly, we show that $\Comp(\mg)$ is an open set of $\Met(M)$ by showing that for every $\mh \in \Comp(\mg)$ and every $R > 0$, we have $\Ball(R,\mh) \subset \Comp(\mg)$. 
So fix $R > 0$ and let $\mh' \in \Ball(R, \mh)$ which is if and only if $\rmet^\cM(\mh, \mh') < R$.
Now,
\[ 
\rmet^\cM(\mg,\mh') \leq \rmet^\cM(\mg, \mh) + \rmet^\cM(\mh,\mh') < \rmet^\cM(\mg, \mh)  + R < \infty
\]
since $\rmet^\cM(\mg,\mh) < \infty$.
That is, $\mh' \in \Comp(\mg)$ and therefore $\Comp(\mg) = \union_{R > 0} \Ball(R,\mh)$. 

To show that $\Comp(\mg)$ is closed, let  $\{\mg_n\} \subset \Comp(\mg)$ be a sequence of rough Riemannian metrics with $\mg \in \Met(\cM)$ with $\rmet^\cM(\mg_n,\mg_\infty) \to 0$. 
We prove that $\mg_{\infty} \in \Comp(\mg)$ by showing that $\rmet^\cM(\mg,\mg_\infty) < \infty$ 
By the convergence, for say $\epsilon = 1$, we obtain a  $N > 0$ such that $\met(\mg_N, \mg_\infty) <  1 < \infty$. 
Since $\mg_n \in \Comp(\mg)$, we have that
\[ 
\rmet^\cM(\mg, \mg_\infty) <  \rmet^\cM(\mg_N,\mg) + \rmet^\cM(\mg_N,\mg_\infty) < \infty.
\]
Therefore $\Comp(\mg)$ is closed. 
\end{proof}

Although standard, we state and prove the following lemma as it is crucially used in the proof of Theorem~\ref{Thm:CompConnected}.
\begin{lem}
\label{fun.cal.lem}
Let $\mathcal{H}$ be a Hilbert space and $T:\mathcal{H}\to \mathcal{H}$ a bounded self-adjoint operator. 
Suppose  there exists $0 < a < b < \infty$ such that $\spec(T) \subset [a, b]$
Given an $\alpha \in [0, \infty)$, the fractional power $T^\alpha$ is again a bounded, positive self-adjoint operator and 
\[ 
a^\alpha \leq \norm{T^\alpha} \leq b^\alpha.
\]
\end{lem}
\begin{proof}
For $\beta \in \R$, $f_\beta: [a,b] \to \R$ given by $f_\beta(x) = x^\beta$ is continuous on $[a,b]$ (and therefore Borel) so $f_{\alpha}(T) = T^\alpha: \Hil \to \Hil$ is again a positive bounded self-adjoint operator.
Moreover, $\norm{T^\alpha}$ is equal to the spectral radius of $T^\alpha$ as $T^\alpha$ is self-adjoint. 
By the spectral mapping theorem using the continuity of $f_\alpha$, we have that $\spec(f_\alpha(T)) = f_{\alpha}(\spec(T))$.
Hence, we obtain $\norm{T^\alpha} \leq b^\alpha$.

Now consider $T^{-\alpha} = f_{-\alpha}(T)$.
Again, the spectral mapping theorem provides $\spec(T^{-\alpha}) \subset [b^{-\alpha}, a^{-\alpha}]$ and therefore, $\norm{T^{-\alpha}} \leq a^{-\alpha}$. 
This is equivalent to  $a^{\alpha} \geq \norm{T^\alpha}$ which together with the earlier estimate finishes the proof. 
\end{proof}

\begin{proof}[Proof of Theorem \ref{Thm:CompConnected}]
To prove $\Comp(\mg)$ is path connected, given $\mg_0, \mg_1 \in \Comp(\mg)$, we show there exists $f\in \Ck{0}([0,1], \Comp(\mg))$ such that $f(0) = \mg_0$ and $f(1) = \mg_1$.
From Theorem~\ref{Thm:Stab}, let $B \in \Ell(\cM)$ such that $\mg_1[u,v] = \mg_0[Bu,Bv]$.
Without loss of generality, we assume that $B$ is $\mg_0$-self-adjoint (as we can consider $\sqrt{B^{\ast,\mg}B}$ in place of $B$ to yield the same transformation).
For $t \in [0,1]$, define 
\[
\mg_t(x)(u,u)=\mg(x)[B^t(x)u,B^t(x)u] 
\]
for $u \in \tanb_x\cM$.

Note now that 
\[
a^{-1} \leq \norm{B(x)}_{\op,\mg_0(x)} \leq a
\]
where $a = \sup_{x\in \cM} \modulus{B(x)}_{\op,\mg_0(x)} < \infty$.
On application of Lemma~\ref{fun.cal.lem}, we deduce that $a^{-t} \leq B^t(x) \leq  a^t$ almost-everywhere $x \in \cM$.
In particular, this shows that $\mg_t \in \Comp(\mg)$.
j\ \text{times}
 
Define $f:[0,1]\to \Comp(\mg)$ by $f(t):=\mg_t$ and show that $f$ is continuous on $[0,1]$.
That is, for a fixed $t\in[0,1]$ and for any given $\epsilon>0$, there exits $\delta>0$ such that, whenever $|t-s|<\delta$, for all $s\in [0,1]$, then $\rmet^\cM(\mg_t,\mg_s)<\epsilon$.
This is equivalent to saying that 
\[
\e^{-\epsilon}\leq \norm{B^s(x) B^{-t}(x)}_{\op, \mg_0(x)} \leq \e^{\epsilon}.
\]

Fix $\epsilon > 0$. 
We consider the case that $s \in \R$ such that  $s - t > 0$.  
By functional calculus, $B^{s}(x)B^{-t}(x) = B^{s-t}(x)$ and so choosing  $\alpha= s-t > 0$ and setting $\delta= \frac{\epsilon}{\log a}$, we find that
\[
\e^{-\epsilon}\leq a^{-\alpha}\leq \norm{B^{s-t}(x)}_{\op,\mg_0(x)} \leq a^{\alpha}\leq \e^{\epsilon}. 
\]
Now, for $t - s > 0$, we consider instead $B^{t-s}(x)$ but the estimate remains the same with the same value of $\delta$. 
Trivially for $t = s$, we have that $B^0 = I$ which corresponds to $\epsilon = 0$.
Together, we see that $f \in \Ck{0}([0,1], \Comp(\mg))$ and hence, the subspace $\Comp(\mg)$ is path-connected. 
\end{proof}

\subsection{The length space property}

Using a similar construction as in the proof of Theorem~\ref{Thm:CompConnected} via Lemma~\ref{fun.cal.lem}, we show that between any two metrics, there exists a midpoint.
\begin{prop}
Let $\mg_0, \mg_1 \in \Comp(\mg)$.
Then, there exists $\mh \in \Comp(\mg)$ such that 
\[
\frac{1}{2}\rmet^{\cM}(\mg_0,\mg_1)=\rmet^{\cM}(\mg_1,\mh)=\rmet^{\cM}(\mg_0,\mh).
\]
\end{prop}
\begin{proof}
Fix $\mg$ and let $\mg_0,\mg_1 \in \Comp(\mg)$ as in the hypothesis. 
By Theorem~\ref{Thm:Stab} and Proposition~10~\cite{BRough}, we obtain $B^2 \in \Ell(\cM)$ self-joint with respect to $\mg_0$  such that 
\[
\mg_1(x)[u,v]
=\mg_0(x)[B(x)u, B(x)v].
=\mg_0(x)[B^2(x)u,v]
\]
Fix the continuous path $[0,1] \ni t \mapsto \mg_t := \mg_0(x)[B^t(x)u, B^t(x)u ]$ as in the proof of Theorem~\ref{Thm:CompConnected}.
Clearly, by construction, the end points $t = 0,1$ agree with our given metrics $\mg_0$ and $\mg_1$ respectively.

Let $\mh := \mg_{\frac{1}{2}}$. 
We claim that this is the midpoint.
First, we show that $\rmet^\cM(\mh, \mg_i) \leq \frac{1}{2}\rmet^{\cM}(\mg_0,\mg_1)$ for $i = 0, 1$.
From $\mg_0, \mg_1 \in \Comp(\mg)$, we have that for almost-every $x \in \cM$, 
\[
\e^{-\rmet^{\cM}(\mg_0,\mg_1)}|u|_{\mg_0(x)}\leq |u|_{\mg_1(x)}\leq \e^{-\rmet^{\cM}(\mg_0,\mg_1)}|u|_{\mg_0(x)}.
\]
This yields
\[
\e^{-\rmet^{\cM}(\mg_0,\mg_1)}\leq \norm{B(x)}_{\op,\mg_0(x)} \leq  \e^{ \rmet^{\cM}(\mg_0,\mg_1)}.
\]
By applying Lemma \ref{fun.cal.lem} with a choice of $b = a^{-1}$ and with $a = \e^{\rmet^{\cM}(\mg_0,\mg_1)}$, we get  
\[
\e^{-\frac{1}{2}\rmet^{\cM}(\mg_0,\mg_1)}\leq \norm{B^{\frac{1}{2}}(x)}_{\op,\mg_0} \leq \e^{\frac{1}{2}\rmet^{\cM}(\mg_0,\mg_1)}.
\]
Therefore, 
\[
\e^{-\frac{1}{2}\rmet^{\cM}(\mg_0,\mg_1)}\modulus{u}_{\mg_0(x)} \leq \norm{u}_{\mh(x)} \leq \norm{u}_{\mh(x)}  \e^{\frac{1}{2}\rmet^{\cM}(\mg_0,\mg_1)}.
\]
Therefore, by definition, 
\[
\rmet^{\cM}(\mg_0, \mh) \leq \log \cbrac{\e^{\frac{1}{2}\rmet^{\cM}(\mg_0,\mg_1)}} = \frac{1}{2}\rmet^{\cM}(\mg_0,\mg_1).
\]

Since $\mg_1[u,v] = \mg_0[Bu,Bv]$ we have that $\mg_0[u,v] = \mg_{1}[B^{-1}u, B^{-1}v]$. 
Consider the curve $t \mapsto \mg_{t}'$ defined by $\mg_{t}'[u,v] = \mg_{1}[B^{-t}u, B^{-t}v]$.
Clearly, $\mg_{\frac12}' = \mh$. 
Applying the same argument to $\mg_{\frac12}'$, as we did to $\mg_{\frac12}'$, we obtain $\rmet^{\cM}(\mg_1, \mh) \leq \frac{1}{2}\rmet^{\cM}(\mg_0,\mg_1)$.

Combining these estimates using the triangle inequality, we have  
\[
\rmet^{\cM}(\mg_0,\mg_1) \leq \rmet^{\cM}(\mg_0,\mh) + \rmet^{\cM}(\mh,\mg_1) \leq\frac{1}{2}\rmet^{\cM}(\mg_0,\mg_1) + \frac{1}{2}\rmet^{\cM}(\mg_0,\mg_1)
\]
and so 
\[
\rmet^{\cM}(\mg_0,\mh)+\rmet^{\cM}(\mg_1,\mh)=\rmet^{\cM}(\mg_0,\mg_1).
\]
Adding $-\rmet^{\cM}(\mg_0, \mh) \geq -\frac{1}{2}\rmet^{\cM}(\mg_0,\mg_1)$ successively then produces the lower bound
\[
\rmet^{\cM}(\mg_1, \mh) \geq \frac{1}{2}\rmet^{\cM}(\mg_0,\mg_1).
\]
Therefore, $\rmet^{\cM}(\mg_1, \mh) = \frac{1}{2}\rmet^{\cM}(\mg_0,\mg_1)$.
Similarly, we find that $\rmet^{\cM}(\mg_0, \mh) = \frac{1}{2}\rmet^{\cM}(\mg_0,\mg_1)$
\end{proof}

Since for complete metrics, the midpoint property implies that the metric is intrinsic, we obtain the following important consequence.
\begin{cor}
\label{Cor:Length}
The space $(\Met(\cM),\rmet^{\cM})$ is a length space.
\qedhere
\end{cor}

\subsection{Characterisation of compactness}

Lastly, we provide a proof of Theorem~\ref{Thm_com_iff_conn}, which characterises the compactness of $\cM$ through $\Met(\cM)$.

\begin{proof}[Proof of Theorem~\ref{Thm_com_iff_conn}]
We first prove that if $\cM$ is compact, then $(\Met(\cM),\rmet^{\cM})$ is connected. 
Fix $\mg \in \Met(\cM)$. 
By Lemma~\ref{Lem:RoughMetEquiv}, this is equivalent to $\mg, \mg^{-1} \in \Lp[loc]{\infty}(\cM;\Tensors[2,0]\cM)$. 
However, on a compact manifold $\Lp[loc]{\infty}(\cM;\Tensors[2,0]\cM) = \Lp{\infty}(\cM;\Tensors[2,0]\cM,\mh)$ with respect to any $\mh \in \Met_{\Ck{\infty}}(\cM)$.
Therefore, $\rmet^{\cM}(\mg,\mh) < \infty$ for each $\mh \in \Met_{\Ck{\infty}}(\cM)$.
By the triangle inequality, we see that $\rmet(\mg,\mg') < \infty$ for any $\mg, \mg' \in \Met(\cM)$ and hence $(\Met(\cM),\rmet^{\cM})$ is connected.

We prove the converse by proving the contrapositive of the converse. 
That is, we prove that if $\cM$ is not compact, then $(\Met(\cM),\rmet^{\cM})$ is disconnected.
Suppose that $\cM$ is not compact.
This means there is a cover $\{U_{\alpha}\}$ for $\cM$ which does not possess a finite subcover. 
However, since $\cM$ is Lindelöf, there is a countable subcover $\set{U_j}_{j=1}^\infty$.

Via this, we create a countable disjoint Borel cover for $\cM$ as follows.
Let 
\[
V_1 := U_1 \quad\text{and}\quad  V_j := U_j\setminus (\cup_{i=1}^{j-1}V_i).
\] 
Define $f: \cM \to [0, \infty)$ by $f(x)=2^j$ for $x\in V_j$. 
Since the sets $V_i$ are Borel, this is clearly a Borel function. 

Now, fix any $\mg \in \Met(\cM)$ and fine $\mg_f := f\mg$. 
We claim $\mg_f$ is a rough Riemannian metric. 
Fix $x\in M$, there exists a chart $U_i$ such that $x\in U_i$ for which the local comparability condition hold for the metric $\mg$ with constant $C_i \geq 1 $ such that 
\[C_i^{-1}|u|_{\psi^*\delta(y)}\leq |u|_{g(y)}\leq C_i|u|_{\psi*\delta(y)},
\]
for almost every $y\in U_i$. 

Now, note that 
\[ 
|u|_{\mg_f(y)}=\sqrt{f(y)}|u|_{\mg(y)} \geq \modulus{u}_{\mg(y)}
\]
since $f \geq 1$.
Moreover, for each $y\in U_i$, there exits $j$ such that either $y \in V_j$ or $y\in \cup_{i=1}^{j-1}V_i$. 
In either case,  $f(y) \leq 2^{j}$.
Therefore, 
\[
2^{-\frac j2} C_u^{-1}|u|_{\psi^*\delta(y)}\leq |u|_{\mg_f(y)} \leq 2^{\frac j2} C_u|u|_{\psi*\delta(y)}.
\]
This shows that $\mg_f \in \Met(\cM)$.

It remains to show that $\rmet^\cM(\mg_f,\mg)=\infty$. 
This can be argued by contradiction. 
Suppose there exists $\rmet^{\cM}(\mg_f,\mg) < \infty$. 
Then,
\[
\e^{-\rmet^{\cM}(\mg_f,\mg)}\modulus{u}_{\mg} \leq \modulus{u}_{\mg_f} \leq \e^{\rmet^{\cM}(\mg_f,\mg)} \modulus{u}_{\mg}. 
\]
By choosing $j > \rmet^{\cM}(\mg_f,\mg)/\log(2)$, on $V_j$ we have 
\[ 
\modulus{u}_{\mg_f(y)}  = 2^j \modulus{u}_{\mg(y)}  \geq \e^{\rmet^{\cM}(\mg_f,\mg)} \modulus{u}_{\mg}.
\]
This is a contradiction so $\rmet^\cM(\mg_f,\mg)=\infty$.
In particular, $\Comp(\mg) \cap \Comp(\mg_f) = \emptyset$ and therefore, $\Comp(\mg) \subsetneqq \Met(\cM)$.
Since Lemma~\ref{lem_closed_open} establishes that $\Comp(\mg)$ is both open and closed, we conclude that $(\Met(\cM),\rmet^{\cM})$ is disconnected.
\end{proof}

\begin{rem}
Note that the same proof shows that, if $\cM$ is not compact, we have at least countable number of disjoint components. 
For that, define $\Phi: \Met(\cM) \to \Met(\cM)$ by $\Phi(\mg) = \mg_f$.
Denote the $j$ product of composition for $j \in \Na$  by
\[
\Phi_j := \underbrace{(\Phi\circ \dots \circ \Phi)}_{j\ \text{times}}.
\] 
Then, we have $\Comp(\Phi_1(\mg)) \neq \Comp(\Phi_2(\mg)) \neq \Comp(\Phi_3(\mg) \dots \neq \Comp(\Phi_j(\mg) \neq \dots$. 
Through an argument similar to that used in the proof above, we can show that $\Comp(\Phi_i(\mg))  \cap \Comp(\Phi_j(\mg)) = \emptyset$ for $i \neq j$.
Therefore, $\sqcup_{j=1}^\infty \Comp(\Phi_j(\mg)) \subset \Met(\cM)$.
\end{rem}

% Bibliography
\bibliographystyle{plain} 

\setlength{\parskip}{0pt}

\end{document}